\documentclass[11pt,oneside,reqno,russian,british]{amsart}
\usepackage[T2A,T1]{fontenc}
\usepackage[koi8-r,latin9]{inputenc}
\usepackage{color}
\usepackage{babel}
\usepackage{verbatim}
\usepackage{units}
\usepackage{mathtools}
\usepackage{amstext}
\usepackage{amsthm}
\usepackage{amssymb}
\usepackage[unicode=true,pdfusetitle,
 bookmarks=true,bookmarksnumbered=false,bookmarksopen=false,
 breaklinks=true,pdfborder={0 0 0},pdfborderstyle={},backref=false,colorlinks=true]
 {hyperref}

\makeatletter

\DeclareRobustCommand{\cyrtext}{%
  \fontencoding{T2A}\selectfont\def\encodingdefault{T2A}}
\DeclareRobustCommand{\textcyr}[1]{\leavevmode{\cyrtext #1}}

\theoremstyle{plain}
\newtheorem{thm}{\protect\theoremname}
\theoremstyle{plain}
\newtheorem{lem}{\protect\lemmaname}
\theoremstyle{definition}
\newtheorem*{defn*}{\protect\definitionname}

\usepackage{calrsfs}
\usepackage{graphicx}
\usepackage{color}
\usepackage{perpage}
\usepackage{upquote}
\usepackage{bbm}
\usepackage[shortlabels]{enumitem}
\usepackage{stmaryrd} 
\usepackage{extarrows} 
\usepackage{romanbar}
\usepackage{breakurl}

\MakePerPage{footnote}
\gdef\SetFigFontNFSS#1#2#3#4#5{} 
\gdef\SetFigFont#1#2#3#4#5{} 
\def\clap#1{\hbox to 0pt{\hss#1\hss}}

\DeclareMathOperator{\dist}{dist}
\DeclareMathOperator{\supp}{supp}

\DeclareMathOperator{\rk}{rnk}

\definecolor{myblue}{rgb}{0.09,0.32,0.44} 
\hypersetup{pdfborder={0 0 0},pdfborderstyle={},colorlinks=true,linkcolor=myblue,citecolor=myblue,urlcolor=blue}

\theoremstyle{remark}
\newtheorem*{qst*}{Question}
\newtheorem*{rmrks*}{Remarks}
\newtheorem*{thm*}{Theorem}

\newlength{\tempindent} 
\newcommand{\lazyenum}{
\setlength{\tempindent}{\parindent} 
\begin{enumerate}[leftmargin=0cm,itemindent=0.7cm,labelwidth=\itemindent,labelsep=0cm,align=left,label=\arabic*)]
\setlength{\parskip}{\smallskipamount}
\setlength{\parindent}{\tempindent}
}

\makeatletter
\def\moverlay{\mathpalette\mov@rlay}
\def\mov@rlay#1#2{\leavevmode\vtop{%
   \baselineskip\z@skip \lineskiplimit-\maxdimen
   \ialign{\hfil$\m@th#1##$\hfil\cr#2\crcr}}}
\newcommand{\charfusion}[3][\mathord]{
    #1{\ifx#1\mathop\vphantom{#2}\fi
        \mathpalette\mov@rlay{#2\cr#3}
      }
    \ifx#1\mathop\expandafter\displaylimits\fi}
\makeatother

\makeatletter
\ifdefined\andify
\renewcommand{\andify}{%
  \nxandlist{\unskip, }{\unskip{} \@@and~}{\unskip{} \@@and~}}
\def\author@andify{%
  \nxandlist {\unskip ,\penalty-1 \space\ignorespaces}%
    {\unskip {} \@@and~}%
    {\unskip \penalty-2 \space \@@and~}%
}
\let\@wraptoccontribs\wraptoccontribs
\@namedef{subjclassname@2020}{%
  \textup{2020} Mathematics Subject Classification}
\fi
\makeatother

\ifpdf
\def\afs#1#2{\href{#1}{\nolinkurl{#2}}}
\else
\def\afs#1#2{\burlalt{#1}{#2}}
\fi

\def\affs#1#2{\href{#1}{\nolinkurl{#2}}}

\numberwithin{lem}{section}

\makeatother

\addto\captionsbritish{\renewcommand{\definitionname}{Definition}}
\addto\captionsbritish{\renewcommand{\lemmaname}{Lemma}}
\addto\captionsbritish{\renewcommand{\theoremname}{Theorem}}
\addto\captionsrussian{\renewcommand{\definitionname}{\inputencoding{koi8-r}ïÐÒÅÄÅÌÅÎÉÅ}}
\addto\captionsrussian{\renewcommand{\lemmaname}{\inputencoding{koi8-r}ìÅÍÍÁ}}
\addto\captionsrussian{\renewcommand{\theoremname}{\inputencoding{koi8-r}ôÅÏÒÅÍÁ}}
\providecommand{\definitionname}{Definition}
\providecommand{\lemmaname}{Lemma}
\providecommand{\theoremname}{Theorem}

\begin{document}
\title[Luzin's problem]{Luzin's problem on Fourier convergence and homeomorphisms}
\author{Gady Kozma and Alexander Olevski\u i}
\begin{abstract}
We show that for every continuous function $f$ there exists an absolutely
continuous homeomorphism of the circle $\varphi$ such that the Fourier
series of $f\circ\varphi$ converges uniformly. This resolves a problem
set by N. N. Luzin.
\end{abstract}

\address{GK: Department of mathematics, the Weizmann Institute of Science,
Rehovot, Israel.}
\email{gady.kozma@weizmann.ac.il}
\address{AO: Department of mathematics, Tel Aviv University, Tel Aviv, Israel.}
\email{olevskii@post.tau.ac.il}
\thanks{GK is supported by the Israel Science Foundation and by the Jesselson
Foundation.}
\maketitle

\section{Introduction}

Is it possible to improve the convergence properties of the Fourier
series by an appropriate change of variable (homeomorphism of the
circle)? The history of the problem goes back to a result of Julius
P\'al \cite{P14} (a Hungarian mathematician, who worked for many
years in Copenhagen), improved by Harald Bohr \cite{B35}.
\begin{thm}
[P\'al-Bohr]Given a real function $f\in C(\mathbb{T})$ one can
find a homeomorphism $\phi:\mathbb{T\to}\mathbb{T}$ such that the
Fourier series of the superposition $f\circ\phi$ converges uniformly
on $\mathbb{T}$.
\end{thm}
For the convenience of the reader we include a proof of the P\'al-Bohr
theorem in \S\ref{sec:The-Pal-Bohr-theorem}, essentially following
the original idea, a surprising use of the Riemann conformal mapping
theorem.

The P\'al-Bohr theorem inspired a number of problems which have been
studied actively. In particular, real proofs providing a visible construction
of $\phi$ were given by Saakyan \cite{S79} and by Kahane and Katznelson
\cite{KK83}. Kahane and Katznelson also proved the complex version
of the theorem, and further showed that for any compact family of
functions $f$ there exists a single homeomorphism $\phi$ carrying
the entire family into the space of functions with uniformly converging
Fourier series. On the other hand, it is possible to construct a real
function $f\in C(\mathbb{T})$ such that no homeomorphism $\phi$
can bring it to the Wiener algebra (answering a problem posed by N.
Luzin), see \cite{O81}.

In all known proofs of the P\'al-Bohr theorem the homeomorphism $\phi$
in general is singular. This prompted Luzin to ask: is it possible,
for an arbitrary $f\in C(\mathbb{T})$, to find an absolutely continuous
homeomorphism $\phi$ such that $f\circ\phi$ has a uniformly converging
of Fourier series?\footnote{Our attempts to discover the history of the problem were only partially
successful. It is mentioned in Bary's book \cite[chapter 4, \S 12]{B64},
attributed to Luzin, who passed away in 1950, before Bary's book was
published (as far as we know, Luzin did not conjecture one direction
or the other for the expected answer). Another piece of information
is from the MathSciNet entry for \cite{O81}, which states that Luzin's
other problem is from the twenties. It is reasonable to assume that
both problems were asked together, but we have no evidence of that.
We thank Carruth McGehee for help with our historical research.} This problem remained open so far. In this paper we resolve this
conjecture.
\begin{thm}
\label{thm:main}Given a real function $f\in C(\mathbb{T})$ one can
find an absolutely continuous homeomorphism $\phi:\mathbb{T\to}\mathbb{T}$
such that the Fourier series of the superposition $f\circ\phi$ converges
uniformly on $\mathbb{T}$.

Further, for any $p<\infty$ it is possible to have in addition that
$\phi'\in L^{p}$.
\end{thm}
We remark that Kahane and Katznelson showed that this cannot be done
if we require also that $\phi\in C^{1}$ \cite[example 5]{KK83}.
Thus an interesting gap remains between our results and those of \cite{KK83}:
is is possible to do the same with a Lipschitz homeomorphism?

We should also mention that certain kinds of improvements in the behaviour
of the Fourier transform are impossible to achieve using absolutely
continuous homeomorphisms. For example, there exists an $f\in C(\mathbb{T})$
such that $\widehat{f\circ\phi}\not\in l_{p}$ for any $1\le p<2$
and any absolutely continuous homeomorphism $\phi$ \cite{O85}. This
result also implies that it would be difficult to answer Luzin's conjecture
using a version of the original construction of P\'al and Bohr (e.g.\ replacing
conformal maps with quasi-conformal ones), as that construction gives
an $H^{1/2}$ function (see below, \S\ref{sec:The-Pal-Bohr-theorem}),
and any $g\in H^{1/2}$ has $\widehat{g}\in l^{p}$ for all $p>1$. 

The paper is organised as follows. In \S \ref{sec:The-Pal-Bohr-theorem}
we prove the P\'al-Bohr theorem. In \S \ref{sec:A-proof-sketch}
we sketch an easier result: that for any continuous function $f$
there exists a H\"older homeomorphism $\varphi$ such that $f\circ\varphi$
has a uniformly bounded Fourier series. This sketch covers about two
thirds of the ideas needed for the full result and we believe it will
be beneficial to the reader. The remainder of the paper contains the
proof of theorem \ref{thm:main}.

\section{\label{sec:The-Pal-Bohr-theorem}The P\'al-Bohr theorem}

Recall that we wish to show that for every real-valued continuous
function $f$ on $\mathbb{T}$ there exists a homeomorphism $\varphi$
such that $f\circ\varphi$ has a uniformly converging Fourier series.
Clearly we may assume $f>0$, by adding a constant if necessary, so
let us do so. Let $\Omega$ be the domain delimited by $f$ in polar
coordinates, namely
\[
\Omega=\{re^{i\theta}:r<f(\theta)\}.
\]
Let $u$ be a conformal map of the disk $\mathbb{D}$ onto $\Omega$
taking 0 to 0. By Caratheo\-dory's theorem \cite[\S I.3]{GM05},
$u$ extends to a one-to-one map of $\partial\mathbb{D}$ onto $\partial\Omega$
and hence $g(t)\coloneqq|u(e^{it})|$ is a change of variables of
$f$ (in other words, $g=f\circ\varphi$). We claim that $g$ has
a uniformly converging Fourier expansion. To see this write $u(z)=\sum c_{k}z^{k}$
and get
\begin{align*}
\iint_{\mathbb{D}}|u'(z)|^{2}\,dz & =\iint_{\mathbb{D}}\bigg|\sum_{k=1}^{\infty}kc_{k}z^{k-1}\bigg|^{2}\,dz\\
 & =\int_{0}^{1}\int_{0}^{2\pi}\bigg|\sum_{k=1}^{\infty}kc_{k}r^{k-1}e^{i(k-1)\theta)}\bigg|^{2}r\,d\theta\,dr\\
\textrm{by Parseval}\qquad & =2\pi\int_{0}^{1}\sum_{k=1}^{\infty}k^{2}|c_{k}|^{2}r^{2k-1}\,dr=\pi\sum_{k=1}^{\infty}k|c_{k}|^{2}.
\end{align*}
Since $\iint_{\mathbb{D}}|u'(z)|^{2}\,dz$ is the area of $\Omega$
and is finite, we get that $g\in H^{1/2}$, the Sobolev space of functions
whose $\frac{1}{2}$-derivative is in $L^{2}$. The theorem is finished
by the following three observations. First, that $H^{1/2}\cap C$
is a Banach algebra (lemma \ref{lem:Banach algebra} below). Hence
$|g|$ is also in $H^{1/2}$ (lemma \ref{lem:ginA}). A continuous,
$H^{1/2}$ function has a uniformly converging Fourier expansion (lemma
\ref{lem:A implies uniform convergence}).\qed
\begin{lem}
\label{lem:Banach algebra}The space $H^{1/2}\cap C$, that is, the
space of all continuous $f:[0,2\pi]\to\mathbb{C}$ with $\sum|k||\widehat{f}(k)|^{2}<\infty$
equipped with the norm $\max\{\Vert f\Vert_{\infty},\linebreak[4](\sum|k||\widehat{f}(k)|^{2})^{1/2}\}$
is a Banach algebra with respect to pointwise multiplication.
\end{lem}
\begin{proof}
Since $\Vert fg\Vert_{\infty}\le\Vert f\Vert_{\infty}\Vert g\Vert_{\infty}$,
we need to show that $\Vert fg\Vert_{H^{1/2}}\le C\Vert f\Vert_A\cdot\Vert g\Vert_A$
for some absolute constant $C$, where $\Vert f\Vert_{A}\coloneqq\max\{\Vert f\Vert_{\infty},\Vert f\Vert_{H^{1/2}}\}$
and $\Vert f\Vert_{H^{1/2}}^{2}=\sum_{k=-\infty}^{\infty}|k||\widehat{f}(k)|^{2}$.
For $k\ge1$ let $P_{k}$ be the following smoothed spectral projections,
\[
(P_{k}f)(t)=\sum_{j=-\infty}^{\infty}\tau_{k}(j)\widehat{f}(j)e^{ijt}\qquad\tau_{k}(j)=\begin{cases}
|j|2^{1-k}-1 & 2^{k-1}<|j|\le2^{k}\\
2-|j|2^{-k} & 2^{k}<|j|\le2^{k+1}\\
0 & \text{othewise,}
\end{cases}
\]
and let $P_{0}f=\sum_{|j|\le1}\widehat{f}(j)e^{ijt}$ so that $f=\sum_{k=0}^{\infty}P_{k}f.$
We get 
\begin{align*}
\Vert fg\Vert_{H^{1/2}} & =\left\Vert \bigg(\sum_{k=0}^{\infty}P_{k}f\bigg)\bigg(\sum_{k=0}^{\infty}P_{k}g\bigg)\right\Vert _{H^{1/2}}\\
 & \le\sum_{k=0}^{\infty}\bigg\Vert P_{k}f\sum_{l=0}^{k}P_{l}g\bigg\Vert_{H^{1/2}}+\sum_{l=0}^{\infty}\bigg\Vert P_{l}g\sum_{k=0}^{l-1}P_{k}f\bigg\Vert_{H^{1/2}}.
\end{align*}
For each of the terms we can write
\begin{align*}
\bigg\Vert P_{k}f\sum_{l=0}^{k}P_{l}g\bigg\Vert_{H^{1/2}} & \le2^{(k+2)/2}\bigg\Vert P_{k}f\sum_{l=0}^{k}P_{l}g\bigg\Vert_{2}\\
 & \le2^{(k+2)/2}\Vert P_{k}f\Vert_{2}\bigg\Vert\sum_{l=0}^{k}P_{l}g\bigg\Vert_{\infty}\le C2^{k/2}\Vert P_{k}f\Vert_{2}\Vert g\Vert_{\infty}
\end{align*}
where the last inequality follows because $\sum_{l=0}^{k}P_{l}$ is
a de la Vall\'ee Poussin kernel, and in particular is the difference
of two Fej\'er kernel, hence has norm at most 3 on $L^{\infty}$.
An identical calculation holds for the terms $\Vert P_{l}g\sum P_{k}f\Vert$
and hence
\begin{align*}
\Vert fg\Vert_{H^{1/2}} & \le C\Vert g\Vert_{\infty}\sum_{k=0}^{\infty}2^{k/2}\Vert P_{k}f\Vert_{2}+C\Vert f\Vert_{\infty}\sum_{l=0}^{\infty}2^{l/2}\Vert P_{l}g\Vert_{2}\\
 & \le C(\Vert g\Vert_{\infty}\Vert f\Vert_{H^{1/2}}+\Vert f\Vert_{\infty}\Vert g\Vert_{H^{1/2}})\le C\Vert f\Vert_{A}\Vert g\Vert_{A},
\end{align*}
showing the Banach algebra property, as needed.
\end{proof}
\begin{lem}
If $f\in H^{1/2}\cap C$ and if $z$ is a function analytic in a neighbourhood
of $\{f(x):x\in\mathbb{T}\}$ then $z\circ f\in H^{1/2}\cap C$.
\end{lem}
\begin{proof}
We follow Wiener's proof of the same fact for Wiener's algebra. We
first note that being in the algebra $A=H^{1/2}\cap C$ is a local
property, namely if some $f:\mathbb{T}\to\mathbb{C}$ satisfies that
for every $x$ there is an $\epsilon>0$ and a function $a_{x}\in A$
such that $a_{x}|_{(x-\epsilon,x+\epsilon)}=f|_{(x-\epsilon,x+\epsilon)}$
then it follows that $f\in A$. Indeed, we take a finite cover of
$\mathbb{T}$ by intervals $(x_{1}-\epsilon_{1},x_{1}+\epsilon_{1}),\dotsc,(x_{n}-\epsilon_{n},x_{n}+\epsilon_{n})$
and take some corresponding partition of unity $p_{1},\dotsc,p_{n}$
with $p_{i}\in A$ then $f=fp_{1}+\dotsc+fp_{n}=a_{x_{1}}p_{1}+\dotsc+a_{x_{n}}p_{n}\in A$
(recall that $A$ is a Banach algebra, by lemma \ref{lem:Banach algebra}).
This shows locality, and it is therefore enough to show that for every
$x$ there exists some $\epsilon>0$ and $a\in A$ such that $a|_{(x-\epsilon,x+\epsilon)}=z\circ f|_{x}$.
By translation invariance we may assume that $x=0$ and by invariance
to addition of constants we may assume that $f(0)=0$.

We next find some functions $\rho_{n}\in A$ such that $\supp\rho_{n}\subseteq[-2/n,2/n]$,
$\rho_{n}\equiv1$ on $[-1/n,1/n]$, $\Vert\rho_{n}\Vert_{A}\le C$
and 
\[
\lim_{n\to\infty}\Vert\rho_{n}(1-e^{itk})\Vert_{A}=0\qquad\forall k\in\mathbb{Z}.
\]
(trapezoidal functions satisfy all these properties, for example).
We fix some $\delta>0$ and write $f=f_{1}+f_{2}$ where $f_{1}$
is a trigonometric polynomial with $|f_{1}(0)|<\delta$ and $\Vert f_{2}\Vert_{A}<\delta$
(for example we can take $f_{1}$ to be a Ces\`aro partial sum of
$f$). We then have
\begin{align*}
\Vert f_{2}\rho_{n}\Vert_{A} & \le C\Vert f_{2}\Vert_{A}\Vert\rho_{n}\Vert_{A}\le C\delta\\
\Vert f_{1}\rho_{n}\Vert_{A} & \le\bigg|\sum_{k}\widehat{f_{1}}(k)\bigg|\Vert\rho_{n}\Vert_{A}+\sum_{k}|\widehat{f_{1}}(k)|\cdot\Vert(1-e^{itk})\rho_{n}\Vert_{A}\\
 & \le C\delta+\sum_{k}|\widehat{f_{1}}(k)|\cdot o(1).
\end{align*}
Hence for $n$ sufficiently large (depending on $\delta$, of course),
we have $\Vert f\rho_{n}\Vert_{A}\le C\delta$. Using the Taylor expansion
of $z$ near zero we can write
\[
z(f\rho_{n})=\sum_{j=1}^{\infty}\frac{z^{(j)}(0)}{j!}(f\rho_{n})^{j}
\]
and since $\Vert(f\rho_{n})^{k}\Vert_{A}\le C^{k}\Vert f\rho_{n}\Vert_{A}^{k}\le(C\delta)^{k}$
we can make the sum converge in $A$ by taking $\delta$ sufficiently
small. This shows that locally $z\circ f\in A$ and proves the lemma.
\end{proof}
\begin{lem}
\label{lem:ginA}$|g|\in H^{1/2}\cap C$.
\end{lem}
\begin{proof}
By lemma \ref{lem:Banach algebra} $A=H^{1/2}\cap C$ is a Banach
algebra and since we already know $g\in A$ we get that $|g|^{2}=\textrm{Re}(g)^{2}+\textrm{Im}(g)^{2}\in A$.
Since $|g|^{2}>0$ and since $\sqrt{\cdot}$ is analytic in the neighbourhood
of any interval $[\epsilon,1/\epsilon]$, the previous theorem shows
that $|g|=\sqrt{|g|^{2}}\in A$.
\end{proof}
\begin{lem}
\label{lem:A implies uniform convergence}If $f\in H^{1/2}\cap C$
then $f$ has a uniformly converging Fourier expansion
\end{lem}
\begin{proof}
By a standard approximation argument, it is enough to show that 
\begin{equation}
\Vert S_{n}(f)\Vert_{\infty}\le C\Vert f\Vert_{A}\label{eq:kol yeled im doktorat}
\end{equation}
for some universal constant $C$, where $S_{n}(f)$ is the $n^{\textrm{th}}$
Fourier partial sum. Indeed, once (\ref{eq:kol yeled im doktorat})
is established we may write $f=f_{1}+f_{2}$ with $f_{1}$ a trigonometric
polynomial and $\Vert f_{2}\Vert_{A}<\epsilon$ and get 
\[
\varlimsup\Vert f-S_{n}(f)\Vert_{\infty}=\varlimsup\Vert f_{2}-S_{n}(f_{2})\Vert_{\infty}<\epsilon+C\epsilon.
\]
Since $\epsilon>0$ was arbitrary, the lemma is proved. So we need
only demonstrate (\ref{eq:kol yeled im doktorat}).

To prove (\ref{eq:kol yeled im doktorat}) we note that $F_{n}$,
the $n^{\textrm{th}}$ Ces\`aro partial sum of $f$ satisfies $\Vert F_{n}\Vert_{\infty}\le\Vert f\Vert_{\infty}\le\Vert f\Vert_{A}$
so it is enough to prove the claim for $S_{n}-F_{n}$. We write
\begin{align*}
\Vert S_{n}-F_{n}\Vert_{\infty} & =\bigg\Vert\sum_{|k|\le n}\frac{|k|}{n}\widehat{f}(k)e^{ikx}\bigg\Vert_{\infty}=\frac{1}{n}\max_{x}\bigg|\sum_{|k|\le n}\sqrt{|k|}\cdot\sqrt{|k|}\widehat{f}(k)e^{ikx}\bigg|\\
 & \stackrel{(*)}{\le}\frac{1}{n}\sqrt{\sum_{|k|\le n}|k|}\Vert f\Vert_{H^{1/2}}\le C\Vert f\Vert_{H^{1/2}}
\end{align*}
where $(*)$ is from Cauchy-Schwarz. The lemma is proved, and so is
the P\'al-Bohr theorem.
\end{proof}

\section{\label{sec:A-proof-sketch}A proof sketch}

The proof of theorem \ref{thm:main} is probabilistic. Let us compare
the approach of this paper to the approach of our older paper \cite{KO98}.
There we investigated the Dubins-Freedman homeomorphism $\phi$ (see
below for more details) and showed that the Fourier series of $f\circ\phi$
is better behaved than that of $f$, but not quite uniformly converging
or even bounded. Here we again start with a Dubins-Freedman-like homeomorphism,
but this time we do not investigate its almost sure behaviour (as
in \cite{KO98}) but its average behaviour. We then start reducing
the randomness step by step, always keeping the same averaged behaviour,
until reaching a degenerate random variable, and that is the required
homeomorphism.

In this section we are going to sketch a proof of an easier result:
that for every continuous function $f$ there is a H\"older homeomorphism
$\varphi$ such that $f\circ\varphi$ has a uniformly bounded Fourier
expansion. This result is weaker than our main result in two aspects:
the homeomorphism is only H\"older rather than absolutely continuous
(this is the more important difference); and the Fourier series is
uniformly bounded rather than uniformly converging. After the sketch
we will briefly describe what additional ideas are needed to get the
main result (whose proof, of course, occupies the rest of the paper).
\begin{lem}
\label{lem:renormalisation simplified}Let $\gamma>0$ and $M$ be
two parameters. Assume $(v_{i,j}:i\in\{1,\dotsc,n\},j\in\mathbb{N})$
satisfy that for every $j$ there is an $l(j)\in\{1,\dotsc,n\}$ and
a $b(j)\in\mathbb{N}$ such that
\[
|v_{i,j}|\le\min\Big\{\frac{1}{|(i-l(j))\bmod n|+1},\frac{1}{b(j)}\Big\}
\]
and assume further that $|\{j:l(j)=l,b(j)=b\}|\le Mb^{\gamma}$ for
all $l$ and $b$. Then there exists a choice of signs $\epsilon_{1},\dotsc,\epsilon_{n}\in\{\pm1\}$
such that 
\[
\bigg|\sum_{i=1}^{n}\epsilon_{i}v_{i,j}\bigg|\le\frac{C}{b(j)^{1/50}}\qquad\forall j,
\]
where $C$ is a constant that depends only on $\gamma$ and $M$.

Here and below $\bmod n$ means a number in $\{-\lfloor(n-1)/2\rfloor,\dotsc,\lfloor n/2\rfloor\}$.
\end{lem}
This lemma and its proof are inspired by a result of Kashin \cite{K79},
who proved a discrete version of Menshov's correction theorem. Roughly
the proof goes as follows. We divide the values of $i$ into blocks,
and choose the signs randomly in each block, ensuring good local behaviour.
We then group the blocks into higher level blocks, and choose signs
randomly inside each second-level block, multiplying with the previously
chosen signs, and so on. This is similar to the method of \emph{renormalisation}
in mathematical physics. 

There is a connection between lemma \ref{lem:renormalisation simplified}
and Koml\'os' conjecture. Recall that Koml\'os' conjecture asks
whether under the condition $\sum_{j}|v_{i,j}|^{2}\le1$ for all $i$
one may choose $\epsilon_{i}$ such that $\big|\sum_{i}\epsilon_{i}v_{i,j}\big|\le C$
for all $j$. We could have used Koml\'os' conjecture for lemma \ref{lem:renormalisation simplified}
if we needed the lemma only for $\gamma<1$, if we did not have the
factor $b(j)^{1/50}$, and if it were proved. Unfortunately, we need
lemma \ref{lem:renormalisation simplified} for $\gamma=91$ (see
page \pageref{pg:91}) and Koml\'os' conjecture is not proved. 

This lemma is the only part of the proof of the H\"older result which
exists in the main result verbatim. The reader is invited to see the
details in the proof of lemma \ref{lem:renormalisation} (lemma \ref{lem:renormalisation}
has two additional parameters, when compared to lemma \ref{lem:renormalisation simplified},
which are needed just because the proof is by induction). 

We have opted to state this technical-sounding lemma first because,
in a way, it was our starting point. Inspired by the use of Spencer's
theorem (a partial result on Koml\'os' conjecture, see \cite{S85})
to solve a version of the flat polynomials conjecture \cite{BBMST20},
we asked ourselves whether these techniques might be useful for Luzin's
conjecture. Even though that path did lead us to a proof, eventually
our proof does not use Spencer's theorem or any other result from
that area.

Continuing, throughout we will identify the circle $\mathbb{T}$ with
the interval $[0,1)$. We will consider homeomorphisms which fix 0,
so they are simply increasing functions from $[0,1]$ into itself.

Let $I$ be a map giving to each dyadic rational $d=k2^{-n}\in(0,1)$
an interval $I(d)\subseteq[0,1]$, possibly degenerate. We call such
functions `an RH-descriptor' where RH stands for `random homeomorphism'.
For each $I$ we introduce a random increasing function $\phi_{I}$
as follows. We set $\phi(0)=0$ and $\phi(1)=1$ deterministically.
Next, we choose $\phi(\frac{1}{2})$ uniformly in $I(\frac{1}{2})$.
We then choose $\phi(\frac{1}{4})$ uniformly in $\phi(\frac{1}{2})I(\frac{1}{4})$
and choose $\phi(\frac{3}{4})$ uniformly in $\phi(\frac{1}{2})+(1-\phi(\frac{1}{2}))I(\frac{3}{4})$.
The process then continues. Assume that we have already chosen $\phi(k2^{-(n-1)})$
for all $k=0,\dotsc,2^{n-1}$. We then, for each odd $1\le k<2^{n}$,
let $\phi(k2^{-n})$ be uniform in the interval 
\[
\phi((k-1)2^{-n})+\big(\phi((k+1)2^{-n})-\phi((k-1)2^{-n})\big)I(k2^{-n}).
\]
This allows to define $\phi$ on all dyadic rationals, after which
it can be continued to the entire interval $[0,1]$ by continuity.
For example, if $I\equiv[0,1]$ then the resulting $\phi_{I}$ is
exactly the Dubins-Freedman random homeomorphism \cite{DB67}. Our
focus in this paper, though, is on the case where $I(d)\subset J$
for some $J$ strictly contained in $[0,1]$, in which case continuity,
and even the H\"older property, are not only trivial but also hold
deterministically, and not just with probability 1, as they do for
the Dubins-Freedman homeomorphism.

Our first two lemmas are for constant $I$. It will be convenient,
therefore, to make the following definition. For an $q\in(0,\frac{1}{2})$
we let $I_{q}\equiv[\frac{1}{2}-q,\frac{1}{2}+q]$ and $\psi=\psi_{q}=\phi_{I_{q}}$.
\begin{lem}
\label{lem:densities}Let $f$ be continuous and let $x,y\in[0,\frac{1}{2}]$.
Then
\[
|\mathbb{E}(f(\psi(x)))-\mathbb{E}(f(\psi(y)))|\le C||f||_{\infty}\sqrt{\frac{|x-y|}{x+y}}
\]
where $C$ may depend on $q$.
\end{lem}
This lemma is not very surprising, as it merely says that if $x$
and $y$ are close (and not too close to 0 or 1) then the densities
of $\psi(x)$ and $\psi(y)$ are quite close. Unfortunately, the densities
of $\psi(x)$ do not have nice closed formulas (the definition of
$\psi$ can certainly be traced to gives closed formulas for $x=k2^{-n}$
for small $n$, but the formulas become unmanageable quickly). The
proof, very roughly, is as follows. We first note that $\mathbb{E}(\psi(x))=x$
(this holds for any $\phi_{I}$ just under the condition that all
$I(d)$ are symmetric around $\frac{1}{2}$). From this we get that
$\mathbb{E}(\psi(2x)-\psi(2y))=2x-2y$, that is the variables $\psi(2x)$
and $\psi(2y)$ are close in the 1-Wasserstein distance. We now note
that $\psi(x)-\psi(y)$ has the same distribution as $U\cdot(\psi(2x)-\psi(2y))$
where $U$ is an independent random variable uniform on $[\frac{1}{2}-q,\frac{1}{2}+q]$.
This smoothens the estimate on the probabilities and makes it into
an estimate on densities (this is similar to the way one uses convolution
to translate an estimate in a cruder norm to an estimate in a finer
norm --- the convolution here is multiplicative but the idea is the
same). We omit all further details of the proof. The square root in
lemma \ref{lem:densities} is just an artefact of our proof, the real
behaviour is probably $|x-y|/(x+y)$. This argument is not used as
is in the proof of the main result, but formula (\ref{eq:Holder away})
is quite similar (but has a different proof from the one sketched
above). 
\begin{lem}
\label{lem:two intervals Holder}Let $J_{1},J_{2}\subseteq[\frac{1}{2}-q,\frac{1}{2}+q]$
be two intervals with $d(y_{1},y_{2})<\epsilon$ for all $y_{i}\in J_{i}$
for some $\epsilon>0$ (in other words, the intervals are both short
and close to one another). Let $f$ and $g$ be two continuous functions.
Then
\[
\bigg|\int_{0}^{1}\big(\mathbb{E}(f\circ\psi\,|\,\psi(\tfrac{1}{2})\in J_{1})-\mathbb{E}(f\circ\psi\,|\,\psi(\tfrac{1}{2})\in J_{2})\big)g\bigg|\le C\epsilon^{c}||f||_{\infty}||g||_{\infty}.
\]
\end{lem}
The proof is similar to the proof of lemma \ref{lem:densities} and
we omit it. The corresponding lemma in the proof of the main result
is lemma \ref{lem:j}, page \pageref{lem:j}.

The next lemma is the main lemma in the proof of the result. It implements
one step of a `reduction of randomness' strategy, that is of starting
from $\phi_{I}$ for $I\equiv[\frac{1}{2}-q,\frac{1}{2}+q]$ and reducing
the randomness step by step (always keeping around an estimate similar
to the one given by lemma \ref{lem:densities}) until reaching a single
homeomorphism with good properties. Unfortunately, the formulation
is a mouthful.
\begin{lem}
\label{lem:main Holder}Let $f$ be a continuous function satisfying
$||f||_{\infty}\le1$, let $q\in(0,\frac{1}{2})$ and let $n,j\in\mathbb{N}$.
Let $I$ be a function from the dyadic rationals into subintervals
of $[\frac{1}{2}-q,\frac{1}{2}+q]$, with the following properties:
\begin{enumerate}
\item For all $m<n$ and all $k\in\{1,\dotsc,2^{m}-1\}$ the interval $I(k2^{-m})$
is degenerate, i.e.\ a single point.
\item For all $k\in\{1,\dotsc,2^{n}-1\}$ odd $|I(k2^{-n})|=2^{2-j}q$.
\item For all $m>n$ and all $k\in\{1,\dotsc,2^{m}-1\}$ odd, $I(k2^{-m})=[\frac{1}{2}-q,\frac{1}{2}+q]$.
\end{enumerate}
Then there exists a function $J$ from the dyadic rationals into subintervals
of $[\frac{1}{2}-q,\frac{1}{2}+q]$ satisfying properties (i) and
(iii) and such that for each $k\in\{1,\dotsc,2^{n-1}\}$ odd, $J(k2^{-n})$
is either the left or right half of $I(k2^{-n})$ (in particular,
property (ii) is satisfied for $J$ with $j+1$ instead of $j$).

Further, for every $u\in\mathbb{N}$, every $r\in\{2^{u-1},\dotsc,2^{u}-1\}$
and every $\xi\in2^{-u-2}\mathbb{Z}\cap[0,1)$ we have the estimate
\begin{equation}
\bigg|\int_{E_{\xi,n}}(\mathbb{E}(f(\phi_{I}(y)))-\mathbb{E}(f(\phi_{J}(y))))\cdot D_{r}(x-\xi)\,dx\bigg|\le C2^{-c(j+|n-u|)}\label{eq:intExin Holder}
\end{equation}
where 
\[
E_{\xi,n}\coloneqq\begin{cases}
[0,1]\setminus([-2^{2-n},2^{2-n}]+2^{1-n}\lfloor\xi2^{n-1}\rfloor & u>n\\{}
[0,1] & u\le n
\end{cases}
\]
and where $D_{r}$ is Dirichlet's kernel.
\end{lem}
Before starting the proof, let us discuss quickly the integration
interval $E_{\xi,n}$. In fact, this has been added only for convenience.
Lemma \ref{lem:main Holder} could have been proved with (\ref{eq:intExin Holder})
replaced with 
\[
\bigg|\int_{0}^{1}(\mathbb{E}(f(\phi_{I}(y)))-\mathbb{E}(f(\phi_{J}(y))))\cdot D_{r}(x-\xi)\,dx\bigg|\le C2^{-c(j+|n-u|)}.
\]
This would certainly simplify the formulation of the lemma, but, unfortunately,
would complicate significantly its proof. This is due to various technical
issues stemming from the vicinity of the peak of the Dirichlet kernel
(the complications arise in the proofs of lemmas \ref{lem:Fpm Holder}
and \ref{lem:int Delta small Holder} below, and since these are not
spelled out in detail, the reader would have to trust us on this point.
But indeed, the complications incurred are significant). Hence we
opted for the statement above which `delays' the handling of the peak
of $D_{r}$ until $n\approx r$.
\begin{proof}
[Sketch of proof of lemma \ref{lem:main Holder}]To choose $J$ is
equivalent to choosing $2^{n-1}$ signs, as our only freedom is in
choosing, for $k\in\{1,\dotsc,2^{n}-1\}$ odd, whether we take $J(k2^{-n})$
to be the left or right half of $I(k2^{-n})$. Denote therefore, for
any $\epsilon=(\epsilon_{1},\epsilon_{3},\dotsc,\epsilon_{2^{n}-1})$,
$\epsilon_{k}\in\{\pm1\}$, the corresponding $J$ by $J^{\epsilon}$
($\epsilon_{k}=1$ means that we take the right half of $I(k2^{-n})$
and $\epsilon_{k}=-1$ means we take the left half). With this notation,
our goal becomes finding some $\epsilon$ such that $J^{\epsilon}$
satisfies (\ref{eq:intExin Holder}).

Fix one $k\in\{1,\dotsc,2^{n-1}\}$, odd. Denote 
\[
F^{\pm}(x)=\begin{cases}
\mathbb{E}(f(\phi_{J^{\epsilon^{\pm}}}(x)) & x\in[(k-1)2^{-n},(k+1)2^{-n}]\\
0 & \text{otherwise}
\end{cases}
\]
where $\epsilon^{\pm1}$ is an $\epsilon$ with $\epsilon_{k}=\pm1$
and the rest chosen arbitrarily (clearly, for $x\in[(k-1)2^{-n},(k+1)2^{-n}]$
only the choice of $\epsilon_{k}$ matters). Denote
\[
\Delta_{k}(x)\coloneqq\frac{1}{2}(F^{+}(x)-F^{-}(x)).
\]
We get $\mathbb{E}(f(\phi_{J^{\epsilon^{\pm}}}(x)))=\mathbb{E}(f(\phi_{I}(x))\pm\Delta_{k}(x)$
for all $x\in[(k-1)2^{-n},(k+1)2^{-n}]$. Summing over $k$ gives
\[
\mathbb{E}(f\circ\phi_{J^{\epsilon}})=\mathbb{E}(f\circ\phi_{I})+\sum_{k}\epsilon_{k}\Delta_{k}.
\]
Our strategy would be to find estimates for $\int\Delta_{k}D_{r}$
and feed them into lemma \ref{lem:renormalisation simplified}. Here
are two such estimates (the corresponding lemmas in the main proof
are lemmas \ref{lem:rs large} and \ref{lem:j large})\phantom\qedhere
\end{proof}
\begin{lem}
\label{lem:Fpm Holder}If $|r|,|s|>2^{n}$ and $\xi\not\in[(k-3)2^{-n},(k+3)2^{-n}]$
then
\[
\bigg|\int F^{\pm}\sum_{l=r}^{s-1}e(l(x-\xi))\bigg|\le\left(\sqrt{\frac{2^{n}}{|r|}}+\sqrt{\frac{2^{n}}{|s|}}\right)\frac{C}{|2^{n}\xi-k|}
\]
\end{lem}
\noindent (as usual, $e(x)=e^{2\pi ix}$)
\begin{proof}
[Proof sketch]Condition on $\phi_{J^{\epsilon}}(k2^{-n})$ and consider
independently the intervals $[(k-1)2^{-n},k2^{-n}]$ and $[k2^{-n},(k+1)2^{-n}]$.
Let $V$ be one of these intervals. We have that $\phi_{J^{\epsilon}}$
is deterministic on $(k\pm1)2^{-n}$ by assumption and on $k2^{-n}$
by the conditioning, so it is deterministic on both ends of $V$.
This means that $\phi_{J}$ restricted to $V$ is simply an appropriately
linearly mapped version of the $\phi_{I}$ from lemma \ref{lem:densities}.
Applying that lemma gives that $F^{\pm}$ satisfies an appropriate
`H\"older away from the boundaries' estimate. Integrating by parts
gives a good estimate for $\int_{V}F^{\pm}\sum e(l(x-\xi))$, still
conditioned on $\phi_{J^{\epsilon}}(k2^{-n})$ (this is the standard
argument that shows that a H\"older function has power-law decaying
Fourier coefficients, the fact that it is only H\"older away from
the boundaries requires no change in the argument, and the $2^{n}$
factors come from the scaling). Integrating the conditioning gives
lemma \ref{lem:Fpm Holder}.
\end{proof}
\begin{lem}
\label{lem:int Delta small Holder}For every $r<s$ and every $\xi$,
\[
\bigg|\int\Delta_{k}(x)\sum_{l=r}^{s-1}e(l(x-\xi))\,dx\bigg|\le C\min\Big\{\frac{1}{|2^{n}\xi-k|},2^{-n}(s-r)\Big\}2^{-cj}
\]
\end{lem}
Lemma \ref{lem:int Delta small Holder} follows from lemma \ref{lem:two intervals Holder}
using the same scaling argument that was used to conclude lemma \ref{lem:Fpm Holder}
from lemma \ref{lem:densities} and we omit the details. (The factor
$2^{-cj}$ in lemma \ref{lem:int Delta small Holder} comes from the
factor $\epsilon^{c}$ in lemma \ref{lem:two intervals Holder}. Unlike
in lemma \ref{lem:Fpm Holder} no cancellation in the exponential
sum is used here, this lemma uses only that the maximum of $\sum e(l(x-\xi))$
on $[(k-1)2^{-n},(k+1)2^{-n}]$ can be estimated by $\min\{|\xi-k2^{-n}|^{-1},s-r\}$.
The cancellation would come in by combining this lemma with lemma
\ref{lem:Fpm Holder}, later).
\begin{proof}
[Proof of lemma \ref{lem:main Holder}, continued]For every $u\in\mathbb{N}$
define $U\coloneqq2^{\max\{u,n\}+2}$. For every $k\in\{1,\dotsc,2^{n}-1\}$
odd and every $\xi\in\frac{1}{U}\mathbb{Z}\cap[0,1)$ define
\[
w_{k,\xi,u}^{0}=\int\Delta_{k}D_{2^{u}}(x-\xi)\,dx.
\]
Next, for every $u\in\mathbb{N}$, for every $0\le s<u$ and every
$t\in[2^{u-1},2^{u})$ divisible by $2^{s+1}$, every $k\in\{1,\dotsc,2^{n-1}\}$
odd and every $\xi\in\frac{1}{U}\mathbb{Z}\cap[0,1)$ we define 
\begin{align*}
w_{k,\xi,u,s,t}^{+} & =\int\Delta_{k}(x)\sum_{z=t+1}^{t+2^{s}}e(z(x-\xi))\,dx,\\
w_{k,\xi,u,s,t}^{-} & =\int\Delta_{k}(x)\sum_{z=-t-2^{s}}^{-t-1}e(z(x-\xi))\,dx.
\end{align*}
We rearrange the vectors $w^{0}$, $w^{+}$ and $w^{-}$ into one
long list $v_{k,j}$ (we consider them to be vectors in $k$, so for
example we could take $v_{k,1}=w_{k,0,1}^{0}$ if we like to start
with $w^{0}$, or we could take $v_{k,1}=w_{k,0,0,0,2}^{+}$ if we
like to start with $w^{+}$). We now apply lemma \ref{lem:renormalisation simplified}
with $n_{\textrm{lemma \ref{lem:renormalisation simplified}}}=2^{n}$.
Lemma \ref{lem:renormalisation simplified} requires us to find, for
every $v_{k,j}$ some $l(j)$ and some $b(j)$ such that 
\begin{equation}
|v_{k,j}|\le\min\Big\{\frac{1}{|k-l(j)\bmod n|+1},\frac{1}{b(j)}\Big\}\label{eq:vkj sketch}
\end{equation}
and to count how many $j$ exist for each choice of $l$ and $b$.
We take $l(j)=\lfloor2^{n}\xi\rfloor$ for the corresponding $\xi$,
and this gives an estimate of the form (\ref{eq:vkj sketch}) from
lemmas \ref{lem:Fpm Holder} and \ref{lem:int Delta small Holder}.
Checking which $b$ is appropriate for each vector is straightforward
and we save the reader all the index checking (note that occasionally
we have to combine the estimates from the two lemmas using the fact
that for any two numbers $x$ and $y$, $\min\{x,y\}\le\sqrt{xy}$).
Eventually we get that 
\[
\{j:b(j)=b,l(j)=l\}\le Cb^{9}\qquad\forall b,l
\]
for some absolute constant. The conclusion of lemma \ref{lem:renormalisation simplified}
is then that there exists a choice of $\epsilon_{k}$ such that
\begin{equation}
\bigg|\sum_{k}\epsilon_{k}v_{k,j}\bigg|\le\frac{C}{b(j)^{1/50}}\qquad\forall j.\label{eq:renormalisation output Holder}
\end{equation}
Finally, to estimate $\int\Delta D_{r}$ for an arbitrary $r$ write
$r$ in its binary expansion
\[
r=2^{u-1}+\sum2^{s_{i}}\qquad t_{i}\coloneqq2^{u-1}+2^{s_{1}}+\dotsb+2^{s_{i}}
\]
for some decreasing sequence $s_{i}$. We get
\[
D_{r}=D_{2^{u-1}}+\sum_{i}\bigg(\sum_{z=t_{i-1}+1}^{t_{i}}e(zx)+\sum_{z=-t_{i}}^{-t_{i-1}-1}e(zx)\bigg).
\]
Hence for any $\xi\in\frac{1}{U}\mathbb{Z}\cap[0,1)$ we have
\begin{multline*}
\int\sum_{k}\epsilon_{k}\Delta_{k}(x)D_{r}(x-\xi)\,dx\\
=\bigg(\sum_{k}\epsilon_{k}w_{k,\xi,u}^{0}+\sum_{i}\sum_{k}\epsilon_{k}w_{k,\xi,u,s_{i},t_{i-1}}^{+}+\sum_{i}\sum_{k}\epsilon_{k}w_{k,\xi,u,s_{i},t_{i-1}}^{-}\bigg).
\end{multline*}
Plugging in the estimate (\ref{eq:renormalisation output Holder})
proves the lemma. The factor $b^{1/50}$ in (\ref{eq:renormalisation output Holder})
plays an important role here, as it makes the estimates for $w^{0,\pm}$
decay exponentially as $|u-n|$ increases and as $s$ decreases (the
factor $2^{-cj}$ from lemma \ref{lem:int Delta small Holder} multiplies
all the terms equally and translates to the $2^{-cj}$ factor in lemma
\ref{lem:main Holder} directly). The fact that the integral is only
on $E_{\xi,n}$ and not on the whole of $[0,1]$ comes from the fact
that lemmas \ref{lem:Fpm Holder} and \ref{lem:int Delta small Holder}
do not work for $\xi$ too close to $k2^{-n}$ (lemma \ref{lem:Fpm Holder}
explicitly and lemma \ref{lem:int Delta small Holder} gives a useless
estimate), so we have to drop the corresponding $w^{0}$ and $w^{\pm}$
and remove the corresponding element from $\sum\epsilon_{k}\Delta_{k}$.
\end{proof}
\begin{thm}
\label{thm:Holder}For every continuous function $f$ and every $\lambda<1$
there is a $\lambda$-H\"older homeomorphism $\psi$ such that $S_{r}(f\circ\psi)$
converges boundedly.
\end{thm}
\begin{proof}
[Proof sketch]Fix $q$ sufficiently small such that any RH-descriptor
$I$ with $I(d)\subseteq[\frac{1}{2}-q,\frac{1}{2}+q]$ gives rise
to a homeomorphism $\phi_{I}$ which is $\lambda$-H\"older surely
(not just almost surely). Now construct a sequence of RH-descriptors
as follows. We start with $I_{0}\equiv[\frac{1}{2}-q,\frac{1}{2}+q]$.
We invoke lemma \ref{lem:main Holder} with $I_{\textrm{lemma \ref{lem:main Holder}}}=I_{0}$,
$n_{\textrm{lemma \ref{lem:main Holder}}}=1$ and $j_{\textrm{lemma \ref{lem:main Holder}}}=1$.
We denote the output of the lemma (i.e.\ $J_{\textrm{lemma \ref{lem:main Holder}}}$)
by $I_{0,1}$. This halves the length of the interval at $\frac{1}{2}$
so now we have $|I_{0,1}(\frac{1}{2})|=q$. We then apply lemma \ref{lem:main Holder}
again, again with $n_{\textrm{lemma \ref{lem:main Holder}}}=1$ but
this time with $j_{\textrm{lemma \ref{lem:main Holder}}}=2$ and $I_{\textrm{lemma \ref{lem:main Holder}}}=I_{0,1}$,
and denote $I_{0,2}\coloneqq J_{\textrm{lemma \ref{lem:main Holder}}}$
(which will have $|I_{0,2}(\frac{1}{2})|=\frac{1}{2}q$). We continue
like this, shrinking the interval at $\frac{1}{2}$ more and more
and finally define 
\[
I_{1}(d)=\lim_{j\to\infty}I_{0,j}(d).
\]
We get that $I_{1}(\frac{1}{2})$ is a single point while $I_{1}(d)=[\frac{1}{2}-q,\frac{1}{2}+q]$
for all $d\ne\frac{1}{2}.$ This makes $I_{1}$ suitable as entry
to lemma \ref{lem:main Holder} with $n_{\textrm{lemma \ref{lem:main Holder}}}=2$
and $j_{\textrm{lemma \ref{lem:main Holder}}}=1$. Again we apply
lemma \ref{lem:main Holder} infinitely many times and get in the
limit an RH descriptor $I_{2}$ with $I_{2}(\frac{1}{4})$, $I_{2}(\frac{1}{2})$
and $I_{2}(\frac{3}{4})$ all degenerate. We continue this process
infinitely many times and get an 
\[
I_{\infty}(d)=\lim_{n\to\infty}I_{n}(d)
\]
which is completely degenerate, i.e.\ corresponds to a single homeomorphism
$\phi$. This $\phi$ will be $\lambda$-H\"older because it is in
the support of $\phi_{I_{0}}$.

To estimate 
\[
\int_{0}^{1}(f\circ\phi)\cdot D_{r}
\]
we simply sum the differences coming from lemma \ref{lem:main Holder}.
Since the error there is $C2^{-c(j+|n-u|)}$, the sum over $j$ simply
gives $C2^{-c|n-u|}$. A crude estimate for the difference between
$E_{n,\xi}$ and $E_{n+1,\xi}$ gives that this can also be bounded
by $C2^{-c|n-u|}$. We get, for every $u\in\mathbb{N}$, every $r\in\{2^{u-1},\dotsc,2^{u}-1\}$
and every $\xi\in2^{-u-1}\mathbb{Z}\cap[0,1)$ that
\[
\bigg|\int_{0}^{1}(f\circ\phi)\cdot D_{r}(\xi)\bigg|\le C\sum_{n=1}^{\infty}2^{-c|n-u|}
\]
which is bounded by a constant independent of $u$. Applying Bernstein's
inequality shows that the estimate holds for all $\xi\in[0,1]$. The
theorem is proved.
\end{proof}
\phantomsection \label{pg:discuss qf}What is needed to go from theorem
\ref{thm:Holder} to our main theorem? Clearly we can no longer allow
our random homeomorphism the flexibility to change by a constant proportion
(our $q$) at every step, as that could never be absolutely continuous.
We have to allow more flexibility where our function $f$ fluctuates
a lot, and less flexibility where $f$ is more flat. In other words,
we need to make $q$ into a function of the relevant dyadic rational
$d$, such that `the local $q$' depends on the behaviour of $f$
in the appropriate area. The `flatness' of $f$ is naturally encoded
using its Haar decomposition and indeed, the necessary condition for
the homeomorphism to be absolutely continuous is a bound for certain
sums of squares of Haar coefficients. Interestingly, these sums turn
out to be dyadic BMO functions, and can thus be estimated by the (dyadic)
John-Nirenberg inequality, see lemma \ref{lem:z Olv}.

Had we applied this `local $q$' strategy naively, we would have ruined
the nice local structure of $\phi$, i.e.\ the fact that every dyadic
interval can be handled with little knowledge of the surrounding,
a fact that was extremely useful in the proofs of lemmas \ref{lem:Fpm Holder}
and \ref{lem:int Delta small Holder}. Why would that ruin this local
structure? Because the local $q$ near $d$ needs to depend on $f$
in the vicinity of $\phi(d)$, which is a global property. We solve
this problem by replacing $\phi$ with $\phi^{-1}$, its inverse as
a homeomorphism. In other words, instead of first choosing $\phi(\frac{1}{2})$
(as we did in the proof sketch above) we first choose which $x$ will
have $\phi(x)=\frac{1}{2}$, and we choose it uniformly in an interval
$[\frac{1}{2}-q,\frac{1}{2}+q]$ where $q$ depends on the behaviour
of $f$ near $\frac{1}{2}$. We then choose $x_{2}$ such that $\phi(x_{2})=\frac{1}{4}$,
and we choose it uniformly in an interval symmetric around $x/2$,
whose width depends on the behaviour of $f$ near $\frac{1}{4}$,
and so on.

This change, however, breaks lemma \ref{lem:renormalisation simplified}
which requires the partition into intervals `in the $x$ axis' to
be uniform, but our replacement of $\phi$ with $\phi^{-1}$ made
the partition in the $y$ axis uniform, while the partition in the
$x$ axis becomes random. To fix this problem we do not break all
intervals at every step (i.e.\ whenever we move from $I_{n}$ to
$I_{n+1}$ in the proof of the theorem). Instead we break into halves
only intervals larger than a certain threshold, leaving the smaller
intervals fixed. The details of this can be seen in the beginning
of \S\ref{sec:Reducing-randomness}, page \pageref{sec:Reducing-randomness}.

Finally, the transition from `uniformly bounded Fourier series' to
`uniformly converging Fourier series' is not particularly problematic.
The modulus of continuity of $f$ has to be taken into account in
lemma \ref{lem:main} (the analog of lemma \ref{lem:main Holder}
in the proof of the main result). The proof of the theorem from lemma
\ref{lem:main} then proceeds by showing that $\int_{E_{\xi,n}}\mathbb{E}(f\circ\phi_{I_{n}})D_{r}$
is a good approximation of $\mathbb{E}(f(\phi_{I_{n}}(\xi)))\int_{E_{\xi_{n}}}D_{r}$.
As $n\to\infty$ the first integral converges to $\int_{0}^{1}(f\circ\phi_{I_{\infty}})D_{r}=S_{r}(f\circ\phi_{I_{\infty}};\xi)$,
the second integral converges to $\int_{0}^{1}D_{r}=1$ and of course
$\mathbb{E}(f(\phi_{I_{n}}(\xi)))\to f(\phi_{I_{\infty}}(\xi))$.

The details of all this fill the next 4 sections.

\section{A hierarchical probabilistic construction}

Our first lemma is proved using a hierarchical random construction,
in the spirit of Kashin (and also related to the Koml\'os conjecture,
as discussed in \S \ref{sec:A-proof-sketch}). Its formulation is
a bit strange because we had to add two parameters ($\alpha$ and
$\beta$) to make the induction work. We note that lemma \ref{lem:renormalisation simplified}
follows from it by simply setting $\alpha=1$ and $\beta=\frac{1}{50}$
(and increasing $M$ to 2, if necessary).
\begin{lem}
\label{lem:renormalisation}Let $\alpha\in(0.99,1]$, $\beta\in[\frac{1}{50},\frac{1}{25})$,
$\gamma>0$ and $M\ge2$ be some parameters. Let $n\in\mathbb{N}$
and let $v_{i,j}\in\mathbb{R}$ ($i\in\{0,\dotsc,n-1\}$ and $j\in\mathbb{N}$)
and let $l(j)\in\{0,\dotsc,n-1\}$ and $b(j)\in\mathbb{N}$ satisfy
\begin{gather*}
|v_{i,j}|\le\min\Big\{\frac{1}{(|(i-l(j))\bmod n|+1)^{\alpha}},\frac{1}{b(j)}\Big\},\\
|\{j:l(j)=l,b(j)=b\}|\le Mb^{\gamma}\qquad\forall b,l.
\end{gather*}
Then there exists a choice of $\varepsilon_{0},\dotsc,\varepsilon_{n-1}\in\{\pm1\}$
such that
\[
\Big|\sum_{i=0}^{n-1}\varepsilon_{i}v_{i,j}\Big|\le\frac{A(\gamma+1)\log M}{(\alpha-0.99)(\frac{1}{25}-\beta)b(j)^{\beta}}\qquad\forall j
\]
for some absolute constant $A$.
\end{lem}
Recall that mod $n$ means a number in $\{-\lfloor(n-1)/2\rfloor,\dotsc,\lfloor n/2\rfloor\}$.

\begin{proof}
The constant $A$ will be chosen later. We argue by induction on $n$.
The case $n=1$ is clear (if $A>1/\log2$), for all values of $\alpha$,
$\beta$, $\gamma$ and $M$. Assume therefore the claim has been
proved for all $n'<n$ and for all values of $\alpha'$, $\beta'$,
$\gamma'$ and $M'$.

Thus we are given $\alpha$, $\beta$, $\gamma$, $M$, $v$, $l$
and $b$. Let $\alpha'\coloneqq\frac{1}{2}(\alpha+0.99)$ and $\beta'\coloneqq\frac{1}{2}(\beta+1/25)$.
Let $K\ge2$ be some integer parameter to be chosen later. Divide
$\{0,\dotsc,n-1\}$ into blocks of size $K$ (the last block may be
smaller) and let $n'\coloneqq\lceil n/K\rceil$ be the number of blocks.
We will first chose signs $\delta_{1},\dotsc,\delta_{n}\in\{\pm1\}$
such that the sum in each block will be controlled. We denote these
sums by $w$, namely, for every $s\in\{0,\dotsc,n'-1\}$, $j\in\mathbb{N}$
we define 
\[
w_{s,j}\coloneqq\sum_{r=0}^{K-1}\delta_{sK+r}v_{sK+r,j}.
\]
If the last block is shorter than $K$, truncate the sum defining
$w$. We wish to choose the $\delta$ such that for every $s\in\{0,\dotsc,n'-1\}$
and every $j$ such that $|sK-l(j)\bmod n|\ge2K$ we have
\begin{gather}
\Big|w_{s,j}\Big|\le\min\Big\{\frac{\sigma K^{1/2}}{(|sK-l(j)\bmod n|+1)^{\alpha'}},\frac{\sigma K^{1/2}}{b(j)^{\beta/\beta'}}\Big\},\label{eq:bound wsjkl}\\
\sigma\coloneqq\sqrt{\lambda(\gamma+1)\Big(\frac{1}{\alpha-\alpha'}+\frac{1}{\beta'-\beta}+\log M\Big)}\nonumber 
\end{gather}
where $\lambda$ is some absolute constant to be defined shortly.
It will be convenient, when $|sK-l(j)\bmod n|<2K$, to denote $w_{s,j}=0$,
so let us do so. The choice of $\delta$ will be random, uniform,
i.i.d. Since the estimates for each block are independent of the other
blocks, let us fix the number of the block $s$. We divide the $j$
into two sets, $\mathcal{S}$ and $\mathcal{T}$ according to whether
$b(j)\ge(|sK-l(j)\bmod n|+1)^{\alpha'\beta'/\beta}$ (this is $\mathcal{S}$)
or not ($\mathcal{T}$). In the first case the second term in (\ref{eq:bound wsjkl})
is the smaller, so we need to estimate the probability that $|w_{s,j}|\le\sigma K^{1/2}b(j)^{-\beta/\beta'}$.
We use $|v_{sK+r,j}|\le1/b(j)$ and Bernstein's inequality for sums
of random variables tell us that 
\[
\mathbb{P}\Big(|w_{s,j}|\ge\frac{\sigma K^{1/2}}{b(j)^{\beta/\beta'}}\Big)\le2\exp(-c\sigma^{2}b(j)^{2(1-\beta/\beta')}).
\]
Now, for every $b$ there are no more than $Cb^{\beta/\beta'\alpha'}\le Cb^{2}$
relevant values of $l$ (recall that we fixed the value of $s$, and
note that we used the inequalities $\alpha'>0.99$ and $\beta<\beta'$).
Further, for each couple $(b,l)$ we have no more than $Mb^{\gamma}$
values of $j$, so all-in-all we have no more than $CMb^{\gamma+2}$
values of $j$ for $b$. A union bound gives
\begin{align*}
\mathbb{P}\Big(\exists j & \in\mathcal{S},|w_{s,j}|\ge\frac{\sigma K^{1/2}}{b(j)^{\beta/\beta'}}\Big)\\
 & \le\sum_{b=1}^{\infty}C_{1}Mb^{\gamma+2}\exp(-c\sigma^{2}b^{2(1-\beta/\beta')}).
\end{align*}
We claim that if $\lambda$ is sufficiently large then this sum is
small. This is a simple calculation but let us do it in detail nonetheless.
We write
\begin{align*}
\lefteqn{\exp(-c\sigma^{2}b^{2(1-\beta/\beta')})=}\qquad\\
 & =\exp\bigg(-c\lambda(\gamma+1)\Big(\frac{1}{\alpha-\alpha'}+\frac{1}{\beta'-\beta}+\log M\Big)b^{2(\beta'-\beta)/\beta'}\bigg)\\
 & \le M^{-c\lambda}\exp\Big(-\frac{c\lambda(\gamma+1)}{\beta'-\beta}b^{50(\beta'-\beta)}\Big).
\end{align*}
For any $\varepsilon>0$ we have $b^{\varepsilon}\ge e\varepsilon\log b$
and thus the second multiplicand is no more than $\exp(-c\lambda(\gamma+1)50e\log b)=b^{-c\lambda(\gamma+1)50e}$.
Pick $\lambda$ so large such that 
\[
M^{-c\lambda}\le\frac{1}{10\cdot C_{1}M}\qquad b^{-c\lambda(\gamma+1)50e}\le b^{-4-\gamma}
\]
(this can be done independently of $\gamma$ or $M$, recall that
$M\ge2$) and get
\[
\mathbb{P}\Big(\exists j\in\mathcal{S},|w_{s,j}|\ge\frac{\sigma K^{1/2}}{b(j)^{\beta/\beta'}}\Big)\le\frac{1}{10}\sum_{b=1}^{\infty}\frac{1}{b^{2}}<\frac{1}{4}.
\]
This finishes the estimate for $\mathcal{S}$.

In the case that $j\in\mathcal{T}$ i.e.\ that $b(j)<(|sK-l(j)\bmod n|+1)^{\alpha'\beta'/\beta}$
we use 
\[
|v_{sK+r,j}|\le(|sK+r-l(j)\bmod n|+1)^{-\alpha}\le C(|sK-l(j)\bmod n|+1)^{-\alpha}
\]
where the second inequality uses our assumption that $|sK-l(j)\bmod n|\ge2K$.
Again using Bernstein's inequality we have
\begin{multline*}
\mathbb{P}\Big(|w_{s,j}|>\frac{\sigma K^{1/2}}{(|sK-l(j)\bmod n|+1)^{\alpha'}}\Big)\\
\le2\exp\big(-c\sigma^{2}(|sK-l(j)\bmod n|+1)^{2(\alpha-\alpha')}\big).
\end{multline*}
For every value $q$ of $|sK-l(j)\bmod n|$ there are at most $2$
possible values of $l(j)$ which give it. The restriction $b<(|sK-l\bmod n|+1)^{\alpha'\beta'/\beta}$
gives that for every $q$ there are no more than $q^{\alpha'\beta'/\beta}\le q^{2}$
values of $b$, and for each couple $(b,l)$ there are no more than
$Mb^{\gamma}\le M(q^{\alpha'\beta'/\beta})^{\gamma}\le Mq^{2\gamma}$
possibilities for $j$. All-in-all the value $q$ corresponds to at
most $CMq^{2\gamma+2}$ values of $j$. Thus we can bound 
\begin{multline*}
\mathbb{P}\Big(\exists j\in\mathcal{T}\text{ s.t. }|w_{s,j}|\ge\frac{\sigma K^{1/2}}{(|sK-l(j)\bmod n|+1)^{\alpha'}}\Big)\\
\le\sum_{q=1}^{\infty}2\exp(-c\sigma^{2}q^{2(\alpha-\alpha')})\cdot CMq^{2\gamma+2}
\end{multline*}
and a calculation similar to the one done above for $\mathcal{S}$
shows that the value of $\sigma$ ensures that this sum is also smaller
than $\frac{1}{4}$, for $\lambda$ sufficiently large. Fix $\lambda$
to satisfy this requirement (uniformly in $\alpha$, $\alpha'$, $M$
and $\gamma$) and the previous one. Since the sum of the probabilities
of all dissenting events is less than $1$, we see that some choice
of $\delta$ for which (\ref{eq:bound wsjkl}) will be satisfied exists.

Before continuing we make one modification of (\ref{eq:bound wsjkl}).
Recall that we defined $w_{s,j}=0$ whenever $|sK-l(j)\bmod n|<2K$.
If not, i.e.\ if $|sK-l(j)\bmod n|\ge2K$, then we can write 
\[
|sK-l(j)\bmod n|\ge\tfrac{1}{2}K|s-\lfloor l(j)/K\rfloor\bmod n'|
\]
allowing to rewrite (\ref{eq:bound wsjkl}) as
\begin{gather}
\Big|w_{s,j}\Big|\le\min\Big\{\frac{2\sigma K^{1/2-\alpha'}}{(|s-\lfloor l(j)/K\rfloor\bmod n'|+1)^{\alpha'}},\frac{\sigma K^{1/2}}{b(j)^{\beta/\beta'}}\Big\}.\label{eq:bound wsjkl-1}
\end{gather}
We are now in a position to apply the lemma inductively. We need numbers
$v_{i',j,}'$, $i'\in\{0,\dotsc,n'-1\}$ and functions $l':\mathbb{N}\to\{0\dotsc,n'-1\}$,
$b':\mathbb{N}\to\mathbb{N}$, satisfying 
\begin{equation}
|v_{i',j}'|\le\min\left\{ \frac{1}{(|i'-l'(j)\bmod n'|+1)^{\alpha'}},\frac{1}{b'(j)}\right\} \label{eq:vpipj}
\end{equation}
and the other requirements (recall that $\alpha'=\frac{1}{2}(\alpha+0.99)$
and $\beta'=\frac{1}{2}(\beta+1/25)$). We will use
\[
v_{i',j}=\frac{w_{s,j}}{2\sigma K^{1/2-\alpha'}}
\]
With $i'=s$ and $l'(j)=\lfloor l(j)/K\rfloor$, and 
\[
b'(j)=\max\bigg\{\bigg\lfloor\frac{2b(j)^{\beta/\beta'}}{K^{\alpha'}}\bigg\rfloor,1\bigg\}.
\]
With these definitions (\ref{eq:vpipj}) follows immediately from
(\ref{eq:bound wsjkl-1}), so the only thing we need to check is,
for an arbitrary given $b'$ and $l'$, for how many values of $j$
do we have $b'(j)=b'$, $l'(j)=l'$. We first count how many $b$
do we have such that
\begin{equation}
\frac{2b^{\beta/\beta'}}{K^{\alpha'}}\in[b',b'+1)\label{eq:kkprime}
\end{equation}
or in $[0,2)$ in the case of $b'=1$. The number of such $b$ can
be estimated by a simple derivative bound giving 
\[
\tfrac{1}{2}K^{\alpha'}\frac{\beta'}{\beta}\Big(\tfrac{1}{2}K^{\alpha'}(b'+1)\Big)^{\beta'/\beta-1}+1\le CK^{2}b'
\]
where we estimated $\beta'/\beta\le2$, $b'+1\le2b'$ and $\alpha'\le1$.
The bound $CK^{2}b'$ for the number of $b$ satisfying (\ref{eq:kkprime})
also holds in the case that $b'=1$ (in which case we need to verify
how many $b$ satisfy $2b^{\beta/\beta'}K^{-\alpha'}\in[0,2)$, but
only the constant is affected). Adding the facts that each value of
$l'$ corresponds to at most $K$ different values of $l$; and that
each couple $(b,l)$ corresponds to $Mb^{\gamma}$ different values
of $j$ we get that the number of $j$ for each couple $(b',l')$
is at most 
\[
K\cdot M\bigg(\tfrac{1}{2}K^{\alpha'}(b'+1)\Big)^{\gamma\beta'/\beta}\cdot CK^{2}b'\le C_{2}MK^{2}\big(Kb'\big)^{2\gamma+1}
\]
where we again bounded $\beta'/\beta\le2$, $b'+1\le2b'$ and $\alpha'\le1$.
We are ready to apply our induction assumption! We apply it with
$\alpha'$, $\beta'$, $\gamma'\coloneqq2\gamma+1$, 
\[
M'\coloneqq C_{2}MK^{2\gamma+3},
\]
$n'$ and the vectors $w_{s,j}/2\sigma K^{1/2-\alpha'}$. We get
that there exists some signs $\mu_{0},\dotsc,\mu_{n'-1}\in\{\pm1\}$
such that
\[
\Big|\sum_{s=0}^{n'-1}\mu_{s}\frac{w_{s,j}}{2\sigma K^{1/2-\alpha'}}\Big|\le\frac{A(\gamma'+1)\log M'}{(\alpha'-0.99)(\frac{1}{25}-\beta')b'(j)^{\beta'}}\qquad\forall j.
\]
Defining $\varepsilon_{i}=\mu_{\lfloor i/K\rfloor}\delta_{i}$ gives
\begin{equation}
\bigg|\sum_{i=0}^{n-1}\varepsilon_{i}v_{i,j}\bigg|\le C\min\Big\{ K^{1-\alpha}\log K,\frac{K}{b(j)}\Big\}+\frac{2A(\gamma'+1)\sigma K^{1/2-\alpha'}\log M'}{(\alpha'-0.99)(\frac{1}{25}-\beta')b'(j)^{\beta'}}.\label{eq:toomanyprimes}
\end{equation}
The first term is from the fact that we zeroed out $w$ when $|j-sK|<2K$.
When moving back to $v$ we need to bound these $v$ and we bound
them using a naive absolute value bound (we use here that $\sum_{i=1}^{K}i^{-\alpha}\le CK^{1-\alpha}\log K$
with $C$ uniform in $\alpha\le1$, a fact which is easy to check).

To make the calculation manageable, we will now bound the second term
in (\ref{eq:toomanyprimes}) by a sequence of quite rough bounds.
For $M'$ we write
\begin{align*}
\log M' & =C+\log M+(2\gamma+3)\log K\\
 & \le C(\gamma+1)\log M\log K.
\end{align*}
For $b'$ we have 
\begin{align*}
(b')^{-\beta'} & \le2(b'+1)^{-\beta'}\stackrel{\textrm{(\ref{eq:kkprime})}}{\le}2\Big(\frac{2b^{\beta/\beta'}}{K^{\alpha'}}\Big)^{-\beta'}\le2b^{-\beta}K^{\beta'}
\end{align*}
For $\sigma$ we bound
\[
\sigma\le C\frac{\gamma+1}{(\alpha-0.99)(\frac{1}{25}-\beta)}\log M.
\]
Finally $\gamma'+1=2(\gamma+1)$, $(\alpha'-0.99)=\frac{1}{2}(\alpha-0.99)$
and $(\frac{1}{25}-\beta')=\frac{1}{2}(\frac{1}{25}-\beta)$. Inserting
everything into (\ref{eq:toomanyprimes}) gives
\begin{align*}
\frac{(\gamma'+1)\sigma K^{1/2-\alpha'}\log M'}{(\alpha'-0.99)(\frac{1}{25}-\beta')b'^{\beta'}} & \le CK^{1/2-\alpha'+\beta'}\log K\frac{(\gamma+1)^{3}\log^{2}M}{(\alpha-0.99)^{2}(\frac{1}{25}-\beta)^{2}b^{\beta}}\\
 & \le CK^{-1/3}\theta^{3}b^{-\beta}\qquad\theta:=\frac{(\gamma+1)\log M}{(\alpha-0.99)(\frac{1}{25}-\beta)}.
\end{align*}
Together with (\ref{eq:toomanyprimes}) we get
\begin{equation}
\bigg|\sum_{i=0}^{n-1}\varepsilon_{i}v_{i,j}\bigg|\le C\min\Big\{ K^{1-\alpha}\log K,\frac{K}{b(j)}\Big\}+C_{3}AK^{-1/3}\theta^{3}b(j)^{-\beta}\label{eq:sum eps i v ij two terms}
\end{equation}
Recall that we need to show $|\sum\varepsilon_{i}v_{i,j}|\le A\theta b(j)^{-\beta}$.
Now is the time to choose $K$. We choose $K=\lceil C_{4}\theta^{7}\rceil$
for some constant $C_{4}$ sufficiently large such that 
\[
K^{-1/3}\le\frac{1}{2C_{3}\theta^{2}}
\]
(and also such that $K\ge2$). To bound the first term in (\ref{eq:sum eps i v ij two terms}),
note that for $b<\theta^{15}$ we have 
\[
\theta b^{-\beta}\ge\theta b^{-1/25}\ge\theta^{2/5}\ge cK^{2/35}\ge cK^{1-\alpha}\log K
\]
while for $b\ge\theta^{15}$ we have
\[
\theta b^{-\beta}=\frac{\theta b^{1-\beta}}{b}\ge\frac{\theta^{1+15\frac{24}{25}}}{b}\ge c\frac{K}{b}
\]
so for any $b$ we get $\min\{K^{1-\alpha}\log K,\frac{K}{b}\}\le C\theta b^{-\beta}$.
Inserting both estimates into (\ref{eq:sum eps i v ij two terms})
gives
\[
\bigg|\sum_{i=0}^{n-1}\varepsilon_{i}v_{i,j}\bigg|\le C_{5}\theta b(j)^{-\beta}+\frac{1}{2}A\theta b(j)^{-\beta}.
\]
Choosing $A=2C_{5}$ completes the induction, and proves the lemma.
\end{proof}

\section{\label{sec:every q}A family of homeomorphisms}

A dyadic rational is a number $d$ of the form $k2^{-n}$ for some
integers $k$ and $n$. If $k$ is odd we define $\rk(d)=n$ and let
also $\rk(0)=-\infty$. Let $\theta$ be a map from the dyadic rationals
in $(0,1)$ into $[\frac{1}{4},\frac{3}{4}]$. Let $n$ be an integer.
Our goal is to define for every $\theta$ and $n$ a homeomorphism
$\psi_{\theta,n}:[0,1]\to[0,1]$. It turns out that it is a little
more natural to first define $\psi_{\theta,n}^{-1}$ which is a `Dubins-Freedman
style' homeomorphism, so we do that, and then define $\psi$ as its
inverse. But first we need some notation.

We define
\begin{gather*}
T_{d}(x)\coloneqq\max\left\{ 1-2^{\rk(d)}|x-d|,0\right\} \\
\nu_{d}\psi\coloneqq\psi(d+2^{-\rk(d)})-\psi(d-2^{-\rk(d)})
\end{gather*}
In words, $T_{d}$ is a triangle function based on a certain dyadic
interval $I$ (whose centre is $d$). We will only apply the functional
$\nu_{d}$ to increasing functions, and then $\nu_{d}\psi$ is the
variation of $\psi$ over $I$.

We may now define $\psi_{\theta,n}^{-1}$. The definition is by induction
on $n$, with the induction base being
\[
\psi_{\theta,0}^{-1}(x)=x\qquad\forall x\in[0,1].
\]
Assume $\psi_{\theta,n-1}$ has been defined for all $\theta$. We
now define for $n\ge1$,
\begin{equation}
\psi_{\theta,n}^{-1}(x)=\psi_{\theta,n-1}^{-1}(x)+\sum_{\rk(d)=n}T_{d}(x)(\theta(d)-\tfrac{1}{2})\nu_{d}\psi_{\theta,n-1}^{-1}\label{eq:psi by triangles}
\end{equation}
where the sum is over all $d\in(0,1)$ of rank $n$ (totally $2^{\rk(d)-1}$
terms). An equivalent definition is to write, for every $k\in\{0,\dotsc,2^{n}\}$,
\begin{equation}
\psi_{\theta,n}^{-1}(k2^{-n})=\begin{cases}
\psi_{\theta,n-1}^{-1}(k2^{-n}) & \mathrlap{k\text{ even}}\\
\psi_{\theta,n-1}^{-1}((k-1)2^{-n})+\theta(k2^{-n})\;\cdot & \qquad\\
\quad(\psi_{\theta,n-1}^{-1}((k+1)2^{-n})-\psi_{\theta,n-1}^{-1}((k-1)2^{-n})) & \mathrlap{k\text{ odd}}
\end{cases}\label{eq:psi by interpolation}
\end{equation}
and complete linearly between these points. In particular we note
that $\psi_{\theta,n}^{-1}(d)$ stabilises for all dyadic $d$, as
$n\to\infty$, in fact when $n>\rk(d)$.

To understand why we call $\psi^{-1}$ a Dubins-Freedman style homeomorphism
just note that if we were to take $\theta$ uniform in $[0,1]$ (which
we do not allow here) then the result is exactly the Dubins-Freedman
homeomorphism of level $n$.

Finally, let us remark that $\psi^{-1}$ has a recursive formula (which
could also have been used to define $\psi^{-1}$). To state it define
\begin{gather}
A\coloneqq\theta\left(\tfrac{1}{2}\right)\nonumber \\
\theta^{-}(d)=\theta\left(\tfrac{1}{2}d\right)\qquad\theta^{+}(d)=\theta\left(\tfrac{1}{2}+\tfrac{1}{2}d\right)\label{eq:theta pm}
\end{gather}
With these definitions we get 
\begin{equation}
\psi_{\theta,n}^{-1}(x)=\begin{cases}
A\psi_{\theta^{-},n-1}^{-1}\left(2x\right) & x\le\tfrac{1}{2}\\
A+(1-A)\psi_{\theta^{+},n-1}^{-1}\left(2x-1\right) & \text{otherwise.}
\end{cases}\label{eq:induct psi-1-1}
\end{equation}
Similarly we have for $\psi$ itself
\begin{equation}
\psi_{\theta,n}(x)=\begin{cases}
\tfrac{1}{2}\psi_{\theta^{-},n-1}\left(\frac{x}{A}\right) & x\le A\\
\tfrac{1}{2}+\tfrac{1}{2}\psi_{\theta^{+},n-1}\left(\frac{x-A}{1-A}\right) & \text{otherwise.}
\end{cases}\label{eq:defAB}
\end{equation}
Both formulas (which are clearly equivalent) are a special case of
lemma \ref{lem:local psi inverse} below (used for $I=[0,\frac{1}{2}]$
and $I=[\frac{1}{2},1]$). 

We say that $I$ is a dyadic interval if there exist some $n\in\mathbb{Z}^{+}$
and $k\in\{1,\dotsc,2^{n}\}$ such that $I=[(k-1)2^{-n},k2^{-n}]$.
We call $n$ the rank of $I$. For every interval $I\subseteq[0,1]$
we define $L_{I}$ to be the affine increasing map taking $[0,1]$
onto $I$. The following lemma (whose proof is merely playing around
with the definitions) formalises the local structure of $\psi$ and
$\psi^{-1}$.
\begin{lem}
\label{lem:local psi inverse}Let $\theta$ and $n$ be as above.
Let I be a dyadic interval $I=[(k-1)2^{-m},k2^{-m}]$ with $m\le n$.
Denote $\alpha=\psi_{\theta,n}^{-1}((k-1)2^{-m})$ and $\beta=\psi_{\theta,n}^{-1}(k2^{-m})$.
Then for every $\theta$, $n$, $I$ and for every $x\in I$,
\begin{equation}
\psi_{\theta,n}^{-1}(x)=\alpha+(\beta-\alpha)(\psi_{\theta\circ L_{I},n-m}^{-1}(L_{I}^{-1}(x)).\label{eq:local psi inverse}
\end{equation}
\end{lem}
A short version of (\ref{eq:local psi inverse}) is 
\[
\left.\psi_{\theta,n}^{-1}\right|_{I}=L_{\psi_{\theta,n}^{-1}(I)}\circ\psi_{\theta\circ L_{I},n-m}^{-1}\circ L_{I}^{-1}.
\]
For $\psi$ itself this can be written as
\begin{equation}
\left.\psi_{\theta,n}\right|_{\psi_{\theta,n}^{-1}(I)}=L_{I}\circ\psi_{\theta\circ L_{I},n-m}\circ L_{\psi_{\theta,n}^{-1}(I)}^{-1}.\label{eq:psi local}
\end{equation}

\begin{proof}
We first note that $\psi_{\theta,n}^{-1}(x)$ is linear on every dyadic
interval of order $n$, as it is a sum of piecewise linear functions
each of which has its jumps of the derivative on dyadic rationals
of rank $\le n$. This shows the case that $m=n$ as the left hand
side of (\ref{eq:local psi inverse}) in linear on $I$, the right
hand side of (\ref{eq:local psi inverse}) is linear on $I$ (as $\psi_{\theta,0}^{-1}(x)=x$
always) and they are equal at the boundaries of $I$.

We will prove the claim by induction on $n$, with the base case being
the case of $n=m$ just proved. We first note that for $x\in I$ the
only non-zero terms in (\ref{eq:psi by triangles}) are those for
which $d\in I$ and hence for every $x\in I$
\begin{align}
 & \psi_{\theta,n}^{-1}(x)\stackrel{\textrm{(\ref{eq:psi by triangles})}}{=}\psi_{\theta,n-1}^{-1}(x)+\sum_{\substack{\rk(d)=n\\
d\in I
}
}T_{d}(x)(\theta(d)-\tfrac{1}{2})\nu_{d}\psi_{\theta,n-1}^{-1}\label{eq:only d in I}\\
 & =\psi_{\theta,n-1}^{-1}(x)+\sum_{\substack{\mathclap{{\rk(d)=n-m}}\\
d\in(0,1)
}
}T_{L_{I}(d)}(x)(\theta(L_{I}(d))-\tfrac{1}{2})\nu_{L_{I}(d)}\psi_{\theta,n-1}^{-1}.\nonumber 
\end{align}
(we used here that $\rk(L_{I}(d))=\rk(d)+m$ for all $d\in(0,1)$,
and that the boundaries of the interval do not have rank $n$ so we
may omit them). Recall that $\psi_{\theta,n}^{-1}=\psi_{\theta,n-1}^{-1}$
on each dyadic rational of rank $<n$ and in particular on the boundaries
of $I$. Hence $\alpha$ and $\beta$ do not depend on whether we
define them with $n$ or $n-1$. This allows to apply the induction
assumption and get 
\begin{align*}
\psi_{\theta,n-1}^{-1}(x) & =\alpha+(\beta-\alpha)\psi_{\theta\circ L_{I},n-m-1}^{-1}(L_{I}^{-1}(x))\\
\nu_{L_{I}(d)}(\psi_{\theta,n-1}^{-1}) & =(\beta-\alpha)\nu_{d}(\psi_{\theta\circ L_{I},n-m-1}^{-1}).
\end{align*}
Similarly
\begin{gather*}
T_{L_{I}(d)}(x)=T_{d}(L_{I}^{-1}(x)).
\end{gather*}
Inserting all these into (\ref{eq:only d in I}) gives
\begin{multline*}
\psi_{\theta,n}^{-1}(x)=\alpha+(\beta-\alpha)\psi_{\theta\circ L_{I},n-m-1}^{-1}(L_{I}^{-1}(x))+(\beta-\alpha)\times\\
\sum_{\substack{\rk(d)=n-m\\
d\in(0,1)
}
}T_{d}(L_{I}^{-1}(x))((\theta\circ L_{I})(d)-\tfrac{1}{2})\nu_{d}(\psi_{\theta\circ L_{I},n-m-1}^{-1}).
\end{multline*}
and by (\ref{eq:psi by triangles}) this is exactly 
\[
\alpha+(\beta-\alpha)\psi_{\theta\circ L_{I},n-m}^{-1}(L_{I}^{-1}(x)),
\]
as needed.
\end{proof}
\begin{lem}
\label{lem:Holder}If $\theta(d)\in[\epsilon,1-\epsilon]$ for some
$\epsilon>0$ and all dyadic $d$ then $\psi_{\theta,n}$ converges
uniformly as $n\to\infty$ and the limit is a H\"older homeomorphism
of $[0,1].$ As $\epsilon\to\frac{1}{2}$, the H\"older exponent
tends to 1.
\end{lem}
\begin{proof}
Examining (\ref{eq:psi by interpolation}) we see that $\psi^{-1}$
is strictly increasing, hence $\psi$ is well-defined. Further, a
simple induction shows that
\[
\epsilon^{n}\le\psi_{\theta,n}^{-1}(k2^{-n})-\psi_{\theta,n}^{-1}((k-1)2^{-n})\le(1-\epsilon)^{n}.
\]
For dyadic intervals of order $m<n$ we may get the same estimate
since $\psi_{\theta,n}(k2^{-m})=\psi_{\theta,m}(k2^{-m})$. For $m>n$
the linearity of $\psi_{\theta,n}^{-1}$ on dyadic intervals of order
$n$ gives
\[
\psi_{\theta,n}^{-1}(k2^{-m})-\psi_{\theta,n}^{-1}((k-1)2^{-m})\le\left(1-\epsilon\right)^{n}2^{n-m}\le\left(1-\epsilon\right)^{m}\qquad\forall m>n.
\]
The lower bound is similar and we can conclude that
\begin{equation}
2^{-m/\delta}\le\psi_{\theta,n}^{-1}(k2^{-m})-\psi_{\theta,n}^{-1}((k-1)2^{-m})\le2^{-\delta m}\qquad\forall m,n\label{eq:only finite for now.}
\end{equation}
for an appropriate $\delta$ depending only on $\epsilon$, which
converges to $1$ when $\epsilon\to\frac{1}{2}$. Since the series
$\psi_{\theta,n}^{-1}(d)$ stabilises for all dyadic rational $d$,
we get that the limit (denote it by $\psi_{\infty}^{-1}$) also satisfies
(\ref{eq:only finite for now.}).

For a general $x<y\in[0,1]$ we may find a dyadic interval contained
in $[x,y]$ of size at least $\frac{1}{4}(y-x)$. In the other direction
we may find two dyadic intervals $I_{1}$, $I_{2}$ such that $[x,y]\subseteq I_{1}\cup I_{2}$
and such that $|I_{1}\cup I_{2}|\le4(y-x)$. Hence
\begin{equation}
(\tfrac{1}{4}(y-x))^{1/\delta}\le\psi_{\theta,n}^{-1}(y)-\psi_{\theta,n}^{-1}(x)\le(4(y-x))^{\delta}\label{eq:Holder psi}
\end{equation}
for all $n$, finite or infinite. Reversing gives a similar estimate
for $\psi$, namely
\begin{equation}
\tfrac{1}{4}|x-y|^{1/\delta}\le|\psi_{\theta,n}(x)-\psi_{\theta,n}(y)|\le4|x-y|^{\delta}.\label{eq:Holder 34}
\end{equation}
Since $\psi_{\theta,n}^{-1}$ stabilises on all dyadic rationals,
we get a dense set of numbers where $\psi_{\theta,n}$ stabilises,
and on these numbers the limit satisfies (\ref{eq:Holder 34}). This
of course shows that $\psi_{\theta,n}(x)$ converges uniformly, and
that the limit also satisfies (\ref{eq:Holder 34}), for all $x,y\in[0,1]$.
\end{proof}
Denote 
\[
\psi_{\theta,\infty}\coloneqq\lim_{n\to\infty}\psi_{\theta,n}.
\]
We note for future use that (\ref{eq:defAB}) extends to $n=\infty$,
namely
\begin{equation}
\psi_{\theta,\infty}(x)=\begin{cases}
\frac{1}{2}\psi_{\theta^{-},\infty}(\frac{x}{A}) & x\le A\\
\frac{1}{2}+\frac{1}{2}\psi_{\theta^{+},\infty}(\frac{x-A}{1-A}) & \text{otherwise,}
\end{cases}\label{eq:psi infinity recurs}
\end{equation}
where $A$ is again $\theta(\frac{1}{2})$. Let us also remark that
during the proof of lemma \ref{lem:Holder} we defined $\psi_{\theta,\infty}^{-1}=\lim_{n\to\infty}\psi_{\theta,n}^{-1}$,
but this notation is not ambiguous since this limit is also the inverse
of $\psi_{\theta,\infty}$, since both $\psi_{n}\to\psi_{\infty}$
and $\psi_{n}^{-1}\to\psi_{\infty}^{-1}$ are uniform (the fact that
$\psi_{n}^{-1}\to\psi_{\infty}^{-1}$ uniformly is implicit in the
proof of lemma \ref{lem:Holder}).

We now move to conditions on $\theta$ which will guarantee that $\psi_{\theta,\infty}$
is absolutely continuous. This clearly requires $\theta$ to be usually
close to $\frac{1}{2}$, so let us introduce a function $q$ from
the dyadic intervals into $[0,\infty)$ and ask that $|\theta(d)-\frac{1}{2}|\le q(d)$
for all dyadic $d$ (it will actually be convenient later to add a
constant and require $|\theta-\frac{1}{2}|\le cq$, but let us ignore
this for a bit). Thus we need to find some condition on $q$ that
will ensure that $\psi_{\theta,\infty}$ is absolutely continuous
whenever $|\theta-\frac{1}{2}|\leq q$. To formulate the condition,
we need two auxiliary functions. The first, $Q_{n}:[0,1]\to[0,\infty)$
is defined by
\[
Q_{n}(x)=Q_{q,n}(x)=q(d)
\]
whenever $x$ is in some dyadic interval $I$, $|I|=2^{-n}$, and
$d$ is the middle of $I$. With $Q_{n}$ defined we denote
\begin{equation}
z(x)=z_{q}(x)=\sum_{n=0}^{\infty}Q_{n}(x).\label{eq:def z}
\end{equation}
We can now state our condition.
\begin{lem}
\label{lem:z to psi bound}There exists a constant $\nu_{0}$ such
that for any $q$ for which 
\begin{equation}
|\{x:|z_{q}(x)|>\lambda\}|\le Ce^{-\lambda}\qquad\forall\lambda>0\label{eq:z cond}
\end{equation}
and for any $\theta$ such that $|\theta-\frac{1}{2}|\leq\nu_{0}q$,
the homeomorphisms $\psi_{\theta,\infty}$ and $\psi_{\theta,\infty}^{-1}$
are absolutely continuous.

Further, for any $p<\infty$ there exists a constant $\nu_{1}(p)$
such that $|\theta-\frac{1}{2}|\le\nu_{1}(p)q$ ensures that $||\psi_{\theta,\infty}'||_{p},||(\psi_{\theta,\infty}^{-1})'||_{p}\le C(p)$.
\end{lem}
Notice that (\ref{eq:z cond}) does not prohibit $z$ from being $\infty$
on a set of zero measure.
\begin{proof}
Assume $|\theta-\frac{1}{2}|\le\nu q$ for some $\nu$. Examine a
finite $n$. Since $\psi_{\theta,n}^{-1}$ is piecewise linear we
may consider the derivative. The relation between $z$ and $\psi$
is given by the following inequality. We claim that for every $n$
and every $x\in[0,1]$,
\begin{equation}
\exp\bigg(-C\nu\sum_{i=0}^{n}Q_{i}(x)\bigg)\le\left(\psi_{\theta,n}^{-1}\right)'(x)\le\exp\bigg(C\nu\sum_{i=0}^{n}Q_{i}(x)\bigg).\label{eq:psi-1'<exp sum q}
\end{equation}
where at dyadic $x$ the left derivative is bounded by the left limits
of the terms, and ditto on the right (and these derivatives, of course,
do not need to be equal). 

We prove (\ref{eq:psi-1'<exp sum q}) by induction on $n$. The base
case $n=0$ is trivial as $\psi_{0}^{-1}(x)=x$ so $\left(\psi_{0}^{-1}\right)'=1$
and satisfies the requirement. Assume therefore that (\ref{eq:psi-1'<exp sum q})
is proved for $n-1$.

Recall (\ref{eq:induct psi-1-1}) and the definitions of $\theta^{\pm}$
and $A$ just before it. We wish to differentiate (\ref{eq:induct psi-1-1})
and for this we note that $A$ is just a number, and further satisfies
\begin{equation}
\exp(-C\nu Q_{q,0})\le2A\le\exp(C\nu Q_{q,0})\label{eq:2A bounds}
\end{equation}
where we used $|\theta\left(\frac{1}{2}\right)-\frac{1}{2}|\le\nu q(\frac{1}{2})=\nu Q_{q,0}(x)$
for all $x$.

Next define $q^{-}$ and $q^{+}$ as for $\theta$, namely $q^{-}(x)=q(\frac{1}{2}x)$
and $q^{+}(x)=q(\frac{1}{2}+\frac{1}{2}x)$. A simple check shows
that, for $x\in[0,\frac{1}{2})$,
\begin{equation}
Q_{q,n}(x)=Q_{q^{-},n-1}(2x).\label{eq:LQ induct}
\end{equation}
Hence, for $x\in[0,\frac{1}{2}]$ non-dyadic,
\begin{align*}
\left(\psi_{\theta,n}^{-1}\right)'(x) & \stackrel{\textrm{\clap{(\ref{eq:induct psi-1-1})}}}{=}2A\left(\psi_{\theta^{-},n-1}^{-1}\right)'\left(2x\right)\\
 & \stackrel{\textrm{\clap{(\ref{eq:2A bounds})}}}{\le}\exp(C\nu Q_{q,0}(x))\left(\psi_{\theta^{-},n-1}^{-1}\right)'\left(2x\right)\\
 & \stackrel{\mathclap{(*)}}{\le}\exp\Big(C\nu Q_{q,0}(x)+C\nu\sum_{i=0}^{n-1}Q_{q^{-},i}\left(2x\right)\Big)\\
 & \stackrel{\textrm{\clap{(\ref{eq:LQ induct})}}}{=}\exp\Big(C\nu Q_{q,0}(x)+C\nu\sum_{i=0}^{n-1}Q_{q,i+1}(x)\Big)\\
 & =\exp\Big(C\nu\sum_{i=0}^{n}Q_{q,i}(x)\Big).
\end{align*}
where $(*)$ is our induction hypothesis. For $x$ dyadic the calculation
above holds identically for both left and right derivatives (for $x=\frac{1}{2}$
only the left derivative), when $Q$ is replaced by appropriate left
or right limits. The lower bound for $(\psi^{-1})'$ is identical.
The case of $x\in[\frac{1}{2},1]$ is identical, with $q^{-}$ and
$\theta^{-}$ replaced by $q^{+}$ and $\theta^{+}$ respectively;
and with (\ref{eq:LQ induct}) replaced by $Q_{q,n}=Q_{q^{+},n-1}(2x-1)$.
This finishes the proof of (\ref{eq:psi-1'<exp sum q}).

Denote for brevity $\psi_{n}\coloneqq\psi_{\theta,n}$ and write (\ref{eq:psi-1'<exp sum q})
as
\[
\left(\psi_{n}^{-1}\right)'(x)\le\exp(C\nu z(x)).
\]
Our assumption (\ref{eq:z cond}) on $z$ shows that, if $\nu$ is
chosen sufficiently small, then $\exp(C\nu z(x))$ is integrable.
Hence $\psi_{n}^{-1}$ are uniformly absolutely continuous, and hence
their limit, $\psi_{\infty}^{-1}$, is absolutely continuous. We next
claim that $(\psi_{n}^{-1})'(x)$ converges for almost all $x$. There
are two ways to see this: the first to repeat the argument that showed
(\ref{eq:psi-1'<exp sum q}) to show that $(\psi_{n}^{-1})'/(\psi_{m}^{-1})'$
is bounded by a partial sum of $Q$-s. The second, to note that $(\psi_{n}^{-1})'$
is a positive martingale with respect to the dyadic filtration $\mathcal{F}_{n}$
and apply the martingale convergence theorem \cite[\S 4.2 (2.11)]{D05}.
By the dominated convergence theorem we get 
\begin{equation}
\lim_{n\to\infty}(\psi_{n}^{-1})'(x)=(\psi_{\infty}^{-1})'(x)\qquad\text{for almost every }x\in[0,1],\label{eq:lim psi n inverse}
\end{equation}
and in particular $(\psi_{\infty}^{-1})'\le\exp(C\nu z(x))$ for almost
every $x\in[0,1]$. This shows the `further' clause of the lemma for
$\psi^{-1}$, as for $\nu$ sufficiently small (\ref{eq:z cond})
gives $||\exp(C\nu z(x))||_{p}<C(p)$. The results for $\psi$ are
then concluded from lemma \ref{lem:integrating is confusing} below.
\end{proof}
\begin{lem}
\label{lem:integrating is confusing}If $\phi$ is an absolutely continuous
homeomorphism of $[0,1]$ with $\frac{1}{\phi'}\in L^{p}$ for some
$p>0$, then $\phi^{-1}$ is also absolutely continuous, with $\int|(\phi^{-1})'|^{p+1}=\int|\phi'|^{-p}$.
\end{lem}
\begin{proof}
Assume without loss of generality that $\phi$ is increasing. We first
show that $\phi^{-1}$ is absolutely continuous. Let therefore $A\subset[0,1]$
be a countable union of intervals. Denote $B=\phi^{-1}A$, so it is
also a countable union of intervals. Then for every $\varepsilon>0$
\[
|A|=\int_{B}\phi'\ge\varepsilon(|B|-|\{x:\phi'(x)<\varepsilon\}|)\ge\varepsilon(|B|-\varepsilon^{p}||\tfrac{1}{\phi'}||_{p}^{p})
\]
where the last inequality follows from Markov's inequality. Choosing
$\varepsilon=(\frac{1}{2}|B|)^{1/p}/||\frac{1}{\phi'}||_{p}$ gives
\[
|\phi^{-1}A|=|B|\le2\left(|A|\cdot||\tfrac{1}{\phi'}||_{p}\right)^{p/(p+1)}
\]
so $\phi^{-1}$ is absolutely continuous. To show the derivative equality
write
\[
\int_{0}^{1}(\phi^{-1})'^{p+1}=\int_{0}^{1}(\phi^{-1})'^{p}\circ\phi=\int_{0}^{1}\frac{1}{\phi'^{p}}
\]
where the first equality is a change of variables.
\end{proof}

\subsection*{Tailoring the homeomorphism family to the function.}

Recall from the discussion on page \pageref{pg:discuss qf} that we
need to tailor the function $q$, which describes a family of homeomorphisms
in this section, and would be used to construct a measure on it in
the next section, to our function $f$ from the statement of theorem
\ref{thm:main}. Here we do that, but first we need some notation.
\begin{defn*}
Recall the Haar functions $\mathbbm{1}_{[0,1]}$, $\mathbbm{1}_{[0,1/2]}-\mathbbm{1}_{[1/2,1]}$
etc. It will be convenient to index them using their support, so for
a dyadic interval $J$, let $h_{J}$ be the function which is $|J|^{-1/2}$
on the left half and $-|J|^{-1/2}$ on the right half (the function
$\mathbbm{1}_{[0,1]}$ will not be associated to any interval, this
will not be a problem).

For a function $f\in L^{2}$ we define $q=q_{f}$, a function on the
dyadic rationals, as follows. Let $I$ be some dyadic interval and
let $d$ be its middle. Then we define
\begin{equation}
q(d)\coloneqq\frac{1}{|I|^{3/2}}\sum_{J\subseteq I}\langle f,h_{J}\rangle^{2}|J|^{1/2}.\label{eq:def q Olv}
\end{equation}
Recall the definition of $z_{q}$ in (\ref{eq:def z}). Applying it
to the $q$ above we get 
\[
z_{q}(x)=\sum_{x\in I\supseteq J}\frac{|J|^{1/2}}{|I|^{3/2}}\langle f,h_{J}\rangle^{2}
\]
(here and below this notation means that the sum is over both $I$
and $J$. Note that $x$ does not need to belong to $J$, only to
$I$).
\end{defn*}
\begin{lem}
\label{lem:z Olv}Let $f$ be a bounded function with $||f||_{\infty}\le1$.
Then $z=z_{q_{f}}$ satisfies $|\{x:z(x)>\lambda\}|\le2e^{-c\lambda}$
for all $\lambda>0$, where $c$ is some universal constant.
\end{lem}
\begin{proof}
Let us first calculate the $L^{1}$ norm of $z=z_{q_{f}}$. By Fubini
\begin{align}
\int_{0}^{1}z(x)\,dx & =\sum_{J\subseteq I}\int_{0}^{1}\frac{|J|^{1/2}}{|I|^{3/2}}\langle f,h_{J}\rangle^{2}\mathbbm{1}\{x\in I\}=\sum_{J\subseteq I}\frac{|J|^{1/2}}{|I|^{1/2}}\langle f,h_{J}\rangle^{2}\nonumber \\
 & \le\frac{\sqrt{2}}{\sqrt{2}-1}\sum_{J}\langle f,h_{J}\rangle^{2}\le\frac{\sqrt{2}}{\sqrt{2}-1},\label{eq:z in L1}
\end{align}
where the last inequality follows since the $h_{J}$ are orthonormal
and $||f||_{2}\le||f||_{\infty}\le1$.

The same calculation shows that $z_{f}$ is a dyadic BMO function.
Indeed, let $U$ be some dyadic interval, let $L_{U}$ be the affine
increasing map taking $[0,1]$ onto $U$. Define 
\[
g_{U}(x)=f(L_{U}(x)).
\]
Then write, for $x\in U$,
\[
z(x)=\sum_{U\subseteq I\supseteq J}\frac{|J|^{1/2}}{|I|^{3/2}}\langle f,h_{J}\rangle^{2}+\sum_{\substack{x\in I\subset U\\
J\subseteq I
}
}\frac{|J|^{1/2}}{|I|^{3/2}}\langle f,h_{J}\rangle^{2}
\]
and note that the first term does not depend on $x$. Hence we may
call it $c_{U}$ and get
\begin{align*}
\int_{U}|z-c_{U}| & =\sum_{J\subseteq I\subset U}\frac{|J|^{1/2}}{|I|^{1/2}}\langle f,h_{J}\rangle^{2}\\
 & \stackrel{(*)}{=}|U|\sum_{J\subseteq I\subseteq[0,1]}\frac{|J|^{1/2}}{|I|^{1/2}}\langle g_{U},h_{J}\rangle^{2}\stackrel{(\dagger)}{\le}|U|\frac{\sqrt{2}}{\sqrt{2}-1},
\end{align*}
where the equality marked by $(*)$ comes from mapping $J$ to $L_{U}^{-1}(J)$,
$I$ to $L_{U}^{-1}(I)$, and noting that the linear change of variable
in the integration together with the fact that $||h_{L_{U}^{-1}(J)}||_{\infty}=\sqrt{|U|}||h_{J}||_{\infty}$
give together a factor of $U$; and where the inequality marked by
$(\dagger)$ comes from applying (\ref{eq:z in L1}) to $g_{U}$ (which
is, of course, also bounded by 1).

Thus $z$ is a dyadic BMO function. The lemma is then concluded by
the John-Nirenberg inequality for dyadic BMO functions \cite[theorem 1d]{SV11}.
\end{proof}
\begin{lem}
\label{lem:abs cont}There exists a universal constant $\eta_{0}$
such that for all $f$ with $||f||_{\infty}\le1$ and all $\theta$
with $|\theta-\frac{1}{2}|\le\eta_{0}q_{f}$, both $\psi_{\theta,\infty}$
and $\psi_{\theta,\infty}^{-1}$ are absolutely continuous.

Further, for any $p<\infty$ there is an $\eta_{1}(p)$ such that
if $|\theta-\frac{1}{2}|\le\eta_{1}(p)q_{f}$ then 
\[
||\psi_{\theta,\infty}'||_{p},||(\psi_{\theta,\infty}^{-1})'||_{p}<C(p).
\]
\end{lem}
\begin{proof}
This follows immediately from lemmas \ref{lem:z to psi bound} and
\ref{lem:z Olv}.
\end{proof}

Before moving forward let us rearrange the formula for $q=q_{f}(\frac{1}{2})$
in a way that will be useful in \S\ref{sec:Random-homeomorphisms}.
We first write
\begin{equation}
q=\sum_{J\subseteq[0,1]}\langle f,h_{J}\rangle|J|^{1/2}\ge\sum_{k=0}^{\infty}2^{-k/2}\sum_{\substack{J\subseteq[0,1]\\
|J|\le2^{-k}
}
}\langle f,h_{J}\rangle^{2}.\label{eq:decompose q Olv}
\end{equation}
(we remark that also $q\le\sqrt{2}/(\sqrt{2}-1)\times$the same sum).
We next recall that the partial sums of the Haar expansion have a
simple form. Let $X_{k}$ be the sum of the first $2^{k}$ terms in
the Haar expansion of $f$. Then for any dyadic $I$ with $|I|=2^{-k}$
and any $x\in I$
\begin{equation}
X_{k}(x)=2^{k}\int_{I}f(y)\,dy.\label{eq:Haar martingale}
\end{equation}
The sum $\sum_{J}\langle f,h_{J}\rangle h_{J}$ is not the partial
Haar expansion of $f$ but rather that of $g\coloneqq f-\int f$,
since the very first Haar function does not correspond to any interval
$J$. Hence
\[
\sum_{|J|\le2^{-k}}\langle f,h_{J}\rangle h_{J}(x)=2^{k}\int_{I}g(y)\,dy.
\]
Applying Parseval we get from (\ref{eq:decompose q Olv})
\begin{equation}
q\ge\sum_{k=0}^{\infty}2^{-k/2}\sum_{i=1}^{2^{k}}2^{k}\bigg(\int_{(i-1)2^{-k}}^{i2^{-k}}g\bigg)^{2}.\label{eq:q Olv to 1/11 lem}
\end{equation}
This is the form of $q$ we will use below.

\begin{rmrks*}The decomposition above for $q$ gives a decomposition
for the corresponding $z$ (recall its definition in (\ref{eq:def z}))
which has a probabilistic intuition behind it. Indeed, write 
\[
z\approx\sum_{k=0}^{\infty}2^{-k/2}z_{k}
\]
where $z_{k}$ is the function given by applying the procedure to
get $z$ from $q$ (which is linear) for one term. Then each $z_{k}$
is the \emph{increasing process} of a martingale, for example $z_{1}$
is the increasing process of the standard Haar martingale $X_{n}$
defined above in (\ref{eq:Haar martingale}). Recall that for any
martingale $X_{n}$, its increasing process is defined as $A_{n}=\sum_{i=1}^{n-1}\mathbb{E}((X_{i+1}-X_{i})^{2}\,|\,X_{1},\dotsc,X_{i})$,
see e.g.\ \cite[\S 4.4]{D05}.

Let us sketch a probabilistic proof of lemma \ref{lem:z Olv}. We
claim that if $X_{n}$ is a bounded martingale, then its increasing
process $A_{n}$ satisfies $\mathbb{P}(A_{\infty}>\lambda)\le2e^{-c\lambda}$.
To see this we apply the Skorokhod embedding theorem to embed the
martingale into Brownian motion, and then $A_{\infty}$ becomes the
time. The result then follows from the fact that the probability of
Brownian motion to stay for time $t$ inside the interval $[-1,1]$
decays exponentially in $t$. Thus each $z_{k}$ has an exponentially
decaying tail, and therefore so does their weighted sum $z$. We skip
any further details.

Our last remark would be useful mostly to readers who have already
examined the proof of lemma \ref{lem:Holder replacement} below. Indeed,
the specific form of $q$ we use is mostly dictated by its use in
that lemma. Its use there leaves two places for flexibility. Lemma
\ref{lem:Holder replacement} could have worked with the sequence
$2^{-k/2}$ replaced with any sequence decaying faster that $1/k^{2}$;
and the term $\sum2^{k}\big(\int g\big)^{2}$, which is just the $L^{2}$
norm of a partial Haar expansion of $g$, could have been replaced
with others norms, e.g.\ with the $L^{1}$ norm. But lemma \ref{lem:z Olv}
does not hold for the $L^{1}$ norm, only for $L^{p}$ norms with
$p\ge2$. Thus our choice of $L^{2}$ norm comes from the need to
have both lemma \ref{lem:z Olv} and \ref{lem:Holder replacement}
hold, each constraining $q_{f}$ from one side.\end{rmrks*}

\section{Random homeomorphisms\label{sec:Random-homeomorphisms}}
\begin{defn*}
\label{def:adimissible eta}We say that a number $\eta>0$ is admissible
if $\eta<\eta_{1}(2)$ where $\eta_{1}$ is from lemma \ref{lem:abs cont}
(in other words if for all $f$ with $||f||_{\infty}\le1$ and all
$|\theta-\frac{1}{2}|\le\eta q_{f}$ we have $||(\psi_{q,\tau,\infty}^{-1})'||_{2}\le C$),
and in addition if the condition $|\theta-\frac{1}{2}|\le\eta q_{f}$
implies that (\ref{eq:only finite for now.}) and (\ref{eq:Holder psi})
 hold with $\delta=\frac{4}{5}$. In other words
\begin{gather}
\left(\tfrac{3}{8}\right)^{m}\le\psi_{\theta,\infty}^{-1}(k2^{-m})-\psi_{\theta,\infty}^{-1}((k-1)2^{-m})\le\left(\tfrac{5}{8}\right)^{m}\label{eq:38 58-1}\\
(\tfrac{1}{4}(y-x))^{5/4}\le\psi_{\theta,\infty}^{-1}(y)-\psi_{\theta,\infty}^{-1}(x)\le(4(y-x))^{4/5}\label{eq:psi-1 45 54}
\end{gather}
for all values of $m$.
\end{defn*}
The rest of the proof holds for any admissible $\eta$, so let us
fix one such $\eta$ (possibly depending on $p$, if we are proving
the `further' clause of theorem \ref{thm:main}) and remove it from
the notation. Further, we allow arbitrary constants like $c$ and
$C$ to depend on it. For a function $\tau$ from the dyadic rationals
into $[-1,1]$ we denote
\[
\psi_{f,\tau,n}\coloneqq\psi_{\frac{1}{2}+\eta q_{f}\tau,n}\qquad\psi_{f,\tau,\infty}\coloneqq\psi_{\frac{1}{2}+\eta q_{f}\tau,\infty}.
\]
For convenience, let us note at this point the locality formulas for
$\psi_{f,\tau,n}$. Define $f^{\pm}$ and $\tau^{\pm}$ as in (\ref{eq:theta pm})
and recall that $L_{I}$ is the affine increasing map taking $[0,1]$
onto $I$. Then
\begin{align}
\psi_{f,\tau,\infty}(x) & =\begin{cases}
\frac{1}{2}\psi_{f^{-},\tau^{-},\infty}\left(\frac{x}{A}\right) & x\le A\\
\frac{1}{2}+\frac{1}{2}\psi_{f^{+},\tau^{+},\infty}\left(\frac{x-A}{1-A}\right) & x>A
\end{cases}\label{eq:psi random pm}\\
\left.\psi_{f,\tau,\infty}\right|_{\psi_{f,\tau,\infty}^{-1}(I)} & =L_{I}\circ\psi_{f\circ L_{I},\tau\circ L_{I},\infty}\circ L_{\psi_{f,\tau,\infty}^{-1}(I)}^{-1},\label{eq:psi random LI}
\end{align}
where, as usual, $A=\theta(\frac{1}{2})=\frac{1}{2}+\eta q_{f}(\frac{1}{2})\tau(\frac{1}{2})$.
The formula is a direct consequence of (\ref{eq:psi infinity recurs}),
(\ref{eq:psi local}) and the facts that $(q_{f})^{-}=q_{f-}$ and
$q_{f}\circ L_{I}=q_{f\circ L_{I}}$, both of which are easy to verify.

Most importantly, from now on we take $\tau$ random, i.i.d., uniformly
distributed in $[-1,1]$. All $\mathbb{E}$ and $\mathbb{P}$ signs
in this section will refer to this measure on $\tau$.
\begin{lem}
\label{lem:exp psi inverse prime 1}For every $f$ and $n$ (including
$n=\infty$),
\[
\mathbb{E}\left(\psi_{f,\tau,n}^{-1}\right)'(y)=1
\]
for every $y\in[0,1]$ if $n$ is finite (for dyadic $y$ we mean
one-sided derivatives) and for almost every $y$ if $n=\infty$.
\end{lem}
\begin{proof}
For $n$ finite this follows from the definition of $\psi$ (\ref{eq:psi by triangles})
because $\mathbb{E}(\tau(d))=0$ for each $d$ (so $\mathbb{E}(\theta(d))=\frac{1}{2})$,
and further, because if $\rk(d)=n$ then $\tau(d)$ and $\psi_{f,\tau,n-1}^{-1}$
are independent. This last claim follows since $\psi_{n-1}^{-1}$
depends only on $\tau(d')$ for $d'$ with $\rk(d')<n$, in particular
$d'\ne d$, and hence the corresponding values of $\tau$ are independent.

The case of $n=\infty$ follows by taking a limit, which is allowed
since $(\psi_{n}^{-1})'\to(\psi_{\infty}^{-1})'$ for almost every
$x$ (\ref{eq:lim psi n inverse}), for every $\tau$. By Fubini,
for almost every $x$ we have $(\psi_{n}^{-1})'\to(\psi_{\infty}^{-1})'$
almost surely (in $\tau$). Exchanging the limit and the expectation
is allowed since $\left(\psi_{n}^{-1}\right)'(x)$ is bounded by a
bound depending only on $f$, by (\ref{eq:psi-1'<exp sum q}), and
finite for almost every $x$, by lemma \ref{lem:z Olv}, so the bounded
convergence theorem applies.
\end{proof}

\begin{lem}
\label{lem:coupling min density}Let $X_{1}$ and $X_{2}$ be two
random variables on $[-1,1]$ with densities $\rho_{1}$ and $\rho_{2}$
respectively. Let 
\[
p=\int_{-1}^{1}\min\{\rho_{1},\rho_{2}\}.
\]
Then there exists a variable $Q$ (`a coupling of $X_{1}$ and $X_{2}$')
taking values in $[-1,1]^{2}$ such that
\begin{enumerate}
\item $Q_{i}$ has the same distribution as $X_{i}$.
\item $\mathbb{P}(Q_{1}=Q_{2})\ge p$.
\end{enumerate}
\end{lem}
(in fact, $\mathbb{P}(Q_{1}=Q_{2})=p$, but we will have no use for
this fact).
\begin{proof}
If $p=0$ then we can take $Q$ to be two independent copies of $X_{1}$
and $X_{2}$ so assume this is not the case. Let $Z$ be a random
variable with density $\frac{1}{p}\min\{\rho_{1},\rho_{2}\}$. If
$p\ne1$ then let $Y_{i}$ be a random variable with density $\frac{1}{1-p}(\rho_{i}-\min\{\rho_{1},\rho_{2}\})$.
Let $Z$, $Y_{1}$ and $Y_{2}$ be all independent. Now throw an independent
coin which succeeds with probability $p$. If the coin succeeds, let
$Q=(Z,Z)$, if not, let $Q=(Y_{1},Y_{2})$. Clause 2 is now obvious,
and clause 1 is not difficult either, because the density of $Q_{i}$
is $p$ times the density of $Z$ plus $1-p$ times the density of
$Y_{i}$ i.e.
\[
p\cdot\frac{1}{p}\min\{\rho_{1},\rho_{2}\}+(1-p)\cdot\frac{1}{1-p}(\rho_{i}-\min\{\rho_{1},\rho_{2}\})=\rho_{i},
\]
as needed.
\end{proof}
\begin{lem}
\label{lem:ttp coupling}For every $0<x_{1}<x_{2}<1$ and every $\lambda\in(0,\frac{1}{4})$
there exists a random variable $T=(T_{1},T_{2})$ taking values in
$[-1,1]^{2}$ such that
\begin{enumerate}
\item $T_{1}$ and $T_{2}$ are both uniform in $[-1,1]$.
\item If we denote 
\[
y_{i}\coloneqq\begin{cases}
\frac{x_{i}}{A_{i}} & x_{i}\le A_{i}\\
\frac{x_{i}-A_{i}}{1-A_{i}} & x_{i}>A_{i}
\end{cases}\qquad A_{i}\coloneqq\frac{1}{2}+\lambda T_{i}
\]
then 
\[
\mathbb{P}(y_{1}=y_{2},\mathbbm{1}_{x_{1}\le A_{1}}=\mathbbm{1}_{x_{2}\le A_{2}})\ge1-C\frac{x_{2}-x_{1}}{\lambda\min\{x_{2},1-x_{1}\}}.
\]
\end{enumerate}
\end{lem}
\begin{proof}
Define auxiliary variables $z_{i}$ by 
\[
z_{i}=\begin{cases}
\frac{x_{i}}{2A} & x_{i}\le A\\
\frac{1}{2}+\frac{x_{i}-A}{2(1-A)} & x_{i}>A
\end{cases}
\]
where $A$ is a random variable uniform in $[\frac{1}{2}-\lambda,\frac{1}{2}+\lambda${]}.
Let $\rho_{i}$ be the density of $z_{i}$. A simple calculation shows
that
\begin{gather*}
\rho_{i}(t)=\mathbbm{1}_{I_{i}}(t)\frac{x_{i}}{4\lambda t^{2}}+\mathbbm{1}_{J_{i}}(t)\frac{1-x_{i}}{4\lambda(1-t)^{2}}\\
I_{i}\coloneqq\left[\frac{x_{i}}{1+2\lambda},\frac{x_{i}}{1-2\lambda}\right]\cap\left[0,\frac{1}{2}\right]\quad J_{i}\coloneqq\left[1-\frac{1-x_{i}}{1-2\lambda},1-\frac{1-x_{i}}{1+2\lambda}\right]\cap\left[\frac{1}{2},1\right]
\end{gather*}
and hence
\[
\int\min\{\rho_{1},\rho_{2}\}\ge1-C\frac{x_{2}-x_{1}}{\lambda\min\{x_{2},1-x_{1}\}}\eqqcolon p.
\]
Now couple $z_{1}$ and $z_{2}$ using lemma \ref{lem:coupling min density}
(call the coupling $Q$) and the coupling succeeds (i.e. $Q_{1}=Q_{2})$
with probability at least $p$. Now define
\[
T_{i}=\begin{cases}
\frac{1}{\lambda}\left(\frac{x_{i}}{2Q_{i}}-\frac{1}{2}\right) & Q_{i}\le\frac{1}{2}\\
\frac{1}{\lambda}\left(\frac{1}{2}-\frac{1-x_{i}}{2-2Q_{i}}\right) & Q_{i}>\frac{1}{2}
\end{cases}
\]
and the properties of $T$ follow from the way we constructed it.
\end{proof}
\begin{lem}
\label{lem:Holder replacement}For every continuous $||f||_{\infty}\le1$,
every interval $[r,s]\subseteq[0,1]$ and every $n$ (not necessarily
integer),
\[
\bigg|\int_{r}^{s}\mathbb{E}(f(\psi_{f,\tau,\infty}(x)))\cdot e(nx)\,dx\bigg|\le Cn^{-1/11}.
\]
\end{lem}
\begin{proof}
We divide into two different cases according to whether $q\coloneqq q_{f}(\frac{1}{2})>n^{-9/10}$
or not ($q_{f}$ from (\ref{eq:def q Olv})). Denote for brevity
$\psi=\psi_{f,\tau,\infty}$.

We start with the case of $q>n^{-9/10}$. In this case we claim that
for every $x_{1}<x_{2}$ 
\begin{equation}
\left|\mathbb{E}(f(\psi(x_{1}))-\mathbb{E}(f(\psi(x_{2})))\right|\le C\frac{x_{2}-x_{1}}{q\min\{x_{2},1-x_{1}\}}.\label{eq:Holder away}
\end{equation}
To see (\ref{eq:Holder away}) couple $\psi(x_{1})$ and $\psi(x_{2})$
as follows. We define $\tau_{1}$ and $\tau_{2}$ two random (dependent)
variables, each of which has the same distribution as $\tau$ i.e.\ a
uniform function from the dyadic rationals into $[-1,1]$. The definition
of the $\tau_{i}$ is as follows. For all dyadic $p\ne\frac{1}{2}$
we take $\tau_{1}(p)=\tau_{2}(p)$ (and of course independent for
different $p$ and uniformly distributed). For $p=\frac{1}{2}$ we
use the variable given by lemma \ref{lem:ttp coupling}, namely, we
use lemma \ref{lem:ttp coupling} with $\lambda_{\textrm{lemma \ref{lem:ttp coupling}}}=\eta q$
and then let $\tau_{i}(\frac{1}{2})=T_{i}$. We now note that, by
(\ref{eq:psi random pm})
\[
\psi_{f,\tau_{i},\infty}(x_{i})=\begin{cases}
\frac{1}{2}\psi_{f^{-},\tau_{i}^{-},\infty}(\frac{x_{i}}{A_{i}}) & x_{i}\le A_{i}\\
\frac{1}{2}+\frac{1}{2}\psi_{f^{+},\tau_{i}^{+},\infty}(\frac{x_{i}-A_{i}}{1-A_{i}}) & \text{otherwise.}
\end{cases}
\]
Note that $\tau^{\pm}$ do not depend on $i$ (here is where we used
that $\tau_{1}(p)=\tau_{2}(p)$ for all $p\ne\frac{1}{2}$) and of
course neither do $f^{\pm}$. Hence lemma \ref{lem:ttp coupling}
gives
\begin{multline*}
\mathbb{P}(\psi_{f,\tau_{1},\infty}(x_{1})=\psi_{f,\tau_{2},\infty}(x_{2}))\ge\mathbb{P}(y_{1}=y_{2},\mathbbm{1}_{x_{1}\le A_{1}}=\mathbbm{1}_{x_{2}\le A_{2}})\\
\ge1-C\frac{x_{2}-x_{1}}{q\min\{x_{2},1-x_{1}\}},
\end{multline*}
where the $y_{i}$ are from lemma \ref{lem:ttp coupling} (the first
inequality is in fact an equality, but we will not need this fact).
Denote the event on the left hand side by $G$. Then
\begin{align*}
\lefteqn{{\left|\mathbb{E}(f(\psi(x_{1})))-\mathbb{E}(f(\psi(x_{2})))\right|=\left|\mathbb{E}(f(\psi_{f,\tau_{1},\infty}(x_{1})))-\mathbb{E}(f(\psi_{f,\tau_{2},\infty}(x_{2})))\right|}}\qquad\\
 & \stackrel{(*)}{=}\left|\mathbb{E}(\mathbbm{1}_{\neg G}\cdot(f(\psi_{f,\tau_{1},\infty}(x_{1})-f(\psi_{f,\tau_{2},\infty}(x_{2})))\right|\qquad\qquad\qquad\qquad\\
 & \le2\mathbb{P}(\neg G)\le C\frac{x_{2}-x_{1}}{q\min\{x_{2},1-x_{1}\}}
\end{align*}
where the equality marked by $(*)$ follows because $\mathbb{E}(\mathbbm{1}_{G}\cdot\text{the same})=0$,
since under $G$ we have $\psi_{f,\tau_{1},\infty}(x_{1})=\psi_{f,\tau_{2},\infty}(x_{2})$.
This shows (\ref{eq:Holder away}).

To conclude from (\ref{eq:Holder away}) an estimate for the integral
we write 
\[
\int_{r}^{s}=\sum_{k=0}^{\lfloor(s-r)n\rfloor-1}\int_{r+k/n}^{r+(k+1)/n}+\int_{r+\lfloor(s-r)n\rfloor/n}^{s}.
\]
The last integral is over an interval less than $1/n$, so it can
contribute no more than $1/n$. For the main term we write for each
$k$,
\begin{multline*}
I_{k}\coloneqq\int_{r+k/n}^{r+(k+1)/n}\mathbb{E}(f(\psi(x)))\cdot e(nx)\,dx\\
=\int_{r+k/n}^{r+(k+1/2)/n}(\mathbb{E}(f(\psi(x)))-\mathbb{E}(f(\psi(x+\tfrac{1}{2n}))))\cdot e(nx).
\end{multline*}
Hence
\begin{align*}
|I_{k}| & \le\int_{r+k/n}^{r+(k+1/2)/n}|\mathbb{E}(f(\psi(x)))-\mathbb{E}(f(\psi(x+\tfrac{1}{2n})))|\\
 & \stackrel{\textrm{\clap{(\ref{eq:Holder away})}}}{\le}C\int_{r+k/n}^{r+(k+1/2)/n}\frac{1/2n}{q\min\{x+1/2n,1-x\}}\,dx\\
 & \le\frac{C}{n^{2}q\min\{r+(k+1/2)/n,1-(r+(k+1/2)/n)\}}.
\end{align*}
Summing over all $k$ gives
\[
\sum|I_{k}|\le\frac{C}{n^{2}q}\sum_{i=0}^{n/2}\frac{1}{1/2n+i/n}\le\frac{C\log n}{nq}
\]
and because of our assumption that $q>n^{-9/10}$ (which we have not
used so far!) we get a bound of $Cn^{-1/10}\log n$. This finishes
the case $q>n^{-9/10}$.

For the case of $q\le n^{-9/10}$ our first order of things is to
rearrange the question slightly. Recall that we wish to estimate $\int\mathbb{E}(f\circ\psi)\cdot e(nx)$.
We first let $g=f-\int f$ and write
\[
\int_{r}^{s}\mathbb{E}(f\circ\psi)\cdot e(nx)\,dx=\int_{r}^{s}\mathbb{E}(g\circ\psi)\cdot e(nx)\,dx+\bigg(\int f\bigg)\cdot\int_{r}^{s}e(nx)\,dx
\]
and the second term is smaller, in absolute value, than $2/n$. For
the first term, apply Fubini and a change of variables and get
\begin{align}
\int_{r}^{s}\mathbb{E}(g(\psi(x)))\cdot e(nx)\,dx & =\mathbb{E}\int_{r}^{s}g(\psi(x))\cdot e(nx)\,dx\nonumber \\
 & =\mathbb{E}\int_{\psi(r)}^{\psi(s)}g(y)e(n\psi^{-1}(y))(\psi^{-1})'(y)\,dy.\label{eq:first psi^-1'}
\end{align}
Recall next the estimate (\ref{eq:q Olv to 1/11 lem}) for $q=q_{f}(\frac{1}{2})$.
Using it with the condition $q\le n^{-9/10}$ gives 
\[
\sum_{i=1}^{2^{k}}2^{k}\bigg(\int_{(i-1)2^{-k}}^{i2^{-k}}g\bigg)^{2}\stackrel{\textrm{(\ref{eq:q Olv to 1/11 lem})}}{\le}2^{k/2}q\le2^{k/2}n^{-9/10}\qquad\forall k.
\]
Let $k$ be the smallest integer such that $2^{k}>n^{7/5}.$ Then
\begin{equation}
\sum_{i=1}^{2^{k}}2^{k}\bigg(\int_{(i-1)2^{-k}}^{i2^{-k}}g\bigg)^{2}\le Cn^{-1/5}.\label{eq:def zeta repeat}
\end{equation}
For each $i$ write
\begin{align}
\lefteqn{{\int_{i2^{-k}}^{(i+1)2^{-k}}g(y)e(n\psi^{-1}(y))(\psi^{-1})'(y)\,dy}}\qquad\nonumber \\
 & =\int_{i2^{-k}}^{(i+1)2^{-k}}g(y)\big(e(n\psi^{-1}(y))-e(n\psi^{-1}(i2^{-k}))\big)(\psi^{-1})'(y)\,dy\nonumber \\
 & \qquad+e(n\psi^{-1}(i2^{-k}))\int_{i2^{-k}}^{(i+1)2^{-k}}g(y)(\psi^{-1})'(y)\,dy\nonumber \\
 & \eqqcolon\text{I}_{i}+\text{II}_{i}\label{eq:int g e psi-1 I II}
\end{align}
We start with the estimate of $\text{II}_{i}$. We wish to condition
on $\psi^{-1}(i2^{-k})$ and $\psi^{-1}((i+1)2^{-k})$ so denote the
$\sigma$-field of these two variables by $\mathcal{B}_{i}$. Then
\begin{align}
\lefteqn{\mathbb{E}\bigg(\int_{i2^{-k}}^{(i+1)2^{-k}}g(y)(\psi^{-1})'(y)\,dy\,\bigg|\,\mathcal{B}_{i}\bigg)}\qquad\nonumber \\
 & =\int_{i2^{-k}}^{(i+1)2^{-k}}g(y)\mathbb{E}\left((\psi^{-1})'(y)\,|\,\mathcal{B}_{i}\right)\,dy\label{eq:stam Fubini}
\end{align}
by Fubini. We have reached a simple, but crucial point in the argument.
We note that $\psi^{-1}$ conditioned on $\mathcal{B}_{i}$ is a different
version of $\psi^{-1}$. Precisely, let $L$ be the linear increasing
map taking $[0,1]$ onto $[i2^{-k},(i+1)2^{-k}]$. Let $h(x)=f(L(x))$.
Let $\sigma(d)=\tau(L(d))$. Then by (\ref{eq:psi random LI}), 
for every $x\in[i2^{-k},(i+1)2^{-k}]$,
\begin{align*}
\psi_{f,\tau,\infty}^{-1}(x) & =\psi_{f,\tau,\infty}^{-1}(i2^{-k})+b_{i}\psi_{h,\sigma,\infty}^{-1}(2^{k}x-i).\\
b_{i} & \!\coloneqq\psi_{f,\tau,\infty}^{-1}((i+1)2^{-k})-\psi_{f,\tau,\infty}^{-1}(i2^{-k})
\end{align*}
This means that
\[
(\psi_{f,\tau,\infty}^{-1})'(x)=2^{k}b_{i}(\psi_{h,\sigma,\infty}^{-1})'\left(2^{k}x-i\right).
\]
We now use lemma \ref{lem:exp psi inverse prime 1} which states that
$\mathbb{E}((\psi_{h,\sigma,\infty}^{-1})'(y))=1$ for almost all
$y$. This shows that
\[
\mathbb{E}\left((\psi_{f,\tau,\infty}^{-1})'(x)\,|\,\mathcal{B}_{i}\right)=2^{k}b_{i}
\]
for almost all $x$. With (\ref{eq:stam Fubini}) we get
\[
\mathbb{E}\bigg(\int_{i2^{-k}}^{(i+1)2^{-k}}g(y)(\psi^{-1})'(y)\,dy\,\bigg|\,\mathcal{B}_{i}\bigg)=2^{k}b_{i}\int_{i2^{-k}}^{(i+1)2^{-k}}g(y)\,dy
\]
or 
\begin{equation}
\left|\mathbb{E}(\textrm{II}_{i}\,|\,\mathcal{B}_{i})\right|=2^{k}b_{i}\bigg|\int_{i2^{-k}}^{(i+1)2^{-k}}g(y)\,dy\bigg|.\label{eq:g psi 2}
\end{equation}
Let $E$ be the set of indices $i$ such that $[i2^{-k},(i+1)2^{-k}]\subseteq[\psi(r),\psi(s)]$.
Sum the right hand side over all $i\in E$ and use Cauchy-Schwarz
to get
\begin{align}
\lefteqn{{\sum_{i\in E}2^{k}b_{i}\bigg|\int_{i2^{-k}}^{(i+1)2^{-k}}g(y)\,dy\bigg|\le}}\qquad\nonumber \\
 & \bigg(\sum_{i\in E}2^{k}b_{i}^{2}\bigg)^{1/2}\bigg(\sum_{i\in E}2^{k}\bigg(\int_{i2^{-k}}^{(i+1)2^{-k}}g(y)\,dy\bigg)^{2}\bigg)^{1/2}\label{eq:psi int g CS}
\end{align}
and the second term is bounded by (\ref{eq:def zeta repeat}) by $Cn^{-1/10}$.
As for the first term we may write
\begin{align}
\sum_{i\in E}2^{k}b_{i}^{2} & =\sum_{i\in E}2^{k}\bigg(\int_{i2^{-k}}^{(i+1)2^{-k}}(\psi_{f,\tau,\infty}^{-1})'\bigg)^{2}\le\sum_{i\in E}\int_{i2^{-k}}^{(i+1)2^{-k}}\left((\psi_{f,\tau,\infty}^{-1})'\right)^{2}\nonumber \\
 & \le\int_{0}^{1}\left((\psi_{f,\tau,\infty}^{-1})'\right)^{2}.\label{eq:sum psi int psi}
\end{align}
Hence
\begin{align*}
\lefteqn{{\bigg|\mathbb{E}\sum_{i\in E}\text{II}_{i}\bigg|\le\mathbb{E}\sum_{i\in E}|\mathbb{E}(\textrm{II}_{i}\,|\,\mathcal{B}_{i})|}}\qquad\\
\textrm{by (\ref{eq:g psi 2})}\qquad & =\mathbb{E}\bigg(\sum_{i\in E}2^{k}b_{i}\bigg|\int_{i2^{-k}}^{(i+1)2^{-k}}g(y)\,dy\bigg|\bigg)\\
\textrm{by (\ref{eq:psi int g CS})-(\ref{eq:sum psi int psi})}\qquad & \le\mathbb{E}Cn^{-1/10}||(\psi_{f,\tau,\infty}^{-1})'||_{2}\le Cn^{-1/10}.
\end{align*}
(we need to condition on $\mathcal{B}_{i}$ in the first inequality
because $E$ is itself random). In the last inequality we used that
$\eta$ is admissible, recall the definition of admissibility in the
beginning of \S \ref{sec:Random-homeomorphisms}. This terminates
the estimate of the second, more interesting part of (\ref{eq:int g e psi-1 I II}).

To estimate $\textrm{I}_{i}$ we note that from $2^{k}\ge n^{7/5}$
and (\ref{eq:psi-1 45 54}) we have $\psi^{-1}(y)-\psi^{-1}(i2^{-k})<Cn^{-28/25}$
for every $y\in[i2^{-k},(i+1)2^{-k}]$ and hence
\[
|e(n\psi^{-1}(y))-e(n\psi^{-1}(i2^{-k}))|\le Cn^{-3/25}.
\]
we get
\[
\textrm{I}_{i}\le Cn^{-3/25}\int_{i2^{-k}}^{(i+1)2^{-k}}(\psi^{-1})'(y)\,dy=Cn^{-3/25}(\psi^{-1}((i+1)2^{-k})-\psi^{-1}(i2^{-k}))
\]
and hence 
\[
\sum_{i\in E}\textrm{I}_{i}\le Cn^{-3/25}.
\]
With the estimate on $\textrm{II}_{i}$ above we get
\begin{multline*}
\bigg|\mathbb{E}\sum_{i\in E}\int_{i2^{-k}}^{(i+1)2^{-k}}g(y)e(n\psi^{-1}(y))(\psi^{-1})'(y)\,dy\bigg|\\
=\bigg|\mathbb{E}\sum_{i\in E}\textrm{I}_{i}+\mathbb{E}\sum_{i\in E}\textrm{II}_{i}\bigg|\le Cn^{-1/10}.
\end{multline*}
This is almost what we need, but what we need precisely is the integral
from $\psi(r)$ to $\psi(s)$. The difference between it and $\sum_{i\in E}\int_{i2^{-k}}^{(i+1)2^{-k}}$
is two short intervals, $[\psi(r),i_{1}2^{-k}]$ and $[i_{2}2^{-k},\psi(s)]$
for some $i_{1}$ and $i_{2}$. The integral on these is much simpler
to estimate as they are short. For example, we may estimate
\begin{multline*}
\mathbb{E}\bigg|\int_{\psi(r)}^{i_{1}2^{-k}}g(y)e(n\psi^{-1}(y))(\psi^{-1})'(y)\,dy\bigg|\le2\mathbb{E}\int_{\psi(r)}^{i_{1}2^{-k}}(\psi^{-1})'\\
\le2(\mathbb{E}(i_{1}2^{-k}-\psi(r)))^{1/2}\bigg(\mathbb{E}\int\left((\psi^{-1})'\right)^{2}\bigg)^{1/2}\le C\cdot2^{-k/2}\le Cn^{-7/10}
\end{multline*}
and similarly for the other interval. We get
\[
\bigg|\mathbb{E}\bigg(\int_{\psi(r)}^{\psi(s)}g(y)e(n\psi^{-1}(y))(\psi^{-1})'(y)\,dy\bigg)\bigg|\le Cn^{-1/10}.
\]
Recalling (\ref{eq:first psi^-1'}) from the beginning of the proof
of the small $q$ case we get
\[
\bigg|\int_{r}^{s}\mathbb{E}(g(\psi(x)))\cdot e(nx)\,dx\bigg|\le Cn^{-1/10}
\]
and as explained just before (\ref{eq:first psi^-1'}), this gives
the same result for $f$. The lemma is thus proved.
\end{proof}
\begin{lem}
\label{lem:j}Let $J_{1},J_{2}\subset[-1,1]$ be intervals and let
$\varepsilon\coloneqq\max\{|y_{1}-y_{2}|:y_{i}\in J_{i}\}$. Let
$[r,s]\subseteq[0,1]$ be an interval and let $n\in\mathbb{R}$ and
define
\[
F_{i}\coloneqq\int_{r}^{s}\mathbb{E}(f(\psi_{f,\tau,\infty}(x))\,|\,\tau(\tfrac{1}{2})\in J_{i})\cdot e(nx)\,dx.
\]
Then
\[
|F_{1}-F_{2}|\le C\varepsilon^{1/22},
\]
uniformly in $n$.
\end{lem}
\begin{proof}
We keep the notation $q=q_{f}(\frac{1}{2})$. Denote  
\[
\alpha\coloneqq\min\{\tfrac{1}{2}+\eta qx:x\in J_{1}\cup J_{2}\}\qquad\beta\coloneqq\max\{\tfrac{1}{2}+\eta qx:x\in J_{1}\cup J_{2}\}
\]
so that $\beta-\alpha\le C\varepsilon$ and then
\begin{align*}
F_{i}^{-} & \coloneqq\int_{[r,s]\cap[0,\alpha]}\mathbb{E}(f(\psi_{f,\tau,\infty}(x))\,|\,\tau(\tfrac{1}{2})\in J_{i})\cdot e(nx)\,dx\\
F_{i}^{+} & \coloneqq\int_{[r,s]\cap[\beta,1]}\mathbb{E}(f(\psi_{f,\tau,\infty}(x))\,|\,\tau(\tfrac{1}{2})\in J_{i})\cdot e(nx)\,dx\\
F_{i}^{0} & \coloneqq F_{i}-(F_{i}^{-}+F_{i}^{+})
\end{align*}
and note that $|F_{i}^{0}|\le C\varepsilon$ as it can be written
as an integral over an interval of length at most $C\varepsilon$.
Hence we need to estimate $|F_{1}^{\pm}-F_{2}^{\pm}|$. The proofs
of the $+$ and $-$ cases are identical, so for brevity, we will
do the calculations for $F^{-}$.

Let us first dispense with an easy case. Assume that $n\ge\varepsilon^{-1/2}$.
In this case, fix some $y\in J_{1}\cup J_{2}$ and condition on $\tau(\frac{1}{2})=y$.
Denote
\[
F^{-}(y)\coloneqq\int_{[r,s]\cap[0,\alpha]}\mathbb{E}(f(\psi_{f,\tau,\infty}(x))\,|\,\tau(\tfrac{1}{2})=y)\cdot e(nx)\,dx
\]
under the conditioning, $\psi_{f,\tau,\infty}(x)=\frac{1}{2}\psi_{f^{-},\tau^{-},\infty}(x/A)$,
$A\coloneqq\frac{1}{2}+\eta qy$, so
\[
F^{-}(y)=\int_{[r,s]\cap[0,\alpha]}\mathbb{E}(f^{-}(\psi_{f^{-},\tau^{-},\infty}(x/A)))\cdot e(nx)\,dx
\]
We change variables and get 
\begin{gather}
F^{-}(y)=\int_{I(y)}\mathbb{E}(f^{-}(\psi_{f^{-},\tau^{-},\infty}(x)))\cdot e(nxA)\cdot A\,dx,\label{eq:F- I(y)}\\
I(y)\coloneqq\left[\frac{r}{A},\frac{s}{A}\right]\cap\left[0,\frac{\alpha}{A}\right].\nonumber 
\end{gather}
We use lemma \ref{lem:Holder replacement} with $f_{\textrm{lemma \ref{lem:Holder replacement}}}=f^{-}$,
$n_{\textrm{lemma \ref{lem:Holder replacement}}}=nA$ and $[r,s]_{\textrm{lemma \ref{lem:Holder replacement}}}=I(y)$
and get 
\[
|F^{-}(y)|\le A\cdot C(nA)^{-1/11}\le C\varepsilon^{1/22}.
\]
Integrating over $y$ in either $J_{i}$ gives that both $F_{i}$
satisfy $|F_{i}|\le C\varepsilon^{1/22}$ and we are done.

Assume therefore that $n<\varepsilon^{-1/2}$. Let $y\in J_{1}\cup J_{2}$
and consider $\psi_{f,\tau,\infty}$ conditioned on $\tau(\frac{1}{2})=y$.
We keep the notations $A$, $F^{-}(y)$ and $I(y)$ above, and write
$\psi=\psi_{f^{-},\tau^{-},\infty}$ for brevity. We apply (\ref{eq:F- I(y)})
to two different $y_{i}$ and subtract, getting
\begin{gather*}
F^{-}(y_{1})-F^{-}(y_{2})=\textrm{I}+\textrm{II}+\textrm{III}\\
\begin{aligned}\textrm{I} & \coloneqq\int_{I(y_{1})\cap I(y_{2})}\mathbb{E}(f^{-}(\psi(x))\cdot(e(nxA_{1})A_{1}-e(nxA_{2})A_{2})\,dx\\
\textrm{II} & \coloneqq\int_{I(y_{1})\setminus I(y_{2})}\mathbb{E}(f^{-}(\psi(x))\cdot e(nxA_{1})A_{1}\,dx\\
\textrm{III} & \coloneqq-\int_{I(y_{2})\setminus I(y_{1})}\mathbb{E}(f^{-}(\psi(x))\cdot e(nxA_{2})A_{2}\,dx.
\end{aligned}
\end{gather*}
and again II and III are integrals over intervals of length at most
$C\varepsilon$, so may be ignored. As for I, we write
\[
|e(nxA_{1}x)-e(nxA_{2})|\le2\pi nx|A_{1}-A_{2}|x\le Cn\varepsilon
\]
so 
\[
\left|e(nxA_{1})A_{1}-e(nxA_{2})A_{2}\right|\le Cn\varepsilon\cdot A_{1}+C\varepsilon\le Cn\varepsilon.
\]
which we plug into the integral in I to get $\textrm{I}\le Cn\varepsilon$.
We get
\[
|F^{-}(y_{1})-F^{-}(y_{2})|\le Cn\varepsilon+C\varepsilon\le C\varepsilon^{1/2}
\]
by our assumption that $n<\varepsilon^{-1/2}$. Integrating over $y_{1}$
and $y_{2}$ gives
\[
|F_{1}^{-}-F_{2}^{-}|\le C\varepsilon^{1/2}.
\]
Since we covered both small and large $n$, we get
\[
|F_{1}^{-}-F_{2}^{-}|\le C\varepsilon^{1/22}.
\]
The estimate for $|F_{1}^{+}-F_{2}^{+}|$ is identical, and the lemma
is proved.
\end{proof}

\section{\label{sec:Reducing-randomness}Reducing randomness}

We have spent many pages on construction of a random homeomorphism
$\psi$ such that the \emph{expectation} of $f\circ\psi$ has some
good properties. In this last section we are going to remove the randomness
step by step, always controlling the expectation, until we are finally
left with a single $\psi$ such that $f\circ\psi$ has the same good
properties.

Let $I$ be a map giving for each dyadic rational $d=k2^{-n}\in(0,1)$
an interval $I(d)\subseteq[-1,1]$ (possibly degenerate). For each
$I$ we define $\phi_{I}=\phi_{I,f,\eta}$ to be our homeomorphism
$\psi_{f,\tau,\infty}$ with $\tau(d)$ taken uniformly in $I(d)$
for all $d$. We call such $I$ an RH-restrictor (RH standing for
Random Homeomorphism). Sometime we will denote $\tau_{I}$ for $\tau$
which has this distribution.

From now on, when we say `condition $\psi$ on $\psi^{-1}(\frac{1}{2})=y$'
or `condition $\psi$ on $\psi(y)=\frac{1}{2}$' we will actually
mean $\psi_{f,\tau,\infty}$ with $\tau(\frac{1}{2})=(y-\frac{1}{2})/\eta q_{f}(\frac{1}{2})$
and the other $\tau(d)$ uniform in $[-1,1]$ (this is, of course,
a version of the usual conditional probability, but it is defined
for all $y$ and not just for almost all $y$, which is convenient).
Another version of the same thinking is the following
\begin{lem}
\label{lem:I dont like induction, to be honest about it}Let $\mathcal{U}=(U_{i})$
be a finite partition of $[0,1]$ into dyadic intervals (i.e.\ $\bigcup\mathcal{U}=[0,1]$
and different $U_{i}\in\mathcal{U}$ are disjoint except possibly
their endpoints). Let $I$ be an RH-restrictor which is degenerate
for any $d$ in the boundary of any $U\in\mathcal{U}$. Then $\phi_{I}^{-1}(U)$
is deterministic for any $U\in\mathcal{U}$ (i.e.\ it is dependent
on $f$ but not random).
\end{lem}
\begin{proof}
We go by induction on the number of intervals in the partition. Assume
the claim has been proved for all partitions with less than $n$ intervals,
and let $\mathcal{U}$ be with $|\mathcal{U}|=n$. Let $U\in\mathcal{U}$
be an interval with smallest size. Let $V$ be $U$'s father i.e.\ $U\subset V$
and $|V|=2|U|.$ Then $V\setminus U$ must also be in $\mathcal{U}$.
Replacing $U$ and $V\setminus U$ with $V$ gives a new partition
$\mathcal{U}'$ with $|\mathcal{U}'|<n$ hence the induction assumption
gives that $\phi_{I}(W)$ is deterministic for any $W\in\mathcal{U}'$
and in particular for all $W\in\mathcal{U}$ other than $U$ and $V\setminus U$.
Next, let $d_{0}$, $d_{1}$ and $d_{2}$ be the beginning, centre
and end of $V$, respectively. Then $\phi_{I}^{-1}(d_{0})$ and $\phi_{I}^{-1}(d_{2})$
are deterministic by the inductive assumption and $I(d_{1})$ is degenerate,
say $\{x\}$ for some $x\in[-1,1]$ (the last claim is because $d_{1}$
is a point in the boundary of $U$). Hence, by lemma \ref{lem:local psi inverse},
\[
\phi_{I}^{-1}(d_{1})=\phi_{I}^{-1}(d_{0})+(\phi_{I}^{-1}(d_{2})-\phi_{I}^{-1}(d_{0}))\left(\frac{1}{2}+\eta q_{f}(d_{1})\tau(d_{1})\right)
\]
which is deterministic, proving the claim.
\end{proof}
\begin{defn*}
For a continuous function $f$ with $||f||_{\infty}\le1$ and a $\delta>0$,
we say that an RH-restrictor $I$ is of type $(f,\delta)$ if there
exists a finite partition $\mathcal{U}=(U_{i})$ of $[0,1]$ into
dyadic intervals and an integer $m\ge-1$ with the following properties.
\begin{enumerate}
\item \label{enu:point below n}$I(d)$ is a single point for every $d$
in the boundary of any $U_{i}$.
\item \label{enu:approx eps}|$\phi_{I,f}^{-1}(U_{i})|\in[\frac{1}{4}\delta,\delta]$
for all $i$.
\item \label{enu:split}If $|\phi_{I,f}^{-1}(U_{i})|>\frac{1}{2}\delta$
then $I(d)$ is a dyadic interval of length $2^{-m}$ for $d$ in
the centre of $U_{i}$.
\item \label{enu:other d}$I(d)=[-1,1]$ for all other $d$.
\end{enumerate}
\end{defn*}
\begin{figure}
\input{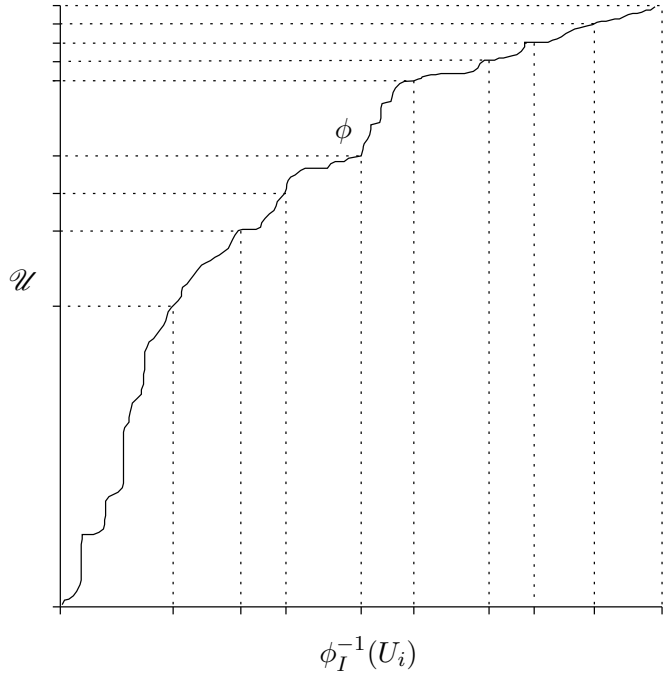}

\caption{An RH restrictor: on the left the dyadic partition, on the bottom
the $\delta$-uniform partition.\label{fig:An-RH-descriptor:}}
\end{figure}
See figure \ref{fig:An-RH-descriptor:}. Of course, not for every
RH-restrictor $I$ one may find some $f$ and $\delta$ such that
$I$ is of type $(f,\delta)$ --- such $I$ are quite special and
if we do not wish to specify $f$ and $\delta$ we will simply say
that `$I$ has a type'. Let us remark on the appearance of $\phi_{I}^{-1}(U_{i})$
in properties \ref{enu:approx eps} and \ref{enu:split} (and indirectly
in \ref{enu:other d} too, as it applies only when \ref{enu:split}
does not). Of course, $\phi_{I}$ is random. But property \ref{enu:point below n}
implies, using lemma \ref{lem:I dont like induction, to be honest about it},
that in fact, $\phi_{I}^{-1}(U_{i})$ is deterministic, and all randomness
is inside the $U_{i}$. Hence properties \ref{enu:approx eps} and
\ref{enu:split} are also a function of $f$ and $I$, and are not
random.

We call the $\mathcal{U}$ and the $m$ the corresponding partition
and the corresponding value (to $I$), respectively.

The following lemma is the main lemma of this paper. It implements
the `reduction of randomness' strategy that we outlined in the beginning
of this section. The reduction in randomness is implemented by moving
from an RH-restrictor for a given $m$, to an RH-restrictor with $m+1$.
Recall the standard definition of the modulus of continuity of a function,
\[
\omega_{f}(\delta)=\sup\{|f(x)-f(y)|:x,y\text{ s.t.}\,\dist(x,y)<\delta\}
\]
where here and below dist is considered cyclically in $[0,1]$ e.g.\ $\dist(0.9,0)=0.1$.
\begin{lem}
\label{lem:main}Let $f$ be a continuous function satisfying $||f||_{\infty}\le\frac{1}{2}$,
let $\eta$ be admissible (recall the definition from page \pageref{def:adimissible eta}),
let $\delta>0$ and let $I$ be an RH-restrictor of type $(f,\delta)$.
Denote by $\mathcal{U}=(U_{i})$ the corresponding dyadic partition
of $[0,1]$, arranged in increasing order, and by $m$ the corresponding
value.

Further, for every $\xi\in[0,1]$ denote by $B_{\xi}$ the union of
the $\phi_{I}^{-1}(U_{i})$ that contains $\xi$ and its two immediate
neighbours. We understand `neighbours' cyclically, e.g.\ if $\xi\in\phi_{I}^{-1}(U_{1})$
then $B_{\xi}$ is the union of $\phi_{I}^{-1}(U_{1})$, $\phi_{I}^{-1}(U_{2})$
and the last $\phi_{I}^{-1}(U_{i})$.

Then there exists an RH-restrictor $J$ of type $(f,\delta)$ with
the same corresponding partition $\mathcal{U}$, with the corresponding
value being $m+1$ and with the following properties
\begin{enumerate}
\item $J(d)\subseteq I(d)$ for all $d$,
\item For every $u\in\mathbb{N}$, every $r\in\{2^{u-1},\dotsc,2^{u}-1\}$
and every $\xi\in2^{-u-2}\mathbb{Z}\cap[0,1)$,
\begin{multline*}
\qquad\qquad\bigg|\int_{E_{\xi}}\big(\mathbb{E}(f(\phi_{I}(x)))-\mathbb{E}(f(\phi_{J}(x)))\big)\cdot D_{r}(x-\xi)\,dx\bigg|\\
\le C\min\{\delta2^{u},1/\delta2^{u}\}^{c}e^{-cm}\omega_{f}(\delta^{c})^{c}.
\end{multline*}
where the set over which integration is performed $E_{\xi}=E_{\xi}(\mathcal{U},I,f,\delta)$
is defined by
\begin{equation}
\qquad\qquad E_{\xi}\coloneqq\begin{cases}
[0,1]\setminus B_{\xi} & 2^{u}>\frac{1}{\delta}\\{}
[0,1] & 2^{u}\le\frac{1}{\delta}.
\end{cases}\label{eq:def Ek}
\end{equation}
\end{enumerate}
\end{lem}
Here and below the constants $C$ and $c$ are allowed to depend on
$\eta$. $D_{r}$ is the usual Dirichlet kernel.

\begin{figure}[b]
\input{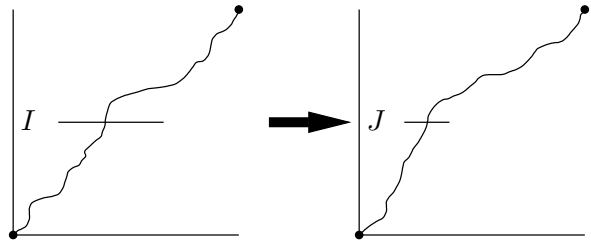}

\caption{Lemma \ref{lem:main} demonstrated for the simplest partition, $\mathcal{U}=\{[0,1]\}$.
The $I$ and $J$ in the picture are $I(\frac{1}{2})$ and $J(\frac{1}{2})$.}
\end{figure}

\begin{proof}
For any $\xi\in[0,1]$ denote $h(\xi)\coloneqq\mathbb{E}(f(\phi_{I}(\xi)))$.
Denote $V_{i}=\phi_{I}^{-1}(U_{i})$ for brevity. Examine one $U_{i}$
in our partition, and recall that $V_{i}$ is not random. The requirements
from $J$ leave relatively little freedom for $J$ inside $U_{i}$.
If $|V_{i}|\le\frac{1}{2}\delta$ then there are no options, $J(d)$
must be $[-1,1]$ for every $d\in U_{i}^{\circ}$. If $|V_{i}|>\frac{1}{2}\delta$
then there are two possibilities. For $d$ in the centre of $U_{i}$
$J(d)$ can be either the left or right half of $I(d)$. For other
$d\in U_{i}^{\circ}$ $J(d)$ must be $[-1,1]$. Thus choosing $J$
is equivalent to choosing a sequence of signs $\varepsilon_{i}\in\{\pm1\}$,
one for each $i$ for which $|V_{i}|>\frac{1}{2}\delta$ (say that
$\varepsilon_{i}=1$ means that we take the right half of $I(d)$
and $\varepsilon_{i}=-1$ the left, and denote for brevity $Y=\{i:|V_{i}|>\frac{1}{2}\delta\}$).
Denote the $J$ that corresponds to a vector $\varepsilon\in\{\pm1\}^{Y}$
by $J^{\varepsilon}$.

Fix one $i\in Y$, one $\xi\in2^{-u-2}\mathbb{Z}\cap[0,1)$ and let
$\varepsilon^{+}$ and $\varepsilon^{-}$ be two vectors with $\varepsilon_{i}^{+}=1$
and $\varepsilon_{i}^{-}=-1$ and define
\begin{gather*}
\Delta_{i}(x)\coloneqq\tfrac{1}{2}(F^{+}(x)-F^{-}(x))\\
F^{\pm}(x)\coloneqq\begin{cases}
\mathbb{E}(f(\phi_{J^{\varepsilon^{\pm}}}(x)))-h(\xi) & x\in V_{i}\\
0 & \text{otherwise.}
\end{cases}
\end{gather*}
(this is a good definition since the values of $\varepsilon$ different
from $\varepsilon_{i}$ have no effect on $\phi|_{V_{i}}$. Hence
they effect neither $F^{\pm}$ nor $\Delta$). In other words, if
we take $J(d)$ to be the left half of $I(d)$, $d$ the centre of
$U_{i}$, then $\mathbb{E}(f\circ\phi_{J})=\mathbb{E}(f\circ\phi_{I})-\Delta$
on $V_{i}$ and otherwise it is $\mathbb{E}(f\circ\phi_{I})+\Delta$.
Summing over $i\in Y$ gives
\begin{equation}
\mathbb{E}(f\circ\phi_{J^{\varepsilon}})=\mathbb{E}(f\circ\phi_{I})+\sum_{i\in Y}\varepsilon_{i}\Delta_{i}.\label{eq:linearises}
\end{equation}
This relation is important because it `linearises' the problem of
choosing a homeomorphism. Our strategy will be to find estimates for
 $\int\Delta_{i}D_{r}(\cdot-\xi)$ for various $i,$ $r$ and $\xi$,
and then apply lemma \ref{lem:renormalisation}. These estimates will
occupy the next three lemmas.  

The reason for subtracting $h(\xi)$ in the definition of $F^{\pm}$
(which, of course, has no affect on $\Delta$) is to get the $\omega(\delta^{c})^{c}$
factor in the statement of lemma \ref{lem:main}. Precisely, we claim
that 
\begin{equation}
||F^{\pm}||_{\infty}\le\omega_{f}(C_{1}(\dist(\xi,V_{i})+\delta)^{4/5})).\label{eq:H justification}
\end{equation}
To see (\ref{eq:H justification}), first note that if $x\in V_{i}$
then $|x-\xi|\le\dist(\xi,V_{i})+\delta$. Let $x'$ be a point in
the boundary of some $\phi_{I}^{-1}(U)$, $U\in\mathcal{U}$, closest
to $x$ among such points, and let $\xi'$ be the point closest to
$\xi$, again from among the points in the boundary of $\phi_{I}^{-1}(U)$,
$U\in\mathcal{U}$. Then the $\delta$-uniformity of $\phi_{I}^{-1}(\mathcal{U})$
says that $|x-x'|\le\frac{1}{2}\delta$ and $|\xi-\xi'|\le\frac{1}{2}\delta$.
We now use the deterministic H\"older condition (\ref{eq:psi-1 45 54})
and get 
\begin{gather*}
|\phi_{J^{\varepsilon^{\pm}}}(x)-\phi_{J^{\varepsilon^{\pm}}}(x')|\le C\delta^{4/5}\qquad|\phi_{I}(\xi)-\phi_{I}(\xi')|\le C\delta^{4/5}\\
|\phi_{J^{\varepsilon^{\pm}}}(x')-\phi_{I}(\xi')|\le C(\dist(\xi,V_{i})+2\delta)^{4/5}
\end{gather*}
where the third inequality follows because $x'$ and $\xi'$ are both
on points where $\phi_{J^{\varepsilon^{\pm}}}=\phi_{I}$. This shows
(still deterministically) that $\phi_{J^{\varepsilon^{\pm}}}(x)-\phi_{I}(\xi)\le C(\dist(\xi,V_{i})+\delta)^{4/5}$
and hence
\begin{equation}
|f(\phi_{J^{\varepsilon^{\pm}}}(x))-f(\phi_{I}(\xi))|\le\omega_{f}(C_{1}(\dist(\xi,V_{i})+\delta)^{4/5})).\label{eq:H before integration}
\end{equation}
Taking expectations shows (\ref{eq:H justification}). It is also
occasionally useful to integrate only the right term, which gives
\begin{equation}
|f(\phi_{J^{\varepsilon^{\pm}}}(x))-h(\xi)|\le\omega_{f}(C_{1}(\dist(\xi,V_{i})+\delta)^{4/5}))\qquad\forall x\in V_{i}.\label{eq:H before integration 2}
\end{equation}
This will be used below.

In light of (\ref{eq:H justification}) let us define 
\begin{equation}
H\coloneqq\omega_{f}(C_{1}(\dist(\xi,V_{i})+\delta)^{4/5}))\label{eq:def H}
\end{equation}
so that $||F^{\pm}||_{\infty}\le H$. Of course $H$ depends on $\xi$,
$i$ and other parameters, but we suppress this in the notation. Both
$h(\xi)$ and $H$ will not play a very important role until the very
end of the proof, at lemma \ref{lem:spectral dyadic}.

Unlike lemmas \ref{lem:j large} and \ref{lem:local estimate} below,
which are only used as sublemmas of lemma \ref{lem:main}, the next
lemma actually has one application after lemma \ref{lem:main} is
finished. For that application, $i$ is not necessarily in $Y$. Hence
keep in mind, when reading this lemma, that $i$ is arbitrary.
\begin{lem}
\label{lem:rs large}With the definitions above, for every $r<s$
and every $\xi$ we have\begin{align}
\bigg|\int_{V_i}F^{\pm}(x)\sum_{l=r}^{s-1}e(l(x-\xi))\,dx\bigg|&\le \min\Big\{\frac{C}{\dist(\xi,V_i)},(s-r)\Big\}\delta H.\nonumber
\intertext{More importantly, }
\bigg|\int_{V_i}F^{\pm}(x)\sum_{l=r}^{s-1}e(l(x-\xi))\,dx\bigg|&\le\left((\delta|r|)^{-\frac{1}{22}}+(\delta|s|)^{-\frac{1}{22}}\right)\frac{C\delta\sqrt{H}}{\dist(\xi,V_i)}.\label{eq:8vakhetzi}
\end{align}
\end{lem}
Recall that dist is considered cyclically.
\begin{proof}
The first clause of the lemma is simple. Indeed, on the one hand
\[
\Big|\sum_{l=r}^{s-1}e(l(x-\xi))\Big|\le s-r
\]
and on the other hand
\[
\Big|\sum_{l=r}^{s-1}e(l(x-\xi))\Big|=\Big|\frac{e(s(x-\xi))-e(r(x-\xi))}{1-e(x-\xi)}\Big|\le\frac{C}{\dist(\xi,x)}
\]
so
\[
\Big|\int_{V_{i}}F^{\pm}(x)\sum_{l=r}^{s-1}e(l(x-\xi))\Big|\le\min\Big\{\frac{C}{\dist(\xi,V_{i})},(s-r)\Big\}\delta H
\]
(recall that $|V_{i}|\le\delta$ and that $||F^{\pm}||_{\infty}\le H$),
establishing the lemma in this case. For the second clause we may
assume that $|r|$ and $|s|$ are both bigger than $1/\delta$, as
in the other case the first clause gives a better estimate. We will
also assume that $\xi\not\in V_{i}$ as otherwise (\ref{eq:8vakhetzi})
is vacuous. Finally, we may assume that $h(\xi)=0$, as subtracting
a constant from $f$ changes neither $\psi_{f,\tau,\infty}$ nor $\phi_{I}$,
and the only loss in generality is that now we may only assume $||f||_{\infty}\le1$
(rather than $||f||_{\infty}\le\frac{1}{2}$, which is one of the
implicit assumptions of the lemma).

We will work in each half of $U_{i}$ separately. Denote therefore
by $U_{i}^{\pm}$ these two halves, and denote $V_{i}^{\pm}\coloneqq\phi_{J^{\varepsilon}}^{-1}(U_{i}^{\pm})$
(no relation to the $\pm$ in $F^{\pm}$, which plays no role here).
For concreteness let us work on $U_{i}^{-}$, the case of $U_{i}^{+}$
is literally identical.  By the locality of $\psi$ (\ref{eq:psi random LI})
$\phi_{J^{\varepsilon}}$ restricted to $V_{i}^{-}$ is a linearly
mapped copy of $\psi$, precisely
\[
\phi_{J^{\varepsilon}}(x)=L_{U_{i}^{-}}(\psi_{f\circ L_{U_{i}^{-}},\tau_{J^{\varepsilon}}\circ L_{U_{i}^{-}},\infty}(L_{V_{i}^{-}}^{-1}(x)))\qquad\forall x\in V_{i}^{-}
\]
Denote $\widetilde{f}\coloneqq f\circ L_{U_{i}^{-}}$. Condition on
$\phi_{J^{\varepsilon}}^{-1}(d)$, where $d$ is the centre of $U_{i}$.
Under the conditioning $\tau_{J^{\varepsilon}}\circ L_{U_{i}^{-}}$
is our standard $\tau$ i.e.\ uniform on $[-1,1]$ for all $d$.
Hence
\begin{align}
X & \!\coloneqq\int_{V_{i}^{-}}\mathbb{E}(f(\phi_{J^{\varepsilon}}(x))\,|\,\phi_{J^{\varepsilon}}(d))\frac{e(r(x-\xi))}{1-e(x-\xi)}\,dx\label{eq:defX}\\
 & \stackrel{\mathclap{\textrm{(\ref{eq:psi random LI})}}}{=}\int_{V_{i}^{-}}\mathbb{E}(\widetilde{f}(\psi_{\widetilde{f},\tau,\infty}(L_{V_{i}^{-}}^{-1}(x))))\frac{e(r(x-\xi))}{1-e(x-\xi)}dx\nonumber \\
 & =\int_{0}^{1}\mathbb{E}(\widetilde{f}(\psi_{\widetilde{f},\tau,\infty}(x)))\frac{e(r(L_{V_{i}^{-}}(x)-\xi))}{1-e(L_{V_{i}^{-}}(x)-\xi)}|V_{i}^{-}|\,dx\nonumber 
\end{align}
where the last equality is a (linear) change of variables. Let $\xi^{*}$
be one of $\xi$, $\xi-1$ or $\xi+1$, whichever is closest to $V_{i}^{-}$,
and note that $\dist(\xi,V_{i})=\inf\{|x-\xi^{*}|:x\in V_{i}\}$.
Let $\widetilde{\xi}=L_{V_{i}^{-}}^{-1}(\xi^{*})$ (note that we are
extending $L_{V_{i}^{-}}^{-1}$ to the whole of $\mathbb{R}$ here
--- originally we defined it only on $V_{i}^{-}$, but it is just
an affine function and we extend it to an affine function on $\mathbb{R}$)
and $\widetilde{r}=r|V_{i}^{-}|$ and with these notations 
\[
X=e(-\widetilde{r}\widetilde{\xi})|V_{i}^{-}|\int_{0}^{1}\mathbb{E}(\widetilde{f}(\psi_{\widetilde{f},\tau,\infty}(x)))\frac{e(\widetilde{r}x)}{1-e(|V_{i}^{-}|(x-\widetilde{\xi}))}\,dx.
\]
Integrate by parts and get
\begin{align*}
X & =e(-\widetilde{r}\widetilde{\xi})|V_{i}^{-}|\bigg(-\int_{0}^{1}\int_{0}^{x}\mathbb{E}(\widetilde{f}(\psi_{\widetilde{f},\tau,\infty}(y)))e(\widetilde{r}y)\,dy\\
 & \quad\cdot\frac{e(|V_{i}^{-}|(x-\widetilde{\xi}))\cdot2\pi\sqrt{-1}|V_{i}^{-}|}{(1-e(|V_{i}^{-}|(x-\widetilde{\xi})))^{2}}\,dx\\
 & +\int_{0}^{1}\mathbb{E}(\widetilde{f}(\psi_{\widetilde{f},\tau,\infty}(y)))e(\widetilde{r}y)\,dy\frac{1}{1-e(|V_{i}^{-}|(1-\widetilde{\xi}))}\bigg).
\end{align*}
We now estimate the integrals over $y$ by lemma \ref{lem:Holder replacement}.
We get
\[
\bigg|\int_{0}^{x}\mathbb{E}(\widetilde{f}(\psi_{\widetilde{f},\tau,\infty}(y)))e(\widetilde{r}y)\,dy\bigg|\le C\widetilde{r}^{-1/11}=C(r|V_{i}^{-}|)^{-1/11}\le C(r\delta)^{-1/11}.
\]
On the other hand we have the trivial bound
\[
\bigg|\int_{0}^{x}\mathbb{E}(\widetilde{f}(\psi_{\widetilde{f},\tau,\infty}(y)))e(\widetilde{r}y)\,dy\bigg|\le||\widetilde{f}||_{\infty}\stackrel{\textrm{(\ref{eq:H before integration 2})}}{\le}H.
\]
Since $\min\{a,b\}\le\sqrt{ab}$ for any two positive numbers $a$
and $b$ we can combine these two estimates to get 
\[
\bigg|\int_{0}^{x}\mathbb{E}(\widetilde{f}(\psi_{\widetilde{f},\tau,\infty}(y)))e(\widetilde{r}y)\,dy\bigg|\le C(r\delta)^{-1/22}\sqrt{H}.
\]
We estimate $|V_{i}^{-}|\le\delta$, $|1-e(\theta)|\ge c\theta$ and
get overall 
\begin{align*}
|X| & \le C\delta\bigg(\int_{0}^{1}(r\delta)^{-1/22}\sqrt{H}\frac{\delta}{\delta^{2}|x-\widetilde{\xi}|^{2}}\,dx+(r\delta)^{-1/22}\sqrt{H}\frac{1}{\delta|1-\widetilde{\xi}|}\bigg)\\
 & \le C\delta(r\delta)^{-1/22}\sqrt{H}\left(\frac{1}{\delta|1-\widetilde{\xi}|}+\frac{1}{\delta|\widetilde{\xi}|}+\frac{1}{\delta|1-\widetilde{\xi}|}\right)
\end{align*}
(note that we used here our assumption $\xi\not\in V_{i}$ which gives
$\widetilde{\xi}\not\in[0,1]$). For every $x\in V_{i}^{-}$, $|x-\xi^{*}|\ge\dist(\xi,V_{i})$
and hence
\[
|x-\widetilde{\xi}|\ge\frac{\dist(\xi,V_{i})}{|V_{i}^{-}|}\ge\frac{\dist(\xi,V_{i})}{\delta}\qquad\forall x\in[0,1].
\]
Hence
\[
|X|\le\frac{C\delta(r\delta)^{-1/22}\sqrt{H}}{\dist(\xi,V_{i})}.
\]
Recalling the definition of $X$ (\ref{eq:defX}) gives
\[
\bigg|\int_{V_{i}^{-}}\mathbb{E}(f(\phi_{J^{\varepsilon}}(x))\,|\,\phi_{J^{\varepsilon}}(d))\cdot\frac{e(r(x-\xi))}{1-e(x-\xi)}\bigg|\le\frac{C\delta(r\delta)^{-1/22}\sqrt{H}}{\dist(\xi,V_{i})}.
\]
This terminates our estimate for $V_{i}^{-}$. The estimate for $V_{i}^{+}$,
as already mentioned, is literally identical, and summing them we
get
\[
\bigg|\int_{V_{i}}\mathbb{E}(f(\phi_{J^{\varepsilon}}(x))\,|\,\phi_{J^{\varepsilon}}(d))\cdot\frac{e(r(x-\xi))}{1-e(x-\xi)}\bigg|\le\frac{C\delta(r\delta)^{-1/22}\sqrt{H}}{\dist(\xi,V_{i})}.
\]
Integrating over $\phi_{J^{\varepsilon}}(d)$ gives
\[
\bigg|\int_{V_{i}}\mathbb{E}(f(\phi_{J^{\varepsilon}}(x)))\cdot\frac{e(r(x-\xi))}{1-e(x-\xi)}\bigg|\le\frac{C\delta(r\delta)^{-1/22}\sqrt{H}}{\dist(\xi,V_{i})}.
\]
Adding these estimates for $r$ and $s$ gives the result.
\end{proof}
\begin{lem}
\label{lem:j large}With the definitions above lemma \ref{lem:rs large},
for every $r<s$ and every $\xi$ we have
\[
\Big|\int\Delta_{i}(x)\sum_{l=r}^{s-1}e(l(x-\xi))\Big|\le C\min\Big\{\frac{\delta}{\dist(\xi,V_{i})},\delta(s-r)\Big\}2^{-m/22}.
\]
\end{lem}
\begin{proof}
Essentially, the lemma follows from lemma \ref{lem:j} using the
same integration by parts that was used to derive the previous lemma
from lemma \ref{lem:Holder replacement}. But let us do it in detail
nonetheless. We start by showing 
\[
\Big|\int\Delta_{i}(x)\sum_{l=r}^{s-1}e(l(x-\xi))\Big|\le\frac{C\delta2^{-m/22}}{\dist(\xi,V_{i})},
\]
which is the more interesting case. Assume $\xi\not\in V_{i}$ as
otherwise the claim is vacuous. Recall that $\Delta_{i}=F^{+}-F^{-}$
and that $F^{\pm}=\mathbb{E}f\circ\phi_{J^{\varepsilon^{\pm}}}-h(\xi)$.
As in the previous lemma we may assume that $h(\xi)=0$ and $||f||_{\infty}\le1$.
We use the locality of $\psi$ (\ref{eq:psi random LI}) to get
\[
\phi_{J^{\varepsilon^{\pm}}}(x)=L_{U_{i}}(\psi_{f\circ L_{U_{i}},\tau_{J^{\varepsilon^{\pm}}}\circ L_{U_{i}},\infty}(L_{V_{i}}^{-1}(x)))\qquad\forall x\in V_{i}.
\]
Denote $\widetilde{f}\coloneqq f\circ L_{U_{i}}$. Note that $\widetilde{\tau}^{\pm}\coloneqq\tau_{J^{\varepsilon^{\pm}}}\circ L_{U_{i}}$
has the following distribution: if $d\ne\frac{1}{2}$ then $\widetilde{\tau}^{\pm}(d)$
is uniform on $[-1,1]$, while $\widetilde{\tau}^{\pm}(\frac{1}{2})$
is uniform on $J^{\varepsilon^{\pm}}(d)$ where $d$ is as usual the
centre of $U_{i}$. Denote
\begin{align*}
X^{\pm} & \coloneqq\int_{V_{i}}F^{\pm}(x)\frac{e(r(x-\xi))}{1-e(x-\xi)}=\int_{V_{i}}\mathbb{E}(\widetilde{f}(\psi_{\widetilde{f},\widetilde{\tau}^{\pm},\infty}(L_{V_{i}}^{-1}(x))))\frac{e(r(x-\xi))}{1-e(x-\xi)}\,dx\\
 & =\int_{0}^{1}\mathbb{E}(\widetilde{f}(\psi_{\widetilde{f},\widetilde{\tau}^{\pm},\infty}(x)))\frac{e(r(L_{V_{i}}(x)-\xi))}{1-e(L_{V_{i}}(x)-\xi)}|V_{i}|\,dx.
\end{align*}
Define $\widetilde{\xi}$ as in the previous lemma with respect to
$V_{i}$ i.e.\ $\widetilde{\xi}=L_{V_{i}}^{-1}(\xi^{*})$ etc., and
$\widetilde{r}=r|V_{i}|$ (also as in the previous lemma) and get
\begin{equation}
X^{\pm}=e(-\widetilde{r}\widetilde{\xi})|V_{i}|\int_{0}^{1}\mathbb{E}(\widetilde{f}(\psi_{\widetilde{f},\widetilde{\tau}^{\pm},\infty}(x)))\frac{e(\widetilde{r}x)}{1-e(|V_{i}|(x-\widetilde{\xi}))}\,dx.\label{eq:locality lemma j}
\end{equation}
Integrate by parts and get
\begin{align*}
X^{\pm} & =e(-\widetilde{r}\widetilde{\xi})|V_{i}|\bigg(-\int_{0}^{1}\int_{0}^{x}\mathbb{E}(\widetilde{f}(\psi_{\widetilde{f},\widetilde{\tau}^{\pm},\infty}(y)))e(\widetilde{r}y)\,dy\\
 & \qquad\qquad\cdot\frac{e(|V_{i}|(x-\xi))\cdot2\pi\sqrt{-1}|V_{i}|}{(1-e(|V_{i}|(x-\widetilde{\xi})))^{2}}\,dx\\
 & \qquad+\int_{0}^{1}\mathbb{E}(\widetilde{f}(\psi_{\widetilde{f},\widetilde{\tau}^{\pm},\infty}(y)))e(\widetilde{r}y)\,dy\frac{1}{1-e(|V_{i}|(1-\widetilde{\xi}))}\bigg).
\end{align*}
This gives us the longest formula in this paper,
\begin{align*}
\lefteqn{\int\Delta_{i}(x)\frac{e(r(x-\xi))}{1-e(x-\xi)}\,dx=\tfrac{1}{2}(X^{+}-X^{-})}\;\\
 & =\frac{e(-\widetilde{r}\widetilde{\xi})|V_{i}|}{2}\bigg(-\int_{0}^{1}\int_{0}^{x}\Big(\mathbb{E}(\widetilde{f}(\psi_{\widetilde{f},\widetilde{\tau}^{+},\infty}(y)))-\mathbb{E}(\widetilde{f}(\psi_{\widetilde{f},\widetilde{\tau}^{-},\infty}(y)))\Big)e(\widetilde{r}y)\,dy\\
 & \qquad\qquad\cdot\frac{e(|V_{i}|(x-\xi))\cdot2\pi\sqrt{-1}|V_{i}|}{(1-e(|V_{i}|(x-\widetilde{\xi})))^{2}}\,dx\\
 & \qquad+\int_{0}^{1}\Big(\mathbb{E}(\widetilde{f}(\psi_{\widetilde{f},\widetilde{\tau}^{+},\infty}(y)))-\mathbb{E}(\widetilde{f}(\psi_{\widetilde{f},\widetilde{\tau}^{-},\infty}(y)))\Big)e(\widetilde{r}y)\,dy\\
 & \qquad\qquad\cdot\frac{1}{1-e(|V_{i}|(1-\widetilde{\xi}))}\bigg).
\end{align*}
We now note that each integral over $y$ above is exactly of the form
given by lemma \ref{lem:j}, with $\varepsilon=2^{-m}$. Hence they
are bounded by $C2^{-m/22}$. Bounding $|V_{i}|\le\delta$ and $|1-e(\theta)|\ge c\theta$
gives
\begin{align*}
\bigg|\int\Delta_{i}(x)\frac{e(r(x-\xi))}{1-e(x-\xi)}\bigg| & \le C\delta\bigg(\int_{0}^{1}2^{-m/22}\frac{\delta}{\delta^{2}|x-\widetilde{\xi}|^{2}}\,dx+2^{-m/22}\frac{1}{\delta|1-\widetilde{\xi}|}\bigg)\\
 & \le C\delta2^{-m/22}\left(\frac{1}{\delta|1-\widetilde{\xi}|}+\frac{1}{\delta|\widetilde{\xi}|}\right).
\end{align*}
Again $|x-\widetilde{\xi}|\ge\dist(\xi,V_{i})/\delta$ for all $x\in[0,1]$,
so
\[
\bigg|\int\Delta_{i}(x)\frac{e(r(x-\xi))}{1-e(x-\xi)}\bigg|\le\frac{C\delta2^{-m/22}}{\dist(\xi,V_{i})}.
\]
Summing the $r$ and $s$ terms gives 
\[
\left|\int\Delta_{i}(x)\sum_{l=r}^{s}e(l(x-\xi))\,dx\right|\le\frac{C\delta2^{-m/22}}{\dist(\xi,V_{i})}.
\]
This is the main estimate of the lemma.

We still need to show the simpler estimate $C\delta(s-r)2^{-m/22}$.
We keep the notations $\widetilde{f}$, $\widetilde{\tau}^{\pm}$
and $\widetilde{\xi}$. For every $l\in\{r,\dots,s-1\}$ we denote
$\widetilde{l}\coloneqq l|V_{i}|$ and then the same locality argument
that gave (\ref{eq:locality lemma j}) gives
\begin{multline*}
\int\Delta_{i}(x)e(l(x-\xi))\,dx\\
=e(-\widetilde{l}\widetilde{\xi})|V_{i}|\int_{0}^{1}\Big(\mathbb{E}(\widetilde{f}(\psi_{\widetilde{f},\widetilde{\tau}^{+},\infty}(x)))-\mathbb{E}(\widetilde{f}(\psi_{\widetilde{f},\widetilde{\tau}^{-},\infty}(x)))\Big)e(\widetilde{l}x)\,dx.
\end{multline*}
We apply lemma \ref{lem:j} directly (without integration by parts)
and get
\[
\left|\int\Delta_{i}(x)e(l(x-\xi))\,dx\right|\le C\delta2^{-m/22}.
\]
We sum over $l$ from $r$ to $s-1$. This gives the second bound,
and proves the lemma. 
\end{proof}
Aggregating the last two lemmas we get
\begin{lem}
\label{lem:local estimate}With the definitions above lemma \ref{lem:rs large},
for every $r<s$ and every $\xi$ we have
\begin{multline*}
\Big|\int\Delta_{i}(x)\sum_{l=r}^{s-1}e(l(x-\xi))\Big|\\
\le C_{1}\delta\min\Big\{\frac{(\min\{\lceil|s|\delta\rceil,\lceil|r|\delta\rceil\})^{-1/44}}{\dist(\xi,V_{i})},(s-r)\Big\}2^{-m/44}H^{1/4}
\end{multline*}
\end{lem}
\begin{proof}
Denote $D=\sum_{l=r}^{s-1}e(l(x-\xi)$. Assume first that both $|r|$
and $|s|$ are larger than $\nicefrac{1}{\delta}$. We use lemma
\ref{lem:rs large} for both $F^{+}$ and $F^{-}$ and summing the
result gives
\[
\bigg|\int\Delta_{i}D\bigg|\le C\delta\frac{(\min\{\lceil|s|\delta\rceil,\lceil|r|\delta\rceil\})^{-1/22}}{\dist(\xi,V_{i})}H^{1/2}
\]
while lemma \ref{lem:j large} gives
\[
\bigg|\int\Delta_{i}D\bigg|\le C\frac{2^{-m/22}}{\dist(\xi,V_{i})}.
\]
For any positive numbers $a$ and $b$, $\min\{a,b\}\le\sqrt{ab}$
so
\[
\min\{2^{-m/22},\lceil|r|\delta\rceil^{-1/22}H^{1/2}\}\le2^{-m/44}\lceil|r|\delta\rceil^{-1/44}H^{1/4}
\]
and similarly for $s$. Applying this to whichever of $r$ and $s$
has a smaller absolute value gives
\[
\left|\int\Delta_{i}D\right|\le C\delta\frac{(\min\{\lceil|s|\delta\rceil,\lceil|r|\delta\rceil\})^{-1/44}}{\dist(\xi,V_{i})}2^{-m/44}H^{1/4}
\]
as needed. The case that either $|r|\le\nicefrac{1}{\delta}$ or
$|s|\le\nicefrac{1}{\delta}$ is similar but simpler. Lemma \ref{lem:rs large}
gives the estimate $|\int\Delta D|\le\delta(s-r)H$, lemma \ref{lem:j large}
gives the estimate $|\int\Delta D|\le C\delta(s-r)2^{-m/22}$, and
we combine them as above. The lemma is thus proved.
\end{proof}
The last step in proving lemma \ref{lem:main} is to choose the $\varepsilon_{i}$.
Recall that for each $i\in Y=\{i:|V_{i}|>\frac{1}{2}\delta\}$ we
need to chose an $\varepsilon_{i}\in\{\pm1\}$. It will be convenient
to add dummy variables, so we will choose $\varepsilon_{i}$ for every
$i$, and ignore those outside $Y$. Denote by $N$ the number of
$U_{i}$ in our partition, and $\Delta_{i}\equiv0$ for every $i\not\in Y$.
Denote 
\[
\Delta(x)=\sum_{i=1}^{N}\varepsilon_{i}\Delta_{i}(x).
\]
The $\varepsilon_{i}$ will be chosen using lemma \ref{lem:renormalisation simplified}.
Here are the details.
\begin{lem}
\label{lem:spectral dyadic}There exists an $\varepsilon=(\varepsilon_{1},\varepsilon_{2},\dotsc,\varepsilon_{N})$
such that for every $u\in\mathbb{N}$, every $r\in\{2^{u-1},\dotsc,2^{u}-1\}$
and for every $\xi\in2^{-u-2}\mathbb{Z}\cap[0,1)$,
\[
\bigg|\int_{E_{\xi}}\Delta(x)D_{r}(x-\xi)\,dx\bigg|\le C\min\{\delta2^{u},1/\delta2^{u}\}^{c}2^{-cm}\omega_{f}(\delta^{c})^{c}.
\]
\end{lem}
\begin{proof}
We need to prepare a sequence of vectors to `feed' into lemma \ref{lem:renormalisation simplified},
which will give us our signs. The estimates of these vectors will
come from lemma \ref{lem:local estimate}. 

We now define our vectors. For every $u\in\mathbb{N}$ let $U$ be
the smallest power of $2$ with $U\ge\max\{2^{u+2},N\}$. For every
$i\in\{1,\dotsc,N\}$ and every $\xi\in\frac{1}{U}\{0,\dotsc,U-1\}$
we define
\[
w_{i,\xi,u}^{0}=\int\Delta_{i}(x)D_{2^{u}}(x-\xi)\,dx
\]
except if $2^{u}>\nicefrac{1}{\delta}$ and $V_{i}\subset B_{\xi}$
($B_{\xi}$ from the statement of lemma \ref{lem:main}), in which
case we define $w_{i,\xi,u}^{0}=0$. Next, for every $u\in\mathbb{N}$,
for every $0\le s<u$ and every $t\in[2^{u-1},2^{u})$ divisible by
$2^{s+1}$, every $i\in\{1,\dotsc,N\}$ and every $\xi\in\frac{1}{U}\{0,\dotsc,U-1\}$
we define
\begin{align*}
w_{i,\xi,u,s,t}^{+} & =\int\Delta_{i}(x)\sum_{z=t+1}^{t+2^{s}}e(z(x-\xi))\,dx,\\
w_{i,\xi,u,s,t}^{-} & =\int\Delta_{i}(x)\sum_{z=-t-2^{s}}^{-t-1}e(z(x-\xi))\,dx,
\end{align*}
again, except if $2^{u}>\nicefrac{1}{\delta}$ and $V_{i}\subset B_{\xi}$,
in which case we define $w_{i,\xi,u,s,t}^{\pm}=0$. This terminates
the list of vectors (we think about them as vectors in $i$) that
we wish to feed into lemma \ref{lem:renormalisation simplified},
after some scaling. 

We note the following convenient fact. For $\xi\in[0,1]$ define $l(\xi)$
by $\xi\in V_{l(\xi)}$. Then
\begin{equation}
\dist(\xi,V_{i})\ge c\delta(|l(\xi)-i\bmod N|+1)\qquad\text{whenever }V_{i}\not\subset B_{\xi}.\label{eq:dist to bmod}
\end{equation}
(We used here that dist is defined cyclically and that $|V_{i}|>\frac{1}{4}\delta$
for all $i$). Similarly,
\begin{equation}
\dist(\xi,V_{i})\le C\delta(|l(\xi)-i\bmod N|+1),\label{eq:dist to bmod-1}
\end{equation}
this time without the restriction $V_{i}\not\subset B_{\xi}$. Another
step that will simplify the analysis below relates to the term $H$
from lemma \ref{lem:local estimate}. Recall that is was defined in
(\ref{eq:def H}) by 
\begin{align*}
H & =\omega_{f}(C(\dist(\xi,V_{i})+\delta)^{4/5})\le\min\Big\{1,C\omega_{f}(\delta^{4/5})\Big(\frac{\dist(\xi,V_{i})+\delta}{\delta}\Big)^{4/5}\Big\}\\
 & \le\min\Big\{1,C\omega_{f}(\delta^{4/5})\Big(\frac{\dist(\xi,V_{i})}{\delta}+1\Big)\Big\}\\
 & \stackrel{\mathclap{\textrm{(\ref{eq:dist to bmod-1})}}}{\le}\min\{1,C\omega_{f}(\delta^{4/5})(|l(\xi)-i\bmod N|+1)\}.
\end{align*}
Denote 
\begin{equation}
\lambda_{l}\coloneqq2^{-m/44}\min\{1,(|l\bmod N|+1)\omega_{f}(\delta^{4/5})\}^{1/4}\label{eq:def lambda_l}
\end{equation}
and get that 
\begin{equation}
2^{-m/44}H^{1/4}\le C\lambda_{l(\xi)-i}.\label{eq:toHixi}
\end{equation}
This will simplify notations in the estimates below. 

We start with $w_{i,\xi,u}^{0}$ for $2^{u}\le\nicefrac{1}{\delta}$.
By lemma \ref{lem:local estimate} and (\ref{eq:toHixi}) 
\begin{align*}
|w_{i,\xi,u}^{0}| & \le C\delta\min\left\{ \frac{1}{\dist(\xi,V_{i})},2^{u+1}\right\} \lambda_{l(\xi)-i}\\
 & \stackrel{(*)}{\le}C\min\left\{ \frac{1}{|l(\xi)-i\bmod N|+1},\delta2^{u+1}\right\} \lambda_{l(\xi)-i}\cdot
\end{align*}
where $(*)$ comes from (\ref{eq:dist to bmod}) and from the observation
that if $V_{i}\subset B_{\xi}$ then the minimum is achieved at $2^{u+1}$
(perhaps up to a multiplicative constant). Define $\widetilde{b}(u,\xi)=\widetilde{b}(u)=\max\{1,\lfloor1/\delta2^{u+1}\rfloor\}$
and get that
\begin{equation}
|w_{i,\xi,u}^{0}|\le\min\left\{ \frac{1}{|l(\xi)-i\bmod n|+1},\frac{1}{\widetilde{b}(u)}\right\} C\lambda_{l(\xi)-i}\label{eq:vtilde}
\end{equation}
We used here $\widetilde{b}$ instead of $b$ because of the factor
$C\lambda_{l(\xi)-i}$. Since it appears in all cases equally, it
will be more convenient to count how many vectors satisfy an inequality
of the form (\ref{eq:vtilde}), for any fixed values of $l$ and $\widetilde{b}$,
and then remove the factor $C\lambda_{l(\xi)-i}$ in the end. To count
the number of $\xi$ with $l(\xi)=l$ for some given $l$, we note
that $\xi\in\frac{1}{U}\mathbb{Z}$ and since $U\le\max\{2^{u+2},2N\}\le\max\{\nicefrac{4}{\delta},2N\}\le\nicefrac{8}{\delta}$
and since $|V_{i}|\le\delta$ for all $i$, we see that each possible
value of $l(\xi)$ repeats at most 9 times. Thus
\[
|\{(\xi,u):2^{u}\le\nicefrac{1}{\delta},l(\xi)=l,\widetilde{b}(u)=\widetilde{b}\}|\le18\qquad\forall l,\widetilde{b}
\]
because we have at most $9$ different $\xi$ which give the same
$l$ and at most 2 different values of $u$ which give the same $\widetilde{b}$
(the largest allowed by the condition $2^{u}\le\nicefrac{1}{\delta}$).

For $2^{u}>\nicefrac{1}{\delta}$ lemma \ref{lem:local estimate},
(\ref{eq:dist to bmod}) and (\ref{eq:toHixi}) give
\[
|w_{i,\xi,u}^{0}|\le C\frac{(\delta2^{u})^{-1/44}\lambda_{i(\xi)-i}}{|l(\xi)-i\bmod N|+1}
\]
so we define $\widetilde{b}(u)=\lfloor(\delta2^{u})^{1/44}\rfloor$.
The multiplicity in $\widetilde{b}$ is at most a constant in this
case, but there is multiplicity $\lfloor\delta U\rfloor+1\le C\delta2^{u}$
in $l$, because $\xi$ is taken in $\frac{1}{U}\mathbb{Z}$. Hence
we get at most $C\delta2^{u}\le C\widetilde{b}(u)^{44}$ vectors satisfying
(\ref{eq:vtilde})

For $w^{\pm}$ and $2^{u}\le\nicefrac{1}{\delta}$ lemma \ref{lem:local estimate},
(\ref{eq:dist to bmod}) and (\ref{eq:toHixi}) give
\[
|w_{i,\xi,u,s,t}^{\pm}|\le C\min\Big\{\frac{1}{|l(\xi)-i\bmod N|+1},\delta2^{s}\Big\}\lambda_{l(\xi)-i}.
\]
Hence we define $\widetilde{b}(\xi,u,s,t)=\widetilde{b}(s)=\lfloor1/\delta2^{s}\rfloor$
and get
\begin{equation}
|w_{i,\xi,u,s,t}^{\pm}|\le\min\Big\{\frac{1}{|l(\xi)-i\bmod N|+1},\frac{1}{\widetilde{b}}\Big\} C\lambda_{l(\xi)-i}.\label{eq:59 vakhetzi}
\end{equation}
The same argument as above gives that each $l(\xi)$ repeats no more
than 9 times, but here we need to count over $u$ and $t$ as well
(with $s$ fixed). The number of possibilities for $t$ given $u$
and $s$ is $2^{u-s-2}$ and $u$ is bounded below by $s+1$ and above
by the requirement $2^{u}\le\nicefrac{1}{\delta}$. Totally we get
\[
|\{(\xi,u,s,t):l(\xi)=l,\widetilde{b}(s)=\widetilde{b}\}|\le9\sum_{\smash{{u\ge s+1}}}^{\smash{2^{u}\le1/\delta}}2^{u-2-s}<\frac{9}{\delta2^{s+1}}\le\frac{9}{2}\widetilde{b}(s)
\]
for all $l$ and $\widetilde{b}$.

Finally, the most complicated case is that of the vectors $w^{\pm}$
for $2^{u}>\nicefrac{1}{\delta}$. Lemma \ref{lem:local estimate},
(\ref{eq:dist to bmod}) and (\ref{eq:toHixi}) give
\begin{equation}
|w_{i,\xi,u,s,t}^{\pm}|\le C\min\Big\{\frac{(\delta2^{u})^{-1/44}}{|i(\xi)-i\bmod N|+1},\delta2^{s}\Big\}\lambda_{i(\xi)-i}.\label{eq:w for u large}
\end{equation}
We define $\widetilde{b}(\xi,u,s,t)=\lfloor\max\{1/\delta2^{s},(\delta2^{u})^{1/44}\}\rfloor$,
so (\ref{eq:59 vakhetzi}) holds. We still need to add up the number
of vectors that correspond to each bound in (\ref{eq:w for u large}).
The number of $\xi$ which give every particular $l$ is bounded by
$\lfloor\delta U\rfloor+1\le C\delta2^{u}$. As for the multiplicity
in $\widetilde{b}$, each value can be achieved either if $\delta2^{s}\in(\frac{1}{\widetilde{b}+1},\frac{1}{\widetilde{b}}]$
and $(\delta2^{u})^{-1/44}\in(\frac{1}{\widetilde{b}+1},\infty)$
or if $\delta2^{s}\in(\frac{1}{\widetilde{b}+1},\infty)$ and $(\delta2^{u})^{-1/44}\in(\frac{1}{\widetilde{b}+1},\frac{1}{\widetilde{b}}]$,
and we need to check, in each case, how many possibilities for $u$,
$s$, $t$ and $\xi$ we have.

In the first case, we get $\delta2^{u}<(\widetilde{b}+1)^{44}$. This
has two implications. First, the number of $u$ which satisfy $1<\delta2^{u}\le(\widetilde{b}+1)^{44}$
can be bounded by $C\log\widetilde{b}$. Second, by the above, each
$l(\xi)$ repeats no more than $C(\widetilde{b}+1)^{44}$ times for
each $t$ and $s$. Further, the restriction $\delta2^{s}\in(\frac{1}{\widetilde{b}+1},\frac{1}{\widetilde{b}}]$
gives that there is at most one possibility for $s$. The number of
possibilities for $t$ given $u$ and $s$ is always $2^{u-s-2}$
and this can be bounded by 
\[
\frac{2^{u}}{2^{s}}\le\frac{(\widetilde{b}+1)^{44}/\delta}{(\widetilde{b}+1)^{-1}/\delta}=(\widetilde{b}+1)^{45}.
\]
So how many $(s,t,u,\xi)$ are there that satisfy $|w_{i,\xi,u,s,t}^{\pm}|\le\min\{1/|i-l\bmod N|+1,\widetilde{b}\}$
for a given $l$ and $\widetilde{b}$? As just explained, there are
$C\log\widetilde{b}$ possibilities for $u$, $C\widetilde{b}^{44}$
possibilities for $\xi$ and $C\widetilde{b}^{45}$ possibilities
for $t$ and $s$ so in total we get no more than $C\widetilde{b}^{89}\log\widetilde{b}\le C\widetilde{b}^{90}$
possibilities.

In the second case ($\delta2^{u})^{-1/44}\in(\frac{1}{\widetilde{b}+1},\frac{1}{\widetilde{b}}]$
so the number of possibilities for $u$ is bounded, and they all satisfy
$\delta2^{u}<(\widetilde{b}+1)^{44}.$ As for $s$, we have that $2^{s}$
ranges between $\frac{1}{\delta(\widetilde{b}+1)}$ and $2^{u}\le\frac{(\widetilde{b}+1)^{44}}{\delta}$.
For each $s$ the number of possible $t$ is $2^{u-s-2}$ so the total
number of possibilities for $(s,t)$ couples is bounded by 
\[
\sum_{s=\lceil\log_{2}1/\delta(\widetilde{b}+1)\rceil}^{\lfloor\log_{2}(\widetilde{b}+1)^{44}/\delta\rfloor}2^{u-s-2}<2^{u-\lceil\log_{2}1/\delta(\widetilde{b}+1)\rceil-1}\le\frac{C(\widetilde{b}+1)^{44}/\delta}{2/(\widetilde{b}+1)\delta}\le C\widetilde{b}^{45}.
\]
Together with the bound on the number of times each $l(\xi)$ repeats
we get a total of $C\widetilde{b}^{89}$ possibilities. This concludes
our bounds. We rearrange all our vectors $w_{i,\xi,u}^{0}$ and $w_{i,\xi,u,s,t}^{\pm}$
into one list $\widetilde{v}_{i,j}$ and define $l(j)=l(u,\xi)$ or
$l(\xi,u,s,t)$, as the case may be, and similarly for $\widetilde{b}$.
The estimates (\ref{eq:vtilde}), (\ref{eq:59 vakhetzi}) become 
\begin{equation}
|\widetilde{v}_{i,j}|\le C_{1}\lambda_{l(j)-i}\min\bigg\{\frac{1}{|l(j)-i\bmod n|+1},\frac{1}{\widetilde{b}(j)}\bigg\}\label{eq:vtilde rearranged}
\end{equation}
and $|\{j:l(j)=l,\widetilde{b}(j)=\widetilde{b}\}|\le C\widetilde{b}^{90}$.\label{pg:vw count end}

It is now time to handle the factors $C_{1}\lambda_{l(j)-i}$. We
have
\begin{align*}
|\widetilde{v}_{i,j}| & \stackrel{\mathclap{\textrm{(\ref{eq:vtilde rearranged})}}}{\le}C_{1}\lambda_{l(j)-i}\min\left\{ \frac{1}{|l(j)-i\bmod n|+1},\frac{1}{\widetilde{b}(j)}\right\} \\
 & \stackrel{\mathclap{\textrm{(\ref{eq:def lambda_l})}}}{\le}C_{1}2^{-m/44}\min\left\{ \frac{1}{|l(j)-i\bmod n|+1},\frac{1}{\widetilde{b}(j)}\right\} \\
 & \qquad\cdot\;\min\{1,(|l(j)-i\bmod n|+1)\omega_{f}(\delta^{4/5})\}^{1/4}\\
 & \le C_{1}2^{-m/44}\min\left\{ \frac{1}{|l(j)-i\bmod n|+1},\min\Big\{\frac{1}{\widetilde{b}(j)},\omega_{f}(\delta^{4/5})^{1/5}\Big\}\right\} .
\end{align*}
Denote $b(j):=\max\{\widetilde{b}(j),\lfloor\omega_{f}(\delta^{4/5})^{-1/5}\rfloor\}$
so that $|\widetilde{v}_{i,j}|\le C_{1}2^{-m/44}/b(j)$. How many
$j$ correspond to a given $l$ and $b$? We can bound the number
simply by $\sum_{\widetilde{b}\le b}C\widetilde{b}^{90}\le Cb^{91}$.
Denote this last constant by $M$, so we have $Mb^{91}$ vectors that
correspond to $b$.

\label{pg:91}Thus we may use lemma \ref{lem:renormalisation simplified}
with the same $M$, with $\gamma=91$ and with the vectors $v_{i,j}=\widetilde{v}_{i,j}2^{m/44}/C_{1}$.
We get a sequence of $\varepsilon_{i}$ such that their inner product
with each vector we fed into the lemma is bounded by $Cb^{-1/50}$.
Estimate 
\begin{align*}
\frac{1}{b} & =\min\left\{ \frac{1}{\widetilde{b}},\frac{1}{\lfloor\omega_{f}(\delta^{4/5})^{-1/5}\rfloor}\right\} \\
 & \le\sqrt{\frac{1}{\widetilde{b}\max\{\lfloor\omega_{f}(\delta^{4/5})^{-1/5}\rfloor,1\}}}\le\frac{C\omega_{f}(\delta^{4/5})^{1/10}}{\widetilde{b}^{1/2}}
\end{align*}
or $b^{-1/50}\le C\widetilde{b}^{-1/100}\cdot\omega_{f}(\delta^{4/5})^{1/500}$.
Inserting the definition of $\widetilde{b}$ from before we get
\begin{align}
\bigg|\sum_{i}\varepsilon_{i}w_{i,\xi,u}^{0}\bigg| & \le C2^{-m/44}\omega_{f}(\delta^{4/5})^{1/500}\begin{cases}
(\delta2^{u})^{1/100} & 2^{u}\le\nicefrac{1}{\delta}\\
(\delta2^{u})^{-1/4400} & 2^{u}>\nicefrac{1}{\delta}
\end{cases}\nonumber \\
\bigg|\sum_{i}\varepsilon_{i}w_{i,\xi,u,s,t}^{\pm}\bigg| & \le C2^{-m/44}\omega_{f}(\delta^{4/5})^{1/500}\min\{(\delta2^{u})^{-1/44},\delta2^{s}\}^{1/100}.\label{eq:sumepsiviwi}
\end{align}
To finish the proof of lemma \ref{lem:spectral dyadic}, assume we
are given some $u\in\mathbb{N}$ and some $r\in\{2^{u-1},\dotsc,2^{u}-1\}$.
Write $r$ in its binary expansion
\[
r=2^{u-1}+\sum2^{s_{k}}\qquad t_{k}\coloneqq2^{u-1}+2^{s_{1}}+\dotsb+2^{s_{k}}
\]
for some decreasing sequence $s_{k}$. We get
\[
D_{r}(x)=D_{2^{u-1}}+\sum_{k}\bigg(\sum_{z=t_{k-1}+1}^{t_{k}}e(zx)+\sum_{z=-t_{k}}^{-t_{k-1}-1}e(zx)\bigg).
\]
Hence for any $\xi\in\frac{1}{U}\{0,\dotsc,U-1\}$ we have
\begin{multline*}
\int_{E_{\xi}}\Delta(x)D_{r}(x-\xi)dx\\
=\sum_{i}\varepsilon_{i}w_{i,\xi,u-1}^{0}+\sum_{k}\sum_{i}\varepsilon_{i}w_{i,\xi,u,s_{k},t_{k-1}}^{+}+\sum_{k}\sum_{i}\varepsilon_{i}w_{i,\xi,u,s_{k},t_{k-1}}^{-}
\end{multline*}
(our assumption that $V_{i}\not\subset B_{\xi}$ whenever $2^{u}>\frac{1}{\delta}$
is satisfied on $E_{\xi}$, recall its definition (\ref{eq:def Ek})).
With (\ref{eq:sumepsiviwi}) we get 
\begin{multline*}
\bigg|\int_{E_{\xi}}\Delta(x)D_{r}(x-\xi)dx\bigg|\\
\le C2^{-m/44}\omega_{f}(\delta^{4/5})^{1/500}\min\{(\delta2^{u})^{1/100},|\log_{2}(\delta2^{u})|(\delta2^{u})^{-1/4400}\}.
\end{multline*}
(the $\log(\delta2^{u})$ in the second estimate comes from the case
that $2^{u}>\nicefrac{1}{\delta}$ in (\ref{eq:sumepsiviwi}), from
the counting on $s$, because as $s$ decreases from $u$ towards
zero there are $\approx\log(\delta2^{u})$ values of $s$ for which
$\min\{(\delta2^{u})^{-1/44},\delta2^{s}\}$ is constant $(\delta2^{u})^{-1/44}$
and only then does it starts to decrease exponentially). Lemma \ref{lem:spectral dyadic}
is thus proved, and due to (\ref{eq:linearises}), so is lemma \ref{lem:main}.
\end{proof}
\phantom\qedhere
\end{proof}
\vspace{-0.8cm}%

\begin{proof}
[Proof of theorem \ref{thm:main}]Recall that the theorem states that
for every continuous function $f$ there is an absolutely continuous
homeomorphism $\psi$ such that $S_{r}(f\circ\psi)$ converges uniformly.
Assume without loss of generality that $||f||_{\infty}\le\tfrac{1}{2}$.
We apply lemma \ref{lem:main} in a doubly infinite process. We start
with the map $I(d)=[-1,1]$ for all $d$, which is of type $(f,1)$
with respect to the trivial partition $\{[0,1]\}$ and to $m=-1$.
We apply lemma \ref{lem:main} and get a map $J_{0}$ which is of
type $(f,1)$ with respect to $\{[0,1]\}$ and to $m=0$ such that
for every $u$, $r$ and $\xi$ as in lemma \ref{lem:main}
\[
\bigg|\int_{E_{\xi}}\Big(\mathbb{E}(f(\phi_{I}(x)))-\mathbb{E}(f(\phi_{J_{0}}(x)))\Big)D_{r}\bigg|\le C2^{-cu}.
\]
We apply lemma \ref{lem:main} to $J_{0}$ and get a $J_{1}$, apply
it again to get $J_{2}$ and so on, and for each $m$ we have for
every $u$, $r$ and $\xi$
\[
\bigg|\int_{E_{\xi}}\Big(\mathbb{E}(f(\phi_{J_{m}}(x)))-\mathbb{E}(f(\phi_{J_{m+1}}(x)))\Big)D_{r}\bigg|\le C2^{-cu-cm}.
\]
The limit of the $J_{m}$ (by which we mean simply the limit for each
dyadic $d$ individually), denote it by $J_{\infty}$, satisfies that
$J_{\infty}(\frac{1}{2})$ is a single point. Summing the errors gives
\[
\bigg|\int_{E_{\xi}}\Big(\mathbb{E}(f(\phi_{I}(x)))-\mathbb{E}(f(\phi_{J_{\infty}}(x)))\Big)D_{r}\bigg|\le C2^{-cu}.
\]
Now because $J_{\infty}(\frac{1}{2})$ is a single point then $J_{\infty}$
is of type $(f,\frac{5}{8})$ with respect to the partition $\{[0,\frac{1}{2}],[\frac{1}{2},1]\}$
and $m=-1$ (the value $\frac{5}{8}$ comes from (\ref{eq:38 58-1})).

Now define $I_{2}=J_{\infty}$ and repeat this procedure. Assume inductively
we have discovered some $I_{k}$ which is of type $(f,\delta_{k})$
with respect to a partition $\mathcal{U}_{k}$ and $m=-1$. We apply
lemma \ref{lem:main} infinitely many times as just explained, get
a sequence of $J_{k,m}$ for $m\in\mathbb{Z}^{+}$ and take the limit
as $m\to\infty$. We get, for every $u$, $\xi$ and $r$,
\[
\bigg|\int_{E_{\xi}}\Big(\mathbb{E}(f(\phi_{I_{k}}(x)))-\mathbb{E}(f(\phi_{J_{k,\infty}}(x)))\Big)D_{r}\bigg|\le C\min\{\delta_{k}2^{u},1/\delta_{k}2^{u}\}^{c}\omega_{f}(\delta_{k}^{c})^{c}.
\]
Now $J_{k,\infty}$ is a single point at the centre of every $U\in\mathcal{U}_{k}$
which satisfies $|\phi_{I_{k}}^{-1}(U)|>\frac{1}{2}\delta_{k}$ and
therefore $J_{k,\infty}$ is of type $(f,\delta_{k+1})$, $\delta_{k+1}\coloneqq\frac{5}{8}\delta_{k}$
with respect to the partition $\mathcal{U}_{k+1}$ that one gets by
splitting each $U\in\mathcal{U}_{k}$ to halves if $|\phi_{I_{k}}^{-1}(U)|>\frac{1}{2}\delta_{k}$
and with respect to $m=-1$. Checking most conditions for a RH-restrictor
of type $(f,\delta_{k+1})$ is straightforward so we check only condition
\ref{enu:approx eps} i.e.\ that $|\phi_{J_{k,\infty}}^{-1}(U)|\in[\frac{1}{4}\delta_{k+1},\delta_{k+1}]$
for all $U\in\mathcal{U}_{k+1}$. There are two cases to check. If
for some $U\in\mathcal{U}_{k}$ we have $|\phi_{I}^{-1}(U)|\le\frac{1}{2}\delta_{k}$
then $U\in\mathcal{U}_{k+1}$, $|\phi_{I_{k}}^{-1}(U)|=|\phi_{J_{k,\infty}}^{-1}(U)|$
and
\[
\frac{1}{4}\delta_{k+1}<\frac{1}{4}\delta_{k}\le|\phi_{I_{k}}^{-1}(U)|\le\frac{1}{2}\delta_{k}<\delta_{k+1}.
\]
The other case is $|\phi_{I_{k}}^{-1}(U)|>\frac{1}{2}\delta_{k}$.
In this case each half of $U$ belongs to $\mathcal{U}_{k+1}$, denote
these halves by $U^{\pm}.$ We use (\ref{eq:38 58-1}) and get
\begin{align*}
|\phi_{J_{k,\infty}}(U^{\pm})| & \le\frac{5}{8}|\phi_{J_{k,\infty}}(U)|=\frac{5}{8}|\phi_{I}(U)|\le\frac{5}{8}\delta_{k}=\delta_{k+1},\\
|\phi_{J_{k,\infty}}(U^{\pm})| & \ge\frac{3}{8}|\phi_{J_{k,\infty}}(U)|=\frac{3}{8}|\phi_{I}(U)|>\frac{3}{8}\cdot\frac{1}{2}\delta_{k}>\frac{5}{32}\delta_{k}=\frac{1}{4}\delta_{k+1}.
\end{align*}
So \ref{enu:approx eps} is proved. This allows to define $I_{k+1}=J_{k,\infty}$
and continue inductively. 

We still have the issue of the integration interval. Indeed, we showed
\begin{multline}
\bigg|\int_{E_{\xi}(\mathcal{U}_{k},I_{k},f,\delta_{k})}\Big(\mathbb{E}(f(\phi_{I_{k}}(x)))-\mathbb{E}(f(\phi_{I_{k+1}}(x)))\Big)D_{r}\bigg|\\
\le C\min\{\delta_{k}2^{u},1/\delta_{k}2^{u}\}^{c}\omega_{f}(\delta_{k}^{c})^{c}\label{eq:same Exi}
\end{multline}
but we wish to compare to the integral over $E_{\xi}(\mathcal{U}_{k+1},I_{k+1},f,\delta_{k+1})$.
Denote for brevity $E_{k}=E_{\xi}(\mathcal{U}_{k},I_{k},f,\delta_{k})$
and $E_{k+1}=E_{\xi}(\mathcal{U}_{k+1},I_{k+1},f,\delta_{k+1})$.
Now, if $2^{u}\le\nicefrac{1}{\delta_{k}}$ then $E_{k+1}=E_{k}=[0,1]$
and there is nothing for us to do. Otherwise $E_{k}\subset E_{k+1}$
and if $\nicefrac{1}{\delta_{k}}<2^{u}\le\nicefrac{1}{\delta_{k+1}}$
then $E_{k+1}\setminus E_{k}$ is a single interval (a union of three
intervals of the type $\phi_{I_{k+1}}^{-1}(U)$, $U\in\mathcal{U}_{k+1}$)
and if $2^{u}>\nicefrac{1}{\delta_{k+1}}$ then it is a union of between
0 and 3 intervals of the type $\phi_{I_{k+1}}^{-1}(U)$, $U\in\mathcal{U}_{k+1}$,
depending on details of the splitting (these intervals might be neighbours
or might not be, this makes no difference to us). In all cases we
apply lemma \ref{lem:rs large} for the RH-restrictor $I_{k+1}$,
for both $F^{+}$ and $F^{-}$, and for all relevant values of $i$.
We get
\[
\frac{1}{2}(F^{+}(x)+F^{-}(x))=\mathbb{E}(f(\phi_{I_{k+1}}(x)))-\mathbb{E}(f(\phi_{I_{k+1}}(\xi))).
\]
(notice that we do not subtract them as we did to get $\Delta$, but
sum them). Thus lemma \ref{lem:rs large} allows to estimate $\int\mathbb{E}(f\circ\phi_{I_{k+1}})$
and we get
\begin{align}
\lefteqn{\bigg|\int_{E_{k+1}\setminus E_{k}}\mathbb{E}(f(\phi_{I_{k+1}}(x))-f(\phi_{I_{k+1}}(\xi)))\cdot D_{r}(x-\xi))\,dx\bigg|}\qquad\qquad\nonumber \\
 & \le\min\bigg\{(2r+1)\delta_{k+1}H,(\delta_{k+1}r)^{-1/22}\sum_{V_{i}\subseteq E_{k+1}\setminus E_{k}}\frac{C\delta_{k+1}\sqrt{H}}{\dist(\xi,V_{i})}\bigg\}\nonumber \\
 & \stackrel{(*)}{\le}C\min\{\delta_{k}2^{u},1/\delta_{k}2^{u}\}^{c}\sqrt{H},\nonumber \\
 & \stackrel{(\dagger)}{\le}C\min\{\delta_{k}2^{u},1/\delta_{k}2^{u}\}^{c}\sqrt{\omega_{f}(\delta_{k}^{c})},\label{eq:change EkEk+1}
\end{align}
where the inequality marked by $(*)$ follows because usually $\dist(\xi,V_{i})\approx\delta_{k+1}$
for $V_{i}\subseteq E_{k+1}\setminus E_{k}$. The only exception is
$\nicefrac{1}{\delta_{k}}<2^{u}\le\nicefrac{1}{\delta_{k+1}}$ where
$\dist(\xi,V_{i})$ might be zero, in which case we use the estimate
$(2r+1)\delta H$ (we changed $\delta_{k+1}$ to $\delta_{k}$ for
shortness, this only affects the constant). The inequality marked
by $(\dagger)$ follows again because $\dist(\xi,V_{i})\le C\delta_{k+1}$
for the relevant intervals. We make one last observation and that
is that
\begin{align}
\lefteqn{{\int_{E_{k}}\big(\mathbb{E}(f(\phi_{I_{k}}(\xi)))-\mathbb{E}(f(\phi_{I_{k+1}}(\xi)))\big)\cdot D_{r}(x-\xi)\,dx}}\qquad\qquad\nonumber \\
 & =\big(\mathbb{E}(f(\phi_{I_{k}}(\xi)))-\mathbb{E}(f(\phi_{I_{k+1}}(\xi)))\big)\Bigg(\int_{E_{k}}D_{r}(x-\xi)\,dx\Bigg)\nonumber \\
 & =\begin{cases}
O(\omega_{f}(\delta_{k}^{c})/\delta_{k}2^{u}) & 2^{u}>\nicefrac{1}{\delta_{k}}\\
\mathbb{E}(f(\phi_{I_{k}}(\xi)))-\mathbb{E}(f(\phi_{I_{k+1}}(\xi))) & 2^{u}\le\nicefrac{1}{\delta_{k}}
\end{cases}\label{eq:silly change xi}
\end{align}
(notice that in the second case this is an equality and not an estimate).
Adding (\ref{eq:same Exi}), (\ref{eq:change EkEk+1}) and (\ref{eq:silly change xi})
gives
\begin{align}
\lefteqn{{\int_{E_{\xi}(\mathcal{U}_{k},I_{k},f,\delta_{k})}\mathbb{E}(f(\phi_{I_{k}}(x))-f(\phi_{I_{k}}(\xi)))\cdot D_{r}(x-\xi)}}\qquad\nonumber \\
 & -\int_{E_{\xi}(\mathcal{U}_{k+1},I_{k+1},f,\delta_{k+1})}\mathbb{E}(f(\phi_{I_{k+1}}(x))-f(\phi_{I_{k+1}}(\xi)))\cdot D_{r}(x-\xi)\nonumber \\
 & =\mathbbm{1}\{2^{u}\le\nicefrac{1}{\delta_{k}}\}(\mathbb{E}(f(\phi_{I_{k}}(\xi)))-\mathbb{E}(f(\phi_{I_{k+1}}(\xi))))\nonumber \\
 & \qquad+O(\min\{\delta_{k}2^{u},1/\delta_{k}2^{u}\}^{c}\omega_{f}(\delta_{k}^{c})^{c}).\label{eq:one term per line ugh}
\end{align}
Fix $u\in\mathbb{N}$, $r\in\{2^{u-1},\dotsc,2^{u}-1\}$ and $\xi\in2^{-u}\{0,\dotsc,2^{u}-1\}$.
As $k\to\infty$ the terms $\{\delta_{k}2^{u},1/\delta_{k}2^{u}\}^{c}$
start small, increase exponentially (recall that $\delta_{k}=\left(\frac{5}{8}\right)^{k}$)
until reaching a constant and then decrease exponentially. Thus if
we summed them (without the term $\omega_{f}(\delta_{k}^{c})^{c}$),
the sum would be constant. Multiplying also by $\omega_{f}(\delta_{k}^{c})^{c}$
we get 
\[
\sum_{k=1}^{\infty}\min\{\delta_{k}2^{u},1/\delta_{k}2^{u}\}^{c}\omega_{f}(\delta_{k}^{c})^{c}\to0
\]
as $u\to\infty$. For the first term on the right hand side of (\ref{eq:one term per line ugh})
we claim that for any $\ell$,
\begin{equation}
\bigg|\sum_{k=1}^{\ell}\mathbbm{1}\{2^{u}\le\nicefrac{1}{\delta_{k}}\}\left(\mathbb{E}(f(\phi_{I_{k}}(\xi)))-\mathbb{E}(f(\phi_{I_{k+1}}(\xi)))\vphantom{\sum}\right)\bigg|\le\omega_{f}(C2^{-(4/5)u}).\label{eq:telescopic}
\end{equation}
Indeed, this is a telescoping sum and hence is equal to either 0 or
to the difference of two averages of $f$ by measures supported in
an interval of length no more than $C2^{-(4/5)u}$ (recall the deterministic
H\"older estimate (\ref{eq:psi-1 45 54})). This shows (\ref{eq:telescopic}).
We conclude that
\[
\bigg|\int_{E_{\xi}(\mathcal{U}_{k},I_{k},f,\delta_{k})}\mathbb{E}\left(f(\phi_{I_{k}}(x))-f(\phi_{I_{k}}(\xi))\text{\ensuremath{\vphantom{\sum}}}\right)D_{r}(x-\xi)\,dx\bigg|\le\epsilon(u)
\]
where $\epsilon(u)$ is some function with $\epsilon(u)\to0$ as $u\to\infty$
($\epsilon$ depends on $f$ but not on anything else, in particular
not on $k$, $r$ or $\xi$). Taking $k\to\infty$ we have that the
$\phi_{I_{k}}$ converge weakly (as measures on $C([0,1])$ with the
supremum norm) to a delta measure on a single homeomorphism $\Phi$,
while $E_{\xi}(\mathcal{U}_{k},I_{k},f,\delta_{k})$ stabilises on
$[0,1]$ as soon as $\delta_{k}<2^{-u}$. Hence weak convergence gives
\[
\bigg|\int_{0}^{1}\big(f(\Phi(x))-f(\Phi(\xi))\big)D_{r}(x-\xi)\,dx\bigg|\le\epsilon(u)
\]
or equivalently,
\begin{equation}
\bigg|S_{r}(f\circ\Phi)(\xi)-f(\Phi(\xi))\bigg|\le\epsilon(u).\label{eq:on a grid}
\end{equation}
The inequality (\ref{eq:on a grid}) implies a uniform bound, namely
\begin{equation}
||S_{r}(f\circ\Phi))-f\circ\Phi)||_{\infty}\to0.\label{eq:everywhere}
\end{equation}
The passage from (\ref{eq:everywhere}) to (\ref{eq:on a grid}) is
a well-known argument, but let us do it in details anyway. We use
Bernstein's inequality for trigonometric polynomials. Recall that
this inequality states that for any trigonometric polynomial $P$
of degree $n$, $||P'||_{\infty}\le2\pi n||P||_{\infty}$. We apply
Bernstein's inequality to $P_{r}\coloneqq S_{r}(f\circ\Phi)-F_{r}(f\circ\Phi)$,
where $F_{r}$ is the Ces\`aro partial sum of $f\circ\Phi$. Since
$F_{r}(f\circ\Phi)\to f\circ\Phi$ uniformly we get that $|P_{r}(\xi)|\le\epsilon_{2}(u)$
for all $\xi\in2^{-u-2}\mathbb{Z}$. Let $x$ be the point where $|P_{r}|$
attains its maximum and let $\xi\in2^{-u-2}\mathbb{Z}$ be such that
$|x-\xi|\le2^{-u-3}$. The mean value theorem now claims that 
\[
|P_{r}(\xi)|\ge|P_{r}(x)|-||P_{r}'||2^{-u-3}\stackrel{(*)}{\ge}|P_{r}(x)|-|P_{r}(x)|\frac{2\pi r}{2^{u+3}}\stackrel{(\dagger)}{\ge}|P_{r}(x)|\Big(1-\frac{\pi}{4}\Big)
\]
where $(*)$ follows from Bernstein's inequality and $(\dagger)$
follows from $r<2^{u}$. Thus $||P_{r}||_{\infty}\le\varepsilon_{2}(u)/(1-\pi/4)$.
Returning to $S_{r}$ (again using that $F_{r}(f\circ\Phi)$ converges
uniformly) shows (\ref{eq:everywhere}).

The homeomorphism $\Phi$ is in fact $\psi_{f,\tau,\infty}$ with
$\tau(d)=\lim_{k\to\infty}I_{k}(d)$ so by lemma \ref{lem:abs cont}
it is absolutely continuous, and if $\eta$ was chosen sufficiently
small also satisfies $\Phi'\in L^{p}$ for any desired $p$, fixed
in advance. The theorem is proved.
\end{proof}

\end{document}